\documentclass{article}

\usepackage[utf8]{inputenc}
\usepackage[english]{babel}
\usepackage{a4wide}
\usepackage{amsmath}
\usepackage{amssymb}
\usepackage{subfig}
\usepackage{hyperref}

\usepackage{graphicx}
\usepackage{epstopdf}
\usepackage{amsfonts}
\usepackage{xcolor}
\usepackage{thm-restate}   
\usepackage{todonotes}
\usepackage{xspace}

\newtheorem{lemma}{Lemma}[section]
\newtheorem{theorem}[lemma]{Theorem}
\newtheorem{remark}[lemma]{Remark}
\newtheorem{proposition}[lemma]{Proposition}
\newtheorem{corollary}[lemma]{Corollary}
\newtheorem{hypothesis}[lemma]{Hypothesis}
\newenvironment{proof}{\par\noindent{\bf Proof. }}{\hfill{$\Box$}\bigskip}

\newcommand{\al}{\operatorname{al}}
\newcommand{\gr}{\operatorname{nu}}

\newcommand{\mytodo}[2]{\todo[size=\tiny, color=#1!50!white]{#2}\xspace}

\newcommand{\mkcom}[1]{\mytodo{orange}{#1}}

\title{On Growth Functions of Ordered Hypergraphs}
\author{Jaroslav Han\v{c}l Jr. 
\and Martin Klazar 
}

\begin{document}

\maketitle

\begin{abstract}
For $k,l\ge2$ we consider ideals of edge $l$-colored complete $k$-uniform hypergraphs $(n,\chi)$ with vertex sets $[n]=\{1, 2, \dots n\}$ for $n\in\mathbb{N}$. An ideal is a set of such colored hypergraphs that is closed to the relation of induced ordered subhypergraph. We obtain analogues of two results of Klazar \cite{Klazar08} who considered graphs, namely we prove two dichotomies for growth functions of such ideals of colored hypergraphs. The first dichotomy is for any $k,l\ge2$ and says that the growth function is either eventually constant or at least $n-k+2$. The second dichotomy is only for $k=3,l=2$ and says that the growth function of an ideal of edge two-colored complete $3$-uniform hypergraphs grows either at most polynomially, or for $n\ge23$ 
at least as $G_n$ where $G_n$ is the sequence defined by $G_1=G_2=1$, $G_3=2$ and $G_n = G_{n-1} + G_{n-3}$ for $n\ge4$. The lower bounds in both dichotomies are tight.
\\


\end{abstract}
\section{Introduction}

Let $k,l \geq 2$ be positive integers and $\mathcal{C}_k$ be the set of all edge $l$-colored complete $k$-uniform hypergraphs $K=(n,\chi)$. 
Here $n\in\mathbb{N}=\{1,2,\dots\}$ and $\chi:\;\binom{[n]}{k}\to[l]$ gives to each $k$-element subsets of $[n]=\{1,2,\dots,n\}$ one of the $l$ colors. 
We call the elements of $\mathcal{C}_k$ \emph{colorings} and order $\mathcal{C}_k$ by the relation $\preceq$ of an induced ordered subhypergraph: $(m,\phi)\preceq(n,\chi)$ 
if and only if there is an increasing injection $f:\;[m]\to[n]$ 
such that for every set $E\subset[m]$ with $|E|=k$ we have $\phi(E)=\chi(f(E))$. For $X\subset\mathcal{C}_k$ and $n\in\mathbb{N}$ we denote by $X_n$ the subset of
$X$ consisting of the colorings with the vertex set $[n]$. We call the map $n\mapsto|X_n|$, where $|X_n|$ is the cardinality of $X_n$, the {\em growth function of $X$}. 
We call $X\subset\mathcal{C}_k$ an {\em ideal} if it is downward closed to $\preceq$: $K\preceq L\in X$ implies $K\in X$. Sets of structures 
closed to a containment relation are also called hereditary or monotone properties of graphs, permutation classes, or downsets of posets. 

In our paper we study growth functions of ideals in $\mathcal{C}_k$. Scheinerman and Zito \cite{SZ94} were the first to investigate full spectrum of growth of hereditary 
graph properties. Their results were extended and made more precise by Balogh, Bollob\'{a}s and Weinreich \cite{BBW00,BBW06}, Alekseev \cite{A93}, Bollob\'{a}s 
and Thomason \cite{BT95,BT97}, Pr\"{o}mel and Steger \cite{PS96} and others. Growth functions of ideals were studied also for other structures: 
for oriented graphs \cite{AS00}, posets \cite{BGP99,BBM06}, words \cite{QZ04,BB05} and permutations \cite{KK03,MT04}. In this article we obtain analogues for
ordered $k$-uniform hypergraphs of the results in Klazar \cite{Klazar08} who investigated ordered graphs. 

Our first main result concerns ideals $X\subset\mathcal{C}_k$ for any $k,l\ge2$ and asserts that the growth function of every $X$ is either eventually constant 
or at least linear. 

\begin{theorem}\label{T1CD}
If $k,l\ge2$ and $X\subset\mathcal{C}_k$ is an ideal of colorings then either $|X_n|=c$ for every $n>n_0$ or $|X_n|\geq n-k+2 $ for every $n\in\mathbb{N}$.
\end{theorem}

\noindent
The lower bound is tight, it is attained for $l=2$ by the ideal of colorings $(n,\chi)\in\mathcal{C}_k$ such that $\chi$ is either constantly $1$ or $\chi(E)=1$ for all 
$E\in\binom{[n]}{k}\setminus\{I\}$ but $\chi(I)=2$ for a $k$-element interval $I$ in $[n]$.

Let the sequence
\begin{equation*}
(G_n)_{n \geq 1} = (1,\,1,\,2,\,3,\,4,\,6,\,9,\,13,\,19,\,28,\,41,\,\dots)
\end{equation*}
be defined by the recurrence $G_1=G_2=1$, $G_3=2$ and $G_n=G_{n-1}+G_{n-3}$ for $n\ge4$. In our second main result we set $k=3$, $l=2$ and prove that the growth function 
of every ideal in $\mathcal{C}_3$ is either at most polynomial or at least as fast as $G_n$.

\begin{restatable}{thmres}{Fibonaccidichotomy} \label{T1FD}
Let $k=3$, $l=2$ and $X\subset\mathcal{C}_3$ be an ideal of colorings. Then there is a constant $c>0$ such that either $|X_n|\le n^c$ for every $n\in\mathbb{N}$,
or $|X_n|\geq G_{n}$ for every $n \geq 23$.
\end{restatable}

\noindent
The lower bound is again tight, it is attained by the ideal of colorings $(n,\chi)\in\mathcal{C}_3$ such that $\chi(E)=2$ iff $E=I_j$ where
$I_1<I_2<\dots<I_r$ is a (possibly empty) family of disjoint $3$-element intervals in $[n]$. Indeed, the number of such colorings
is the same as the number of $s$-tuples $(x_1,x_2,\dots,x_s)\in\{1,3\}^s$ such that $x_1+x_2+\dots+x_s=n$. We made some effort to optimize the constant $23$ but with 
more effort it could be probably further lowered.

Our article has the following sections. Section 2 is devoted to the constant dichotomy in Theorem~\ref{T1CD}. The quasi-Fibonacci dichotomy (we explain this terminology 
at the beginning of Section 3) in Theorem~\ref{T1FD} is proven 
in Sections 3--5. In Section 3 we introduce various colorings of triples, which we call \emph{wealthy} colorings of \emph{types $W_1$--$W_4$}, and prove that presence 
of these colorings in an ideal makes it grow at least as fast as $G_n$. Section 4 concerns a technical tool of so called \emph{crossing matrices},
which are three-dimensional matrices with entries in $\{0,1,*\}$ that are associated to colorings. The proof of Theorem~\ref{T1FD} is completed in Section 5. 
Section 6 contains concluding remarks and two conjectures.

We recall and fix some notation. $\mathbb{N}=\{1,2,3,\dots\}$ are the natural numbers, $\mathbb{N}_0=\mathbb{N}\cup\{0\}$, 
$[n]=\{1,2,\dots, n\}$ and $[n]_{0}=\{0,1,\dots, n\}$ with $n \in \mathbb{N}_0$ and $[0]=\emptyset$. For $a,b\in\mathbb{N}$ with $a\le b$ we write $[a,b]\subset\mathbb{N}$ 
for the set $\{a,a+1,\dots,b\}$. $|X|$ is the cardinality of a finite set $X$. For $k\in\mathbb{N}$ and $A \subset [n]$ we denote by $\binom{A}{k}$ the set of 
$k$-element subsets of $A$. For a set $A$ of integers, $\min(A)$ and $\max(A)$ denotes the minimum and maximum element of $A$, respectively. For any two sets 
$A, B \subset\mathbb{N}$ we write $A < B$ if $\max A < \min B$, equivalently if $a<b$ for every $a\in A$ and $b\in B$. We write $A = \{a_1<a_2<\dots <a_k\}$ for the set $A = \{a_1, a_2, \dots, a_k\}\subset\mathbb{N}$ such that $a_1<a_2<\dots <a_k$. In the next section the numbers $k,l\ge2$ denote, if it is not said else, the cardinality of edges in colorings in $\mathcal{C}_k$ and the number of colors, respectively.

\section{Proof of the constant dichotomy in Theorem~\ref{T1CD}}

In this section we prove Theorem~\ref{T1CD}. Recall that $\mathcal{C}_k$ is the set of all edge $l$-colored ordered complete $k$-uniform hypergraphs $(n,\chi)$ 
where $n\in\mathbb{N}$ and $\chi:\;\binom{[n]}{k}\rightarrow [l]$. The elements of $\mathcal{C}_k$ are called colorings and for 
$K=(n,\chi)\in \mathcal{C}_k$, $n$ is the {\it size of $K$}. On $\mathcal {C}_k$ we have the relation $\preceq$ of induced ordered subgraph, defined in Section 1. 
The {\it reversal} of a coloring $K=(n,\chi)$ is the coloring $(n,\phi)$ where for every $E\subset[n]$ with $|E|=k$,
$$
\phi(E)=\chi(\{n-x+1:\ x\in E\})\,.
$$ 
For $X\subset{\cal C}_k$, $X_n$ is the set of colorings in $X$ with the vertex set $[n]$. For $(n,\chi)\in{\cal C}_k$ and $A\subset[n]$, $A$ is  
{\it $\chi$-homogeneous} or {\it $\chi$-monochromatic} if $\chi(E)$ is constant for any $E \in \binom{A}{k}$; we omit `$\chi$-' if it is clear from the context. 
We write $\chi(\binom{A}{k})=c$ if $\chi$ colors every $k$-set of $A$ with $c$. Let $C=(n,\chi)$ be a coloring and $X$ be a subset of $[n]$. We define a new coloring $D=([|X|],\chi')$ by restricting $\chi$ to $k$-subsets of $X$ and relabeling the elements of $X$ in increasing order as $1,2,\dots,|X|$. Of course, $D\preceq C$. 
We say that $D$ arises by {\em restriction and normalization of $C$ to $X$.}

To prove Theorem \ref{T1CD} we define special colorings producing linearly many subcolorings of a given size, and then we characterize simple colorings containing only a few subcolorings. The crucial Proposition~\ref{P0} says that an ideal either contains all special colorings or consists of only simple colorings. 

Let $f, h\in\mathbb{N}_{0}$ and $r,g\in\mathbb{N}$ satisfy $k=f+g+h$ and $r\geq k$. We set 
$$
    E_i=[f]\cup[f+i,f+g+i-1]\cup[n-h+1,n],\ i=1,2,\dots,r-k+2\,.
$$ 
A coloring $(n,\chi)$ is {\it $r$-rich of type $T_{f,g,h}$} if $n=2r-k+1$ and for two colors $a\ne b$ one has $\chi(E_i)=a$ for $i=1,2,\dots,r-k+1$ and $\chi(E_{r-k+2})=b$. Note that every $E_i$ is a $k$-set because $r\geq k$ and hence $n-h+1>r-h+1\ge f+g+i-1$. A coloring $(n,\chi)$ is {\em $r$-rich} if it is $r$-rich of type $T_{f,g,h}$ for some $f,h\in\mathbb{N}_0$ and $g\in\mathbb{N}$. 
Colors of other edges in $(n,\chi)$ are not restricted.

\begin{lemma}\label{L-RICH}
If an ideal $X\subset{\cal C}_k$ contains an $r$-rich coloring for every $r\geq k$ then $|X_n|\geq n-k+2$ for every $n\geq k$.
\end{lemma}

\begin{proof}Let $X\subset{\cal C}_k$ be an ideal, $r\ge k$, $n=2r-k+1$, and $K=(n,\chi)\in X$ be an $r$-rich coloring of type $T_{f,g,h}$. We consider the $r-k+2$ colorings $C_j=(r,\chi_j)\in X_r$, $j \in [r-k+1]_0$, obtained from $K$ by deleting $r-k+1$ numbers from $[n]$, $j$ of them immediately after $[f]$ and $r-k+1-j$ of them immediately before $[n-h+1,n]$, and normalising the remaining elements of $K$. We set $L_i=[f]\cup[f+i+1, f+i+g]\cup[r-h+1,r]$ for $i\in\mathbb{N}_0$ (the elements of $L_i$ are after normalization). Therefore
\begin{align*}
    (\chi_0(L_i),\, i \in [r-k]_0)   &=(a,\,a,\,\dots,\,a,\,a,\,a)\,, \\
    (\chi_1(L_i),\, i \in [r-k]_0) &=(a,\,a,\,\dots,\,a,\,a,\,b)\,, \\
    (\chi_2(L_i),\, i \in [r-k-1]_0) &=(a,\,a,\,\dots,\,a,\,b)\,, \\
    &\vdots \\
    (\chi_{r-k+1}(L_i),\, i \in [0]_0) &=(b)
\end{align*}
by the definition of $K$ and $C_j$. The colorings $C_j$, $j \in [r-k+1]_0$, are pairwise distinct and hence $|X_r|\geq r-k+2$. 
\end{proof}

Let $c$ with $c\geq k$ be a positive integer. A coloring $(n,\chi)$ is {\it $c$-simple} if the following holds.
\begin{itemize}
\item[C1.] The set $[c+1, n-c]$ is $\chi$-homogeneous.
\item[C2.] For every $k-1$ distinct vertices $v_1,v_2,\dots,v_{k-1}\in [n]$ where $v_1\in[c]\cup[n-c+1,n]$, the $k$-sets $\{v_1,\dots,v_{k-1},w\}$ with $w\in[2c+1,n-2c]$, $w\ne v_1,\dots,v_{k-1}$, have in $\chi$ the same color.
\end{itemize}
Any coloring with $n\leq\min(2c+k,4c+1)=2c+k$ is trivially $c$-simple (the set in C1 has at most $k$ elements or the set for $w$ in C2 is empty). An ideal $X$ is {\it $c$-simple} if all its colorings are $c$-simple.

\begin{lemma}\label{L-SIMPLE}
If an ideal $X$ is $c$-simple then $|X_n|$ is constant for every $n>5c$.
\end{lemma}
\begin{proof}Let $X$ be a $c$-simple ideal. 
For a coloring $C=(n,\chi)\in X$ with $n\ge5c+1$ we consider the coloring $C'=(n-1,\chi')$ obtained by restricting and normalizing $C$ to $[n]\setminus\{2c+1\}$. So $C'\preceq C$ and $C'\in X$. We show that for $n>5c$ the correspondence $C\mapsto C'$ is injective. Hence for $n>5c$ the numbers $|X_n|$ weakly decrease and the claim follows.

Let $n>5c$ and $C=(n,\chi)$ and $D=(n,\psi)$ be distinct colorings from $X$. So $\chi(E)\ne\psi(E)$ for some edge $E\in\binom{[n]}{k}$. If $2c+1\notin E$ then $E$ survives in both $C'$ and $D'$ hence $C'\ne D'$. Let $2c+1\in E$. We have either $E\subset[c+1,n-c]$ or there is an $u\in E\cap([c]\cup[n-c+1,n])$. In the former case, since $n>5c$ implies $|[c+1,n-c]|>3c>k$, we find an $x\in[c+1,n-c]\setminus E$ and set $F=(E\setminus\{2c+1\})\cup\{x\}$. Similarly, in the latter case, since $|[2c+1,n-2c]|>c\geq k$, there is an $x\in[2c+1,n-2c]\setminus E$ and we set $F=(E\setminus\{2c+1\})\cup\{x\}$. By $c$-simplicity of $X$, in either case $\chi(F)=\chi(E)\ne\psi(E)=\psi(F)$. Since $F$ survives in both $C'$ and $D'$, again $C'\ne D'$.
\end{proof}

\begin{lemma}\label{L3.5.}
Let $r\geq k$ be an integer, $(n,\chi)$ be a coloring, and $A\subset[n]$ be a $\chi$-homogeneous set of the maximum cardinality. Suppose that $A'\subset A$ arises by deleting $k(r-k+1)$ elements both from the beginning and the end of $A$. Suppose that $A'\ne\emptyset$ and that $A'$ is not an interval in $[n]$. Then $(n,\chi)$ contains an $r$-rich coloring.
\end{lemma}

\begin{proof}
By the assumption on $A'$ we may assume that $|A|\geq 2k(r-k+1)+2$. We have $2k+1$ pairwise disjoint sets 
$$
    A=B_1\cup\dots\cup B_k\cup A'\cup C_1\cup\dots\cup C_k
$$ 
where $B_1<\dots<B_k<A'<C_1<\dots<C_k$ and $|B_i|=|C_i|=r-k+1$. Since $A'$ is not an interval, an element $e\in[n]\setminus A$ exists with $B_k<e<C_1$. Since $|A|$ is maximum with respect to the monochromaticity, there is a $k$-set $E$ with $e\in E$, $E\setminus\{e\}\subset A$, and $\chi(E)\ne\chi(\binom{A}{k})$. We select $E$ that has the minimum number of elements greater than $e$. It follows from $|E\cap A|=k-1$ that there are two indices $i_0,j_0\in[k]$ with $E\cap B_{i_0}=E\cap C_{j_0}=\emptyset$. We define $E^-=E\cap(\cup_{i=1}^{i_0-1}B_i)$, $E^+=E\cap(\cup_{j=j_0+1}^{k} C_j)$, and $E_0=E\setminus(E^-\cup E^+)$. Clearly $E_0\ne\emptyset$ as $e\in E_0$. Consider the $k$-set $E'$ obtained from $E$ by exchanging the elements $\max B_{i_0}$ and $\max E_0$. It follows that $\chi(E')=\chi(\binom{A}{k})$ --- either by the minimality property of $E$ if $\max E_0>e$, or by $E'\subset A$ if $\max E_0=e$. Repeatedly shifting the middle part $E_0$ of $E$ to the left in $B_{i_0}$ (in the second step we exchange $\max(E'\setminus(E^-\cup E^+))$ and the second largest element of $B_{i_0}$, and so on) we obtain $r-k+1$ sets with size $k$ and the same color $\chi(\binom{A}{k})$. We define $D=B_{i_0}\cup E\cup C'$, where $C'\subset C_{j_0}$ is any subset with size $|C'|=r-k$ that completes $D$ to the right cardinality $|D|=2r-k+1$. Then $(|D|,\chi')$, obtained by restriction and normalization of $(n, \chi)$ to $D$, is an $r$-rich coloring of type $T_{|E^-|,|E_0|,|E^+|}$ and is contained in $(n,\chi)$. 
\end{proof}

\begin{lemma} \label{L3.6.}
Let $r$ with $r\geq k$ be an integer, $(n,\chi)$ be a coloring, and $s$ be the maximum size of a  $\chi$-homogeneous subset of $[n]$. Let $A\subset[n]$ be $\chi$-homogeneous with size $|A|=s-2k(r-k+1)$ and $B\subset[n]$ be  $\chi$-homogeneous with $A<B$ or $B<A$. If $|A|\geq k(r-k+1)$ and $|B|\geq (2k+2)r$ then $(n,\chi)$ contains an $r$-rich coloring.
\end{lemma}

\begin{proof}
Let $A$ and $B$ be as stated, with $|A|\geq k(r-k+1)$ and $|B|\geq (2k+2)r$. Note that then both $|A|,|B|\ge r$ and $|A|,|B|\geq k(r-k+1)$.  
We assume that $A<B$, the case $A>B$ is treated by passing to the reversals. The first case is when $\chi(\binom{A}{k})=a\ne b=\chi(\binom{B}{k})$. We take the last $k$ vertices of $A$ and the first $k$ vertices of $B$, $a_k<\dots<a_1<b_1<\dots<b_k$, and consider the colors $\chi(F_i)=\chi(\{a_{k-i},\dots,a_1,b_1,\dots,b_i\})$ for $0\le i\le k$. Clearly, $\chi(F_0)=a$ and $\chi(F_k)=b$. Let $t \ge 1$ be the first index with $\chi(F_t)\ne a$ and let $D$ consist of the last $r-t+1$ vertices of $A$ and the first $r-k+t$ vertices of $B$ ($|A|,|B|\ge r$). Then $|D| = 2r-k+1$ and the first $r-k+1$ intervals in $D$ of size $k$ have color $a$ but the next one has a different color. Thus $(|D|,\chi')$, obtained by restriction and normalization of $(n, \chi)$ to $D$, is an $r$-rich coloring of type $T_{0,k,0}$ and is contained in $(n,\chi)$.

The second case is when $\chi(\binom{A}{k})=\chi(\binom{B}{k})=a$. Consider the coloring $L=(|A\cup B|,\psi)$ obtained by restricting and normalizing $(n,\chi)$ to $A\cup B$. Let $A'=[|A|]$ and $B'=[|A|+1,|A\cup B|]$ be the counterparts of $A$ and $B$ in the domain of $L$. Since $L\preceq (n,\chi)$, it suffices to find an $r$-rich coloring in $L$. We split the first $k(r-k+1)$ vertices of $A'$ and the last $k(r-k+1)$ vertices of $B'$ in tuples $A_i$ and $B_i$, respectively, satisfying
$$
    A_1<A_2<\dots<A_k<B_1<B_2<\dots<B_k,\, |A_i|=|B_i|=r-k+1\, .
$$
Since 
$$
    |A'|+|B'|-(r-k+1)\geq s-2k(r-k+1)+(2k+2)r-(r-k+1)=s-(2k+1)r+(2k+2)r>s\,,
$$
there is an edge $F\subset A'\cup B'$, $F=\{f_1<\dots<f_k\}$, such that $f_k<B_k$ and $\psi(F)\ne a$. Note that $f_k\in B'$. Among all such edges $F$ we take one with the minimum last element $f_k$. Since $|F\cap A'|\leq k-1$, we may take an index $i_0\in[k]$ such that $F\cap A_{i_0}=\emptyset$. We set $D=A_{i_0}\cup F\cup B_k^{-}$ and $t=|F\cap(\cup_{l=1}^{i_0-1} A_l)|$, where $B_k^{-}$ is an arbitrary subset of $B_k$ such that $|B_k^{-}|=r-k$. Shifting the part of $F$ after $A_{i_0}$ to the left in $A_{i_0}$, like in the proof of Lemma \ref{L3.5.}, and using minimality of the last element $f_k$, we get an $r$-rich coloring of the type $T_{t,k-t,0}$.
\end{proof}

By $R_m(a,l)$ we denote the Ramsey number for $m$-tuples and $l$ colors, $R_m(a,l)$ is the smallest $n\in\mathbb{N}$ such that every $l$-coloring of $\binom{[n]}{m}$ has a homogenous set $A\subset[n]$ with size $|A|=a$. 

\begin{lemma}\label{L3.7.}
Let $r\geq k$, $R = \max \{ R_{i}(r-1, l), i=1,2,\dots,k-1
\}$ and $(n,\chi)$ be a coloring. Let $A\subset[n]$ be a set with $|A|\geq 2(k-1)R$, $v_1,\dots,v_{k-1}\in[n]$ be distinct vertices such that $v_1<A$ or $v_1>A$, and let $A'\subset A$ arise by deleting both the first and last $(k-1)R$ elements of $A$. Suppose that not all edges $E_w=\{v_1,\dots,v_{k-1},w\}$, where $w\in A'\backslash\{v_1,\dots,v_{k-1}\}$, have the same color. Then $(n,\chi)$ contains an $r$-rich coloring.
\end{lemma}

\begin{proof}
Let $v_1<A$, the case $v_1>A$ is symmetric. We relabel the vertices so that $v_1<\dots<v_{k-1}$. By the assumption we have two (distinct) vertices $w_1,w_2\in A'$ such that $\chi(E_{w_1})=a_1\ne a_2=\chi(E_{w_2})$. Without loss of generality $w_1<w_2$. We divide $A$ into $2k-1$ disjoint sets
$$
    B_1<B_2<\dots< B_{k-1}<A'<C_1<C_2<\dots<C_{k-1},\, |B_i|=|C_i|=R\, .
$$ 
Clearly, $v_1<A$ implies that there are indices $i_0,j_0\in[k-1]$ such that $\{v_1,v_2,\dots,v_{k-1}\}$ is disjoint to both $B_{i_0}$ and $C_{j_0}$. We set $v_k=n+1$ and define the indices $p,q\in[k-1]$ by $v_p<B_{i_0}<v_{p+1}$ and $v_{q}<C_{j_0}<v_{q+1}$. Clearly,  $p\leq q$. We set $F=\{v_1,\dots,v_p\}\cup\{v_{q+1},\dots,v_{k-1}\}$, $s=q-p+1$, and $b=s+r-k$. Clearly, $s\in[k-1]$, $b\in[r-k+1,r-1]$, and $|F|+s=k$. From $|B_{i_0}| = R \geq R_{s}(b,l)$ we conclude that there is a set $X\in\binom{B_{i_0}}{b}$ such that $\chi(F\cup G)=c$ for some color $c\in[l]$ and all $G\in\binom{X}{s}$. Let $X=\{x_1<x_2<\dots<x_b\}$. We set $a=a_1 = \chi(E_{w_1})$ and $E=E_{w_1}\setminus F$ if $c\ne a_1$, and $a=a_2 = \chi(E_{w_2})$ and $E=E_{w_2}\setminus F$ if $c=a_1$. We denote
$$
    E =\{e_1<e_2<\dots<e_s\}=\{v_{p+1},\, \dots,\, v_q,\, w_i\}\quad (\mbox{$i=1$ or $i=2$})
$$
and for $t\in[s]$ consider the colors 
$$
    c_t=\chi(F\cup\{x_{b-s+t+1}<\dots<x_b<e_1<\dots<e_t\})\,.
$$
Since $c_s=a\ne c$, we may take the minimum index $I\in\{1,\dots,s\}$ such that $c_I\ne c$. We set
$$
    D=F\cup\{x_I,\, \dots,\, x_b\}\cup E\cup\ C_{j_0}^{-}\, ,
$$
where $C_{j_0}^{-}$ is the set of the first $b-2s+I$ elements of $C_{j_0}$. Then by restricting and normalizing of $(n,\chi)$ to $D$ we get an $r$-rich coloring of type $T_{p,s,k-q-1}$ that is contained in $(n,\chi)$ (the middle part of size $s$ moving to the left in the set $X$ starts as $\{x_{b-s+I+1}<\dots<x_b<e_1<\dots<e_I\}$).
\end{proof}

\begin{proposition}\label{P0}
For every $r\geq k$ there is a constant $c=c(r)\in\mathbb{N}$ such that every ideal $X$ of colorings either contains an $r$-rich coloring or is $c$-simple.
\end{proposition}

\begin{proof}
We assume that $r\geq k$ and that $X$ is an ideal of colorings not containing an $r$-rich coloring. We set $R = \max \{ R_{i}(r-1, l), i=1,2,\dots,k-1\}$ (as in Lemma \ref{L3.7.}),
$$
    d=\max\{(2k+2)r,\, 3k(r-k+1)\} \quad \text{and} \quad c=\max\{(k-1)R,\, R_k(d,l)\}\, .
$$
We prove that $X$ is $c$-simple. Let $(n,\chi)\in X$ be arbitrary. We may suppose $n>2c+k$ since smaller colorings are trivially $c$-simple. We take a $\chi$-homogenous set $A\subset[n]$ with the maximum cardinality. By the definition of $c$, $|A|\geq d$. Let $B\subset A$ be the $\chi$-homogenous set obtained from $A$ by deleting both the first and the last $k(r-k+1)$ elements. By Lemma \ref{L3.5.}, $B$ is an interval in $[n]$. By Lemma \ref{L3.6.} we have $\min B<c+1$ and $\max B>n-c$, because $|B|\geq k(r-k+1)$ and $c\geq R_k((2k+2)r,l)$. Thus $[c+1,n-c]$ is a $\chi$-homogenous set and condition C1 in the definition of $c$-simplicity is satisfied.

Now we assume that $n > 4c+1$, for else the remaining condition C2 in the definition of $c$-simplicity is satisfied trivially. Let $v_1\in[c]\cup[n-c+1,n]$ and $v_2,\dots,v_{k-1}\in[n]$ be arbitrary $k-1$ distinct vertices. Because $|[c+1, n-c]| \geq 2c \geq 2(k-1)R$, we can use Lemma~\ref{L3.7.} with the set $[c+1, n-c]$ and obtain that all the edges $\{v_1,\dots,v_{k-1},w\}$, where $w\in[2c+1,n-2c]$ and $w\ne v_i$, have the same color. We see that $(n,\chi)$ is $c$-simple.
\end{proof}

\begin{proof}(Proof of Theorem \ref{T1CD}.) Let $X$ be an ideal of colorings. If $X$ contains an $r$-rich coloring for every $r\geq k$ then $|X_n|\geq n-k+2$ for every $n\geq k$ by Lemma \ref{L-RICH}. Otherwise, by Proposition \ref{P0}, the ideal $X$ is $c$-simple for some $c\in\mathbb{N}$. Applying Lemma \ref{L-SIMPLE} we get that $|X_n|$ is constant for all $n>n_0$.
\end{proof}

\section{Wealthy colorings}\label{sec_wealthy}

In this section we begin the proof of Theorem~1.2 on the quasi-Fibonacci dichotomy which is restated below. The proof will be completed in Section 5. 
The main goal of this section is to introduce certain ``wealthy'' colorings of triples and to estimate from below the growth functions of ideals containing 
these colorings. Before that say few more words on the quasi-Fibonacci dichotomy. The word ``Fibonacci'' refers to the sequence
$$
(F_n)_{n\geq 1}=(1,\,1,\,2,\,3,\,5,\,8,\,13,\,21,\,\dots)
$$
of the well known Fibonacci numbers, defined by the recurrence $F_1=F_2=1$ and $F_n=F_{n-1}+F_{n-2}$ for $n>2$. This resembles the recurrence 
$G_n=G_{n-1}+G_{n-3}$ defining the key sequence for Theorem~1.2, which is one reason why we refer to the Fibonacci numbers in the name of the 
dichotomy. Another reason is that several of the lower bounds for wealthy colorings in fact use $F_n$. The growths of the two sequences are  
$F_n\approx 1.618^n$ and $G_n\approx 1.466^n$. Yet another reason is that our Theorem~1.2 is an analogue of and is inspired by the following 
theorem in Klazar \cite{Klazar08}. This theorem was independently obtained, alongside with other results on ideals of ordered graphs, by Balogh, Bollob\'as and Morris 
\cite{BBM_Nes_sb}.

\begin{theorem}[\cite{Klazar08}, \cite{BBM_Nes_sb}]\label{T0}
Let $k=l=2$. Then for every ideal of colorings $X\subset\mathcal{C}_2$ (i.e. hereditary property of ordered graphs) there is a constant $c>0$ such that either 
$|X_n|<n^c$ for every $n\ge 2$ or $|X_n|\geq F_{n-1}$ for every $n\ge 1$.
\end{theorem}

\noindent
The reader will have no problem to find an ideal $X\subset\mathcal{C}_2$ showing that the lower bound is tight. So our second main result, 
restated next, is a close analog to the previous theorem; we actually use some parts of the proof of the previous theorem in the proof of the next one.

\Fibonaccidichotomy*

\noindent
The proof will proceed along similar lines as in \cite{Klazar08} but is considerably more complicated. In its first third we define here various 
``wealthy'' colorings and prove that the growth function of an ideal containing a large number of wealthy coloring grows at least as $G_n$. We call these colorings 
\emph{$r$-wealthy colorings of type $W_i$}, $r\in\mathbb{N}$ and $i\in[4]$. Some of them also have subtypes $W_{i,j}$. The meaning of the parameter $i$ is
that the underlying set of the coloring is $[ir]$ or $[ir+1]$. In the second third of the proof of Theorem~\ref{T1FD} in Section 4 we associate to colorings $(n,\chi)\in\mathcal{C}_3$
three-dimensional ``crossing'' matrices $M\colon[r]\times[s]\times[t]\to\{0,1,*\}$. In the last third of the proof in Section 5 we complete it by combining the results on wealthy colorings 
and crossing matrices. From now on always $k=3$ and $l=2$. In fact, we will use the two colors $\{0,1\}$, not $[2]$.

\subsection{Wealthy colorings $W_1$}

Let $r\ge 3$. A coloring $K=(r,\chi)$ is {\em $r$-wealthy of type $W_1'$} if, for $a, b \in \{ 0,1 \}$ with $a \ne b$, $K$ or its reversal satisfies
\begin{equation} \label{W1_1}
	\chi(\{1, 2, i\}) = \left\{ 
	      \begin{array}{ll}
	 		   	a        \quad & \text{for even } i \in [3, r]\, ,  \\
	 			b 		 \quad & \text{for odd } i \in [3, r]\, .
	 			\end{array} \right.
\end{equation}
 Similarly, $K$ is {\it $r$-wealthy of type $W_1''$ } if, for $a,b\in\{0,1\}$ with $a\ne b$, $K$ satisfies
\begin{equation} \label{W1_2}
	\chi(\{1, i, r\}) = \left\{ 
	      \begin{array}{ll}
	 		   	a        \quad & \text{for even } i \in [2, r-1]\, ,   \\
	 			b 		 \quad & \text{for odd } i \in [2, r-1]\, .
	 			\end{array} \right.
\end{equation}
Other edges may have any color. We call the $r$-wealthy colorings of type $W_1'$ and $W_1''$ summarily {\it $r$-wealthy of type $W_1$}. For $r=1,2$ these are 
just empty colorings (with no edge). Note that $r$-wealthy colorings of type $W_1$ are closed to taking reversals. 

\begin{lemma} \label{LW1}
If an ideal of colorings $X$ contains for every $r\geq3$ an $r$-wealthy coloring of type $W_1$ then $|X_n|\geq 2^{n-2}$ for every $n\in\mathbb{N}$.
\end{lemma}
\begin{proof}
Let $X$ be as given. It follows that either for infinitely many $r\in\mathbb{N}$ the ideal $X$ contains an $r$-wealthy coloring $L_r'$ of type $W_1'$, 
or for infinitely many $r\in\mathbb{N}$ it contains an $r$-wealthy coloring $L_r''$ of type $W_1''$. In fact, we may replace `infinitely many' with `every'. 

For $n = 1, 2$ the bound is trivial. Let $n \geq 3$. In the former case, for every $A \subset [3, n]$ there is a coloring $K_A = (n,\chi_A)\preceq L_{2n-2}'$ such that $\chi_A(\{1, 2, i\}) = 0$ if and only if $i\in A$. In the latter case, for every $B \subset [2, n-1]$ there is a coloring $K_B = (n, \chi_B)\preceq L_{2n-2}''$ such that $\chi_B(\{1, i, n\}) = 0$ if and only if $i \in B$. The sets $A$, resp. $B$, may be chosen in $2^{n-2}$ ways and for different $A$s, resp. $B$s, the colorings $K_A$, resp. $K_B$, are different. The bound follows.
\end{proof}

\noindent
The bound is actually a tight one, for consider the set of colorings $X\subset\mathcal{C}_3$ defined by $(n,\chi)\in X$ if and only if for every $E\in\binom{[n]}{3}$ 
with $E\not\supset\{1,2\}$ one has $\chi(E)=0$. Then it is easy to see that $X$ is in fact an ideal, that for every $r\in\mathbb{N}$ it contains an $r$-wealthy coloring
of type $W_1'$, and that for every $n\ge2$ one has $|X_n|=2^{n-2}$.

\subsection{Wealthy colorings $W_2$}

Let $r \in \mathbb{N}$. We say that a coloring $K = (n, \chi)$ is {\it $r$-wealthy of type $W_{2,1}'$} if $n = 2r+1$ and 
\begin{align} \label{W2_1}
	\chi(\{i,\,r+j,\, 2r+1\}) = \left\{ \begin{array}{ll}
     1 & \mbox{ for }\ i=j\, , \\
     0 & \mbox{ for }\ i\neq j
    \end{array} \right. 
\end{align}
where $i, j \in [r]$. Similarly we say that $K$ is {\it $r$-wealthy of type $W_{2,2}'$} if $n = 2r+1$ and 
\begin{align} \label{W2_2}
	\chi(\{i,\, r+j,\, 2r+1\}) = \left\{ \begin{array}{ll}
     1 & \mbox{ for }\ i \leq j\, , \\
     0 & \mbox{ for }\ i > j
    \end{array} \right.
\end{align}
where $i, j \in [r]$. In both colorings, colors of unspecified edges may be arbitrary. We visualize $K$ as an $r\times r$ matrix $(a_{i,j})$ 
such that $\chi(i,r+j,2r+1)$ is the entry in row $i$ and column $j$. A coloring $H$ is {\it $r$-wealthy of type $W_{2,1}$} if $H$ can be obtained either from 
an $r$-wealthy coloring $K$ of type $W_{2,1}'$ by swapping colors $0$ and $1$ and/or reversing the 
order of vertices in the interval $[1,r]$ and/or reversing the order of vertices in the interval $[r+1, 2r]$ and/or permuting the order of 
the three intervals $[1, r]$, $[r+1, 2r]$ and $\{2r+1\}$. Colorings of \emph{type $W_{2,2}$} are defined in an analogous way. 
A coloring of \emph{type $W_2$} is of type $W_{2,1}$ or type $W_{2,2}$. 
 
Swapping of the colors simply means that in the above definitions $1$ and $0$ are exchanged. Reversal of the order in an interval, for example 
in $[1,r]$ for type $W_{2,1}'$, means that in equation (\ref{W2_1}) we replace the left side with
$\chi(\{r-i+1,r+j,2r+1\})$. Permuting the order of the three intervals means, for example for type $W_{2,2}''$ and the reversing permutation 
(sending $123$ to $321$), that in equation (\ref{W2_2}) we replace the left side with $\chi(\{1,1+j,r+1+i\})$. We say more on symmetries of colorings and 
matrices at the beginning of Section~\ref{subsec_review2dim}. 

If $K=(2r+1,\chi)$ is an $r$-wealthy coloring of type $W_{2,1}$ (resp. $W_{2,2}$) and if
$K$ was obtained from an $r$-wealthy coloring $K'=(2r+1,\chi')$ of type $W_{2,1}'$ (resp. $W_{2,2}'$) by the above symmetries so that the intervals $[r]$, $[r+1,2r]$ and $\{2r+1\}$ of $K'$ were permuted in the intervals $A$, $B$ and $\{c\}$, we say that $A$, $B$ and $\{c\}$ are the \emph{base sets of $K$}. They play a role in the end of the proof of Theorem~\ref{T1FD}.

We use the next characterizations of the sequences $F_n$ and $G_n$. A \emph{binary string} is one from $\{0,1\}^*$.

\begin{lemma} \label{LFibG}
For the Fibonacci sequence $(F_n)_{n\geq1}$ and sequence $(G_n)_{n\geq1}$ the following holds ($n\ge2$).
\begin{enumerate}
\item $F_n$ equals to the number of binary strings $w = w_1w_2\ldots w_{n-2}$ of length $n-2$ not containing substring $00$. The same holds 
for substring $11$.
\item $F_n$ equals to the number of binary strings $w = w_1w_2\ldots w_{n-2}$ of length $n-2$ not containing substrings $w_{2i}w_{2i+1} = 10$ and $w_{2i-1}w_{2i} = 01$, resp. vice versa.
\item We have $2^n>F_n\ge G_n$ for every $n\in\mathbb{N}$ and the last inequality is strict for $n>4$.
\end{enumerate}
\end{lemma}

\begin{proof}
All these results follow easily by induction on $n$. 
\end{proof}

Recall that the identity matrix $I_n$ has size $n\times n$, has $1$s on the main diagonal, and $0$s elsewhere. Similarly, the upper triangular matrix $U_n$ 
has size $n\times n$, has $1$s on the main diagonal and above it, and $0$s elsewhere. A matrix $M$ is contained in another matrix $N$, or $M$ is a submatrix of $N$, 
if $M$ can be obtained from $N$ by deleting rows and columns. 

\begin{lemma} \label{L:DiagonalPaths}
For $n\in\mathbb{N}$, let $I_n$ and $U_n$ be the $n\times n$ identity matrix and the $n\times n$ upper triangular matrix, respectively.
\begin{enumerate}
\item For any binary string $w = w_1w_2\ldots w_{2n-1}$ of length $2n-1$ avoiding any substring $w_iw_{i+1}=11$ there exists a matrix $M$ of size $n \times n$ that is contained in $I_{2n}$ and such that 
$M(i,i) = w_{2i-1}$ for $i \in [n]$ and $M(i, i+1) = w_{2i}$ for $i \in [n-1]$.
\item For any binary string $w = w_1w_2\ldots w_{2n-1}$ of length $2n-1$ avoiding any substring $w_{2i-1}w_{2i} = 10$ and $w_{2i}w_{2i+1} = 01$ there is a matrix $M$ of size $n\times n$ that is  contained in $U_{3n}$ and such that $M(i,i) = w_{2i-1}$ for $i \in [n]$ and $M(i, i+1) = w_{2i}$ for $i \in [n-1]$.
\end{enumerate}
\end{lemma}

\begin{proof}
1. Suppose that $w = w_1w_2\ldots w_{2n-1}$ avoids consecutive substrings $11$. To get $M$, for each $i \in [n]$ we choose a row and a column of $I_{2n}$. For different $i$ the chosen rows are different, and so are the chosen columns. The matrix $M$ will consists of the chosen rows and columns. 

We proceed as follows. For $w_{2i-1} = a\in\{0,1\}$ we choose the row $2i-a$. For $w_{2i} = b\in\{0,1\}$ we choose the column $2i+1-b$, except for $i=n$ when we always choose the column number one. For an example with $n=4$ see Fig.~\ref{fig:forbidden} (a). It follows that the resulting matrix $M$ has the stated property.

\begin{figure}
    \centering
    \subfloat[$w=0100101$]{{\includegraphics[width=6cm]{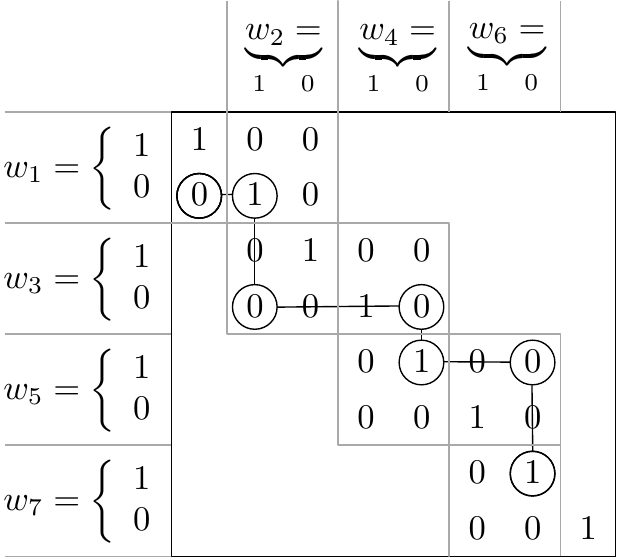} }}%
    \qquad
    \subfloat[$w=01110$]{{\includegraphics[width=6cm]{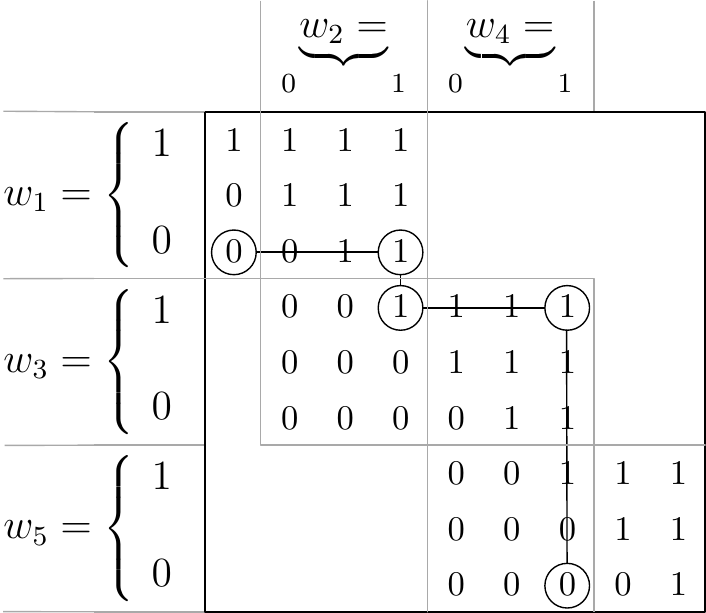} }}%
    \caption{The submatrices of $I_{2n}$ and $U_{3n}$ for sequences $w$ avoiding certain patterns.}%
    \label{fig:forbidden}%
\end{figure}

2. Suppose that $w = w_1w_2\ldots w_{2n-1}$ avoids consecutive substrings $w_{2i-1}w_{2i} = 10$ and $w_{2i}w_{2i+1} = 01$. 
We choose rows and columns like in part 1. If $w_{2i-1} = a\in\{0,1\}$ we choose the row $3i-2a$. If $w_{2i} = b\in\{0,1\}$ we choose the column $3i-1+2b$, except for $i=n$ when we choose the column number one. For an example  with $n=3$ see Fig.~\ref{fig:forbidden} (b). Again, the resulting matrix $M$ has the stated property.
\end{proof}

Probably, the sizes of the matrices $I_{2n}$ and $U_{3n}$ in Lemma~\ref{L:DiagonalPaths} are not minimal. We think that we could prove the lemma with matrices $I_{3n/2}$ and $U_{2n}$. 

\begin{proposition} \label{LW2}
If an ideal of colorings $X$ contains for every $r \in \mathbb{N}$ an $r$-wealthy coloring of type $W_2$ then $|X_n| \geq F_n$ for every $n \in \mathbb{N}$.
\end{proposition}

\begin{proof}
Let $X$ be an ideal of colorings. We first assume that for every $r\in\mathbb{N}$, $K_r = (2r+1, \chi)$ is an $r$-wealthy coloring of type $W_2'$ such that $K_r\in X$. Let $n=2m+1\in\mathbb{N}$ be odd. We take the $2m$-wealthy coloring $K_{2m}=(4m+1,\chi)\in X$ of type $W_2'$ and apply part 1 of Lemma~\ref{L:DiagonalPaths}. By it, for each binary string $w=w_1w_2\dots w_{2m-1}$ avoiding consecutive substrings $11$ there  is a coloring $K_w=(n,\chi_w)=(2m+1,\chi_w)$ contained in $K_{2m}$, thus in $X$, such that for every $i\in[m]$ one has $\chi_w(\{i,m+i,2m+1\})=w_{2i-1}$ and for every $i\in[m-1]$ one has $\chi_w(\{i,m+1+i,2m+1\})=w_{2i}$. For different strings $w$ these colorings are different, and using part 1 of Lemma~\ref{LFibG} we get the lower bound $|X_{2m+1}| \geq F_{2m+1}$.

To bound $|X_{n}| = |X_{2m}|$ for even $n \in \mathbb{N}$ we use a variant of the latter justification. We take any binary string $w = w_1w_2\dots w_{2m-2}0$ of length $2m-1$ that avoids consecutive substrings $11$. By Lemma~\ref{L:DiagonalPaths} there is an $m \times m$ matrix $M_{\omega}$ contained in $I_{2m}$ such that $M_{\omega}(i,i) = w_{2i-1}$ for $i \in [n]$ and $M_{\omega}(i, i+1) = w_{2i}$ for $i \in [n-1]$. Any of these matrices produces a different coloring $K_{\omega} = (2m+1, \chi_{\omega})$ such that $\chi_{\omega}(\{i,m+i, 2m+1\}) = w_{2i-1}$ for $i \in [m]$ and $\chi_{\omega}(\{i, m+1+i\}) = w_{2i}$ for $i \in [m-1]$. We denote, for any latter string $w$, a colorings $K_{\omega}'$ by restricting and normalizing $K_{\omega}$ to $[2m+1] \setminus \{m\}$. For different strings $\omega$ the colorings $K_{\omega}' = (n, \chi')$ are different, because $w_{2m-1} = 0$, and hence by part 1 of Lemma~\ref{LFibG} we have $|X_{2m}| \geq F_{2m}$.

When $X$ contains for every $r\in\mathbb{N}$ an $r$-wealthy coloring of type $W_2''$, we argue similarly. For odd $n = 2m+1$ we take the $3m$-wealthy coloring $K_{3m}=(6m+1,\chi)\in X$ of type $W_2''$, apply part 2 of Lemma~\ref{L:DiagonalPaths} and part 2 of Lemma~\ref{LFibG} and obtain $|X_{2m+1}| \geq F_{2m+1}$. For even $n = 2m$ we argue similarly and restrict to the particular strings $w$ with $w_{2m-1} = 0$ and again apply part 2 of Lemma~\ref{L:DiagonalPaths} and part 2 of Lemma~\ref{LFibG}. Now we restrict our colorings to the set $[2m+1] \setminus \{m\}$ and obtain $|X_{2m}| \geq F_{2m}$ since $w_{2m-1} = 0$ guarantees that the restricted colorings are mutually different.

In the general case when $X$ contains an $r$-wealthy coloring of type $W_2$, we consider one of the above described transformations $T$ transforming the ``canonical'' coloring $W_2'$ or $W_2''$ to the given $r$-wealthy coloring of type $W_2$. The images of the colorings $K_w$ under $T$ give then the stated lower bound for the general $r$-wealthy coloring of type $W_2$.
\end{proof}

\subsection{Wealthy colorings $W_3$}

Let $r \in \mathbb{N}$. We introduce wealthy colorings of type $W_3$. A coloring $K = (n,\chi)$ is 
\begin{itemize}
 \item {\it $r$-wealthy of type $W_{3, 1}'$} if $n = 3r$ and
  \begin{align} \label{W3_1}
	\chi(\{i,\, r+i,\, 2r+j\}) = \left\{ \begin{array}{ll}
     1 & i=j\, , \\
     0 & i \not= j\, ,
    \end{array} \right.
  \end{align}
  where $i, j \in [r]$. Colors of the remaining edges are not specified. We say that a coloring $H$ is {$r$-wealthy of type $W_{3, 1}$} if $H$ can be obtained from $K$ by swapping colors $0$ and $1$ and/or reversing the order of vertices in some of the intervals $[1, r]$, $[r+1, 2r]$ and $[2r+1, 3r]$ and/or permuting these intervals.
 \item {\it $r$-wealthy of type $W_{3, 2}'$} if $n = 3r$ and 
  \begin{align} \label{W3_2}
	\psi(\{i,\, r+i,\, 2r+j\}) = \left\{ \begin{array}{ll}
     1 & i \leq j\, , \\
     0 & i > j\, ,
    \end{array} \right.
  \end{align}
  where $i, j \in [r]$. Again, colors of the remaining edges are arbitrary. We say that coloring $H$ is {$r$-wealthy of type $W_{3, 2}$} if $H$ can be obtained from $K$ by swapping colors $0$ and $1$ and/or reversing the order of vertices in some of the intervals $[1, r]$, $[r+1, 2r]$ and $[2r+1, 3r]$ and/or permuting these intervals.
 \item {\it $r$-wealthy of type $W_{3, 3}$ } if $n = 3r+1$ and either $K$ or its reversal satisfy that for any $i \in [r]$ there are distinct numbers $a_i, b_i, c_i \in \{3i-2, 3i-1, 3i\}$ such that $\chi(\{a_i, b_i, 3r+1 \}) \ne \chi(\{a_i, c_i, 3r+1\})$.
\end{itemize}
We call these colorings summarily $W_3$ type colorings. Their symmetries are the same as those explained in the definition of type $W_2$ colorings.

As for type $W_2$ colorings but more simply here we define the \emph{base sets} of an $r$-wealthy coloring $(3r,\chi)$ of type $W_{3,1}$ or $W_{3,2}$ to be the intervals $[1, r]$, $[r+1, 2r]$ and $[2r+1, 3r]$.

Let $k\leq n$ and $A=(a_1,a_2, \dots, a_k)$ be a $k$-tuple with $a_i=(c_i,d_i)\in [n]^2$ satisfying
$$
1 \leq c_1 < c_2 < \dots < c_k \leq n \quad \text{and} \quad 1 \leq d_1 < d_2 < \dots < d_k \leq n
$$
---we call $A$ a \emph{$n$-chain}. We denote by $A^*$ the $n\times n$ \emph{binary}---with entries $0$ and $1$---matrix that has $1$s exactly in the positions $a_i$, $i \in [k]$.

\begin{lemma}\label{L:Dictionary}
Let $k\leq n$, $A = (a_1, a_2, \dots, a_k)$ be an $n$-chain, and $A^*$ be the corresponding $n\times n$ binary matrix. Then $A^*$ is a submatrix of the identity matrix $I_{2n}$.
\end{lemma}

\begin{proof}
In fact, we prove that $A^*$ is a submatrix of the identity $I_{2n-k}$. Let $A$ be an $n$-chain. We obtain some $n$ rows and $n$ columns of $I_{2n}$ forming the matrix $A^*$ by deleting some rows and columns from $I_{2n}$ as follows. We additionally set $a_0 = (0, 0)$, $a_{k+1} = (n+1, n+1)$, and consider the differences 
$$
(c'_i,\, d'_i):=(c_i-c_{i-1},\, d_i-d_{i-1})=a_i-a_{i-1},\, i \in [k+1]\, .
$$
We denote by $\overline{A}$ the $(n+1)\times(n+1)$ matrix obtained from $A^*$ by adding at the bottom of $A^*$ and to the right of it a zero row and a zero column, and changing the zero in their intersection 
to $1$. 

To any $a_{i} = (c_{i}, d_{i}) \in A \cup \{(n+1, n+1)\}$ for $i\in[k+1]$ we assign the submatrix $J_i$ of $I_{2n+1-k}$ formed by the rows and columns (with indices) in the interval $[c_{i-1}+d_{i-1}-(i-2), c_i+d_i-i]$, $i\in [k+1]$, and the submatrix $M_i$ of $\overline{A}$ formed by the rows in $[c_{i-1}+1, c_i]$ and the columns in $[d_{i-1}+1, d_i]$. Note that every $J_i$
is the identity matrix with size $c_i'+d_i'-1$, and that the (in general non-square) $c_i'\times d_i'$ matrix $M_i$ has $1$ in the south-east corner and zeros elsewhere. 

We show that each $J_i$ contains $M_i$. This, together with the fact that the matrices $J_i$, resp. $M_i$, cover all ones in $I_{2n+1-k}$, resp. in $\overline{A}$, and follow one after another, proves the stated claim. One easily checks that deleting the first $d_i'-1$ rows of $J_i$ and the following $c_i'-1$ columns of $J_i$ yields $M_i$. 
Indeed, since the first $d_i'-1$ ones of $J_i$ are contained in deleted rows and the following $c_i'-1$ ones are contained in the deleted columns, only one $1$ (the last one) survives for $M_i$. 
The size of the resulting matrix is that of $M_i$ as it has $c_i'+d_i'-1 - (d_i'-1) = c_i'$ rows and $c_i'+d_i'-1 - (c_i'-1) = d_i'$ columns. 

\begin{figure}
    \centering
    \subfloat[$10 \times 10$ matrix $\overline{A}$]{{\includegraphics[width=4.5cm]{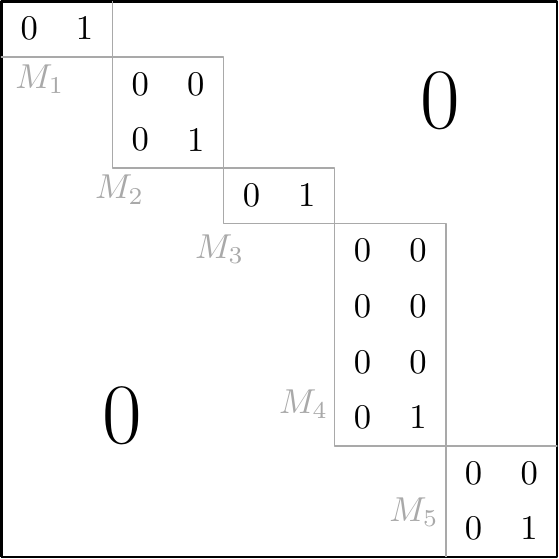} }}%
    \qquad
    \subfloat[Deleted rows in $I_{2n+1-k}$]{{\includegraphics[width=5cm]{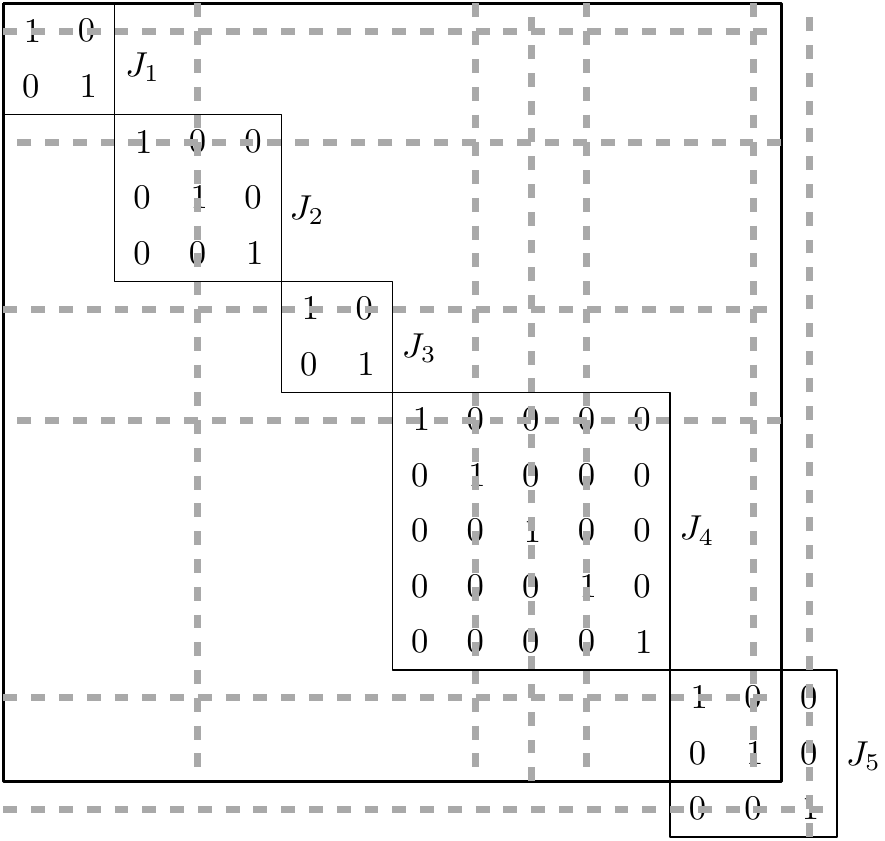} }}%
    \caption{Finding $n$-chain $A = \{(1, 2), (3, 4), (4, 6), (8, 8)\}$ with $n=9$ and $k=4$ as a submatrix $\overline{A}$ in the identity matrix $I_{2n+1-k}$.}%
    \label{fig:forbiddenM}%
\end{figure}

Thus $\overline{A}$ is a submatrix of $I_{2n+1-k}$ and $A^*$ is a submatrix of $I_{2n-k}$. The reader can follow the construction on an example in Figure \ref{fig:forbiddenM}, where $A = \{(1, 2), (3, 4), (4, 6), (8, 8)\}$ and $n=9$.
\end{proof}

For $m\in\mathbb{N}$ we call $C = (c_1, c_2, \dots, c_{2m+1})\subset[m+1]^2$ a {\it southeast path in $[m+1]^2$} if 
$$
c_1 = (1, 1),\, c_{i+1} - c_i \in \{(0,\, 1),\,\ (1,\, 0)\},\, \mathrm{and}\ c_{2m+1} = (m+1, m+1)\, .
$$
Clearly, the number of southeast paths in $[m+1]^2$ equals $\binom{2m}{m}$ because they correspond to the $m$-element subsets (of the steps $(0,1)$, say) of the set $[2m]$ (of all steps).

\begin{proposition} \label{bij_chain_path}
For $m\in\mathbb{N}$ the set of $m$-chains is in bijection with the set of southeast paths in $[m+1]^2$, and therefore we have exactly $\binom{2m}{m}$ $m$-chains.
\end{proposition}

\begin{proof}
We consider an $m\times m$ matrix drawn in the plane as an $m\times m$ square array of $m^2$ unit squares which we label by the coordinates $(i,j)\in[m]^2$ in the matrix way,
top to bottom and left to right. We label their corners by the elements of $[m+1]^2$ also in the matrix way. Then any southeast path $C$ in $[m+1]^2$ consists of $2m+1$ of these
corners, starts in the corner $(1,1)$ and ends in the corner $(m+1,m+1)$. To each $C$ we associate the binary $m\times m$ matrix $D^*$ with $1$s exactly in the squares around which 
$C$ makes a left turn, that is, of the squares whose both bottom and both left side array points lie in $C$. Clearly, $D^*$ is a matrix of a unique $m$-chain. In the other way, to any $m$-chain $D$ we associate a southeast 
path $C$ in $[m+1]^2$ that starts in the corner $(1,1)$, then goes horizontally until it reaches the left boundary of the column containing the first $1$ of $D^*$, then $C$ makes a right turn and 
goes vertically until it reaches the bottom boundary of the row containing the first $1$, then $C$ makes a left turn around the first $1$, then continues in the similar way to make a 
left turn around the second $1$ in $D^*$, and so on until $C$ finishes in the corner $(m+1,m+1)$. If $D^*$ is the zero matrix or if $C$ made turns around all $1$s in $D^*$, $C$ goes horizontally 
to the right boundary of column $m$ and then vertically to the corner $(m+1,m+1)$. The two described  associations are inverses of one another and give the required bijection.
\end{proof}

\begin{corollary}\label{new_cor}
For any $m\ge1$ there are at least $\binom{2m-2}{m-1}$ $m$-chains $D$ such that the last column of $D^*$ contains only zeros. The same bound holds for the number of $m$-chains $E$ such that the last row of $E^*$ contains only zeros.
\end{corollary}
\begin{proof}
This number is at least the number of $m$-chains that have in both the last column and the last row only zeros. These bijectively correspond to $(m-1)$-chains, and we can use formula 
from the previous proposition. 
\end{proof}

\begin{proposition} \label{LW312}
If an ideal of colorings $X$ contains for every $r\geq 3$ an $r$-wealthy coloring of type $W_{3, 1}$, or for every $r\geq 3$ an $r$-wealthy coloring of type $W_{3, 2}$, then for every $n\in\mathbb{N}$,
$$
|X_n| > \frac{2^{2(n-2)/3}}{\sqrt{2n}}>\frac{0.28\cdot1.587^n}{\sqrt{n}}\, .
$$
Hence $|X_n|\ge G_n$ for every $n \ge 23$.
\end{proposition}

\begin{proof}
First we consider $r$-wealthy colorings of type $W_{3,1}$. Without loss of generality we consider an ideal $X$ that for every $r\geq 3$ contains an $r$-wealthy coloring $K_r = (3r, \chi_r)$ of type $W_{3, 1}'$. We handle the general case by a transformation argument as in the end of the proof of Proposition~\ref{LW2}. Let $H = (n, \chi)$ be a coloring such that $n = 2m_1+m_2$. We define the $m_1\times m_2$ {\em partition matrix $P[H]$ of $H$} 
by $P[H](i, j) = \chi(\{i, m_1+i, 2m_1+j\})$, $i\in[m_1]$ and $j\in[m_2]$. If $H = K_r$ then, clearly, $P[H]$ is the identity matrix $I_r$.

Now let $n = 3m$ and $m_1 = m_2 = 2m$. We set $r = 2n$ and consider an arbitrary $m$-chain $D$. Since $D^*$ is a subset of $I_{2m}$ and $I_{2m} = P[K_r]$ is the partition matrix of $r$-wealthy coloring $K_r$, it follows by Lemma~\ref{L:Dictionary} that the coloring $L_D=(3m,\chi_D)\in X$ such that $\chi_D(\{i,m+i,2m+j\}) = D^*(i,j)$ is contained in $K_r$, thus in $X$.
It is clear that $D\ne D'\implies L_D\ne L_{D'}$. Hence by Proposition~\ref{bij_chain_path} 
and by the lower bound  $\binom{2m}{m}\ge\frac{1}{2\sqrt{m}}2^{2m}$ for every $m\in\mathbb{N}$ (see N.\,D. Kazarinoff \cite{kaza}) we have that
$$
    |X_n|=|X_{3m}|\ge\binom{2m}{m}\geq\frac{2^{2m}}{\sqrt{4m}} > \frac{2^{2n/3}}{\sqrt{4n/3}}\, .
$$
To lowerbound $|X_n|=|X_{3m-1}|$, we take the colorings from the set $T=\{L_D\;|\;\forall\,i\in[m]:\ D^*(i,m)=0\}$, that is, for the $D^*$ having in the last $m$-th column only zeros. By deleting the vertex $3m$ in each $L_D \in T$
we obtain colorings $T'\subset X_{3m-1}$ that are still mutually different because their partition matrices arise from those $D$ just by deleting the last zero column. By Corollary~\ref{new_cor},
$$
    |X_n|=|X_{3m-1}|\ge|A'|=|A|\ge\binom{2m-2}{m-1}\ge\frac{2^{2m-2}}{\sqrt{4m-4}}>\frac{2^{2(n-2)/3}}{\sqrt{4n/3}}\, .
$$
To lowerbound $|X_n|=|X_{3m-2}|$ we argue similarly, we consider the colorings $L_D$ such that $D$ has in the last $m$-th row only zeros and delete in those $L_D$ the two vertices $m$ and $2m$. Again by Corollary~\ref{new_cor},
$$
    |X_n|=|X_{3m-2}|\ge\binom{2m-2}{m-1}\ge\frac{2^{2m-2}}{\sqrt{4m-4}}>\frac{2^{2(n-1)/3}}{\sqrt{4n/3}}\, .
$$
This gives the stated bound.

For the $r$-wealthy colorings of type $W_{3,2}$ the argument is similar, and we only sketch the necessary modifications to the argument from the previous paragraph. The new $n$-chains are now
the $n\times n$ binary matrices that are obtained from the (old) $n$-chains by changing every zero above and to the right of any $1$ also to $1$. Equivalently, the new $n$-chains are the $n\times n$ binary 
matrices with no $0$ above or to the right of an $1$; each such matrix is uniquely determined by the (old) $n$-chain of $1$s with no other $1$ to the left or below. We then have an analogy to Lemma~\ref{L:Dictionary} and can
find every new $n$-chain as a submatrix in the matrix $U_{2n}$ (in the proof we simply mirror the steps in $I_{2n}$ by steps in $U_{2n}$). We have the same formula for the number of new $n$-chains as in 
Proposition~\ref{bij_chain_path}, the south-east paths $C$ are now exactly borders between the area of $0$s and the area of $1$s. Considering the colorings $L_D=(n,\chi)\in X$ where $n=3m$, $D$ is a new $m$-chain, and
the partition matrix of the coloring equals $D^*$, and recalling that each $D$ is determined by a certain (old) $m$-chain in $D$, we see that we get the same lower bounds on $|X_{3m}|,|X_{3m-1}|$, and $|X_{3m-2}|$
as before in the previous paragraph.

To bound an $n_0$ such that $|X_n|\ge G_n$ if $n\ge n_0$, we estimate the numbers $G_n$ from above. Since $G_0=G_1=G_2=1$ and $G_n=G_{n-1}+G_{n-3}$ for $n\ge3$, induction shows that
for every integer $n\ge0$ one has $G_n\le c \alpha^n$ where $\alpha=1.46557\dots<1.466$ is the only positive root of the polynomial $x^3-x^2-1$ and $c = 0.417\cdots$ is a constant generated by the initial terms of $G_n$. Thus we need an $n_0$ such that if $n\ge n_0$ then
$$
    \frac{0.343 \cdot 1.587^n}{\sqrt{n}} > 0.418 \cdot 1.466^n\, ,
$$
where $0.343 < 2^{-4/3} / \sqrt{4/3}$. Since $(0.343/0.418)^2>0.673$ and $(1.587/1.466)^2>1.171$, we need an $n_0$ such that if $n\ge n_0$ then $0.69 \cdot 1.171^n > n$. It is easy to compute that $n_0=23$ suffices, which gives the stated $n_0$. 
\end{proof}

\begin{lemma} \label{LW33}
If an ideal of colorings $X$ contains for every $r\in\mathbb{N}$ an $r$-wealthy coloring of type $W_{3,3}$ then $|X_n|\geq F_n$ for every $n\in\mathbb{N}$.
\end{lemma}
\begin{proof}
This bound follows by applying \cite[Lemma 3.10]{Klazar08}. To any coloring $K=(n,\chi)$, $n\ge2$, of triples we associate a coloring $K'=(n-1,\psi)$ of pairs, $\psi\colon\binom{[n-1]}{2}\to\{0,1\}$, by
$$
    \psi(\{x,\, y\}) =\chi(\{x,\, y,\, n\})\, .
$$
It follows that $X'=\{K'\;|\;K\in X\}$ is an ideal of colorings of pairs. Since we assume that we may take for any $r$ the coloring $K\in X$ to be an $r$-wealthy coloring of type $W_{3,3}$, it follows that $X'$ 
contains for every $r$ an $r$-wealthy coloring of type $2$ as defined in \cite{Klazar08} before \cite[Lemma 3.9]{Klazar08} (these are coloring of pairs in $[3r]$ such that no triple $\{3i-2,3i-1,3i\}$, $i \in [r]$,
is monochromatic). By the second claim of \cite[Lemma 3.10]{Klazar08}, $|X_n'|\geq F_n'$ for every $n\ge1$ where $F_n'$ are the Fibonacci numbers in \cite{Klazar08} which relate to the $F_n$ here by $F_n'=F_{n+1}$.
Since for every two colorings of triples $K_1=(n,\chi_1)$ and $K_2=(n,\chi_2)$, $n\ge1$, we have $K_1'\ne K_2'\implies K_1\ne K_2$, we deduce that
$$
    |X_n|\ge|X_{n-1}'|\ge F_{n-1}'=F_n,\, n\ge2\, .
$$
For $n=1$ this bound holds trivially as well.
\end{proof}

\subsection{Wealthy colorings $W_4$}

We introduce wealthy colorings of type $W_4$. A coloring $K = (n,\chi)$ is 

\begin{itemize}
\item {\it $r$-wealthy of type $W_{4, 1}$} if $n = 4r$ and none of the consecutive intervals $\{4i-3, 4i-2, 4i-1, 4i\}$, $i \in [r]$, is monochromatic, 
\item {\it $r$-wealthy of type $W_{4, 2}'$} if $n = 4r$ and for every $i \in [r]$ there are distinct numbers $a_i, b_i, c_i \in \{r+3i-2, r+3i-1, r+3i\}$ such 
that $\chi(\{i, a_i, b_i\}) \ne \chi(\{i, a_i, c_i\})$. A coloring $K = (4r,\chi)$ is $r$-wealthy of type $W_{4, 2}$ if it is obtained from an $r$-wealthy coloring of type $W_{4, 2}'$
by possibly reversing the order of elements in $[r]$ and/or swapping the intervals $[r]$ and $[r+1,4r]$ (so that they become $[3r]$ and $[3r+1,4r]$, respectively). 
\end{itemize}
We call these colorings summarily $W_4$ type colorings. Observe that both types of $W_4$ colorings are closed to reversal.

As for type $W_2$, $W_{3,1}$ and $W_{3,2}$ colorings we need later the \emph{base sets} of an $r$-wealthy coloring $K=(4r,\chi)$ of type $W_{4,2}$ that 
is obtained by the above symmetries from an $r$-wealthy coloring $K'=(4r,\chi)$ of type $W_{4,2}'$. These are the two intervals $A$ and $B$ obtained from the intervals $[r]$ and $[r+1,4r]$ of $K'$ by their possible swapping.

\begin{proposition} \label{LW41}
If an ideal of colorings $X$ contains for every $r \in \mathbb{N}$ an $r$-wealthy coloring of type $W_{4,1}$ then $|X_n|\geq G_{n}$ for every $n \in \mathbb{N}$, where the numbers $G_n$ are defined at the start of the article.
\end{proposition}

\begin{proof}
Let $r \in \mathbb{N}$ and $K_r = (4r,\chi_r)\in X$ be an $r$-wealthy coloring of type $W_{4,1}$. For $i \in [r]$ consider sets
$$
    S_i = \{4i-3,\, 4i-2,\, 4i-1,\, 4i\} \quad \mbox{and} \quad V = \{S_i \; | \; i \in [r]\}\, .
$$
To simplify the situation by the Ramsey theorem we consider the coloring $c\colon\binom{V}{3} \to \{0, 1\}^{220}$, where $220=\binom{12}{3}$, defined by 
\begin{equation*}
	c(S_{i_1},\, S_{i_2},\, S_{i_3}) = (\chi_r(E) \mid E \subset S_{i_1} \cup S_{i_2} \cup S_{i_3},\, |E| = 3) \in \{0,\, 1\}^{220},\ 1 \le i_1 < i_2 < i_3 \leq r\, ,
\end{equation*}
where the triples $E$ are ordered lexicographicly according to their vertices: $E=\{a_1<a_2<a_3\}$ comes before $E'=\{a_1'<a_2'<a_3'\}$ if and only if $a_1<a_1'$ or ($a_1=a_1'$ and $a_2<a_2'$) or ($a_1=a_1'$ and $a_2=a_2'$ and $a_3<a_3'$). 
Since $r$ may be as large as we need and $X$ is an ideal, by the Ramsey theorem for $3$-uniform hypergraphs we may suppose that the coloring $c$ is constant. This simplification of $K_r$ is called the {\it shift condition}.

Since in each $S_i$ we have two triples with distinct colors, there exist triples $A \subset S_1$ and $A'\subset A\cup(A+4)$ such that $\chi_r(A') \ne \chi_r(A)$. Indeed, since $S_i$ is not monochromatic, there are two triples $E_1, E_2 \subset S_i$ such that $\chi_r(E_1) \ne \chi_r(E_2)$ and $|E_1 \cap E_2| = 2$, and easy discussion shows that either $A = E_1$ or $A = E_2$ works. For all $i \in [r]$ we set $T_i = A+4(i-1)\subset S_i$, and $L_r = (3r, \psi_r)$ to be the restriction and normalization of $\chi_r$ to (the triples in) $T_1\cup T_2 \cup \dots \cup T_r$. Clearly $L_r\preceq K_r$ for every $r \in \mathbb{N}$. Each $L_r$ is determined by its restriction to $[9]$ because of the shift condition; we set $M=(9,\psi)$ to be this restriction. Without loss of generality, $\psi(A)=\psi(T_1) = 1$.

We reveal the connection of colors of $M$. We may suppose that all triples $E\subset[9]$ with $|E \cap T_i| = 1$, for $i = 1, 2, 3$, have the same color $\psi(E)$. If not, then we find triples $E_1$, $E_2$ such that $|E_1 \cap E_2| = 2$ and $|E_j \cap T_i| = 1$, for $i \in [3]$ and $j \in [2]$. It follows that there is an $r$-wealthy subcoloring of type $W_1$ and Lemma \ref{LW1} applies. Let $s = \psi(\{1, 4, 7\})$. We claim that $\psi(E) = s$ for all triples $E\subset[9]$ with $|E \cap T_1| = 1$ and $|E \cap T_2| = 2$ (or equivalently, by the shift condition, $|E \cap T_3| = 2$). Indeed, if there is a triple $E = \{a, b, c\}$ with $a \in T_1$, $b, c \in T_2$, $b < c$, and $\psi(E) = t \ne s$ ($\{s, t\} = \{0, 1\}$) then the shift condition implies 
$$
    \psi_r(\{a,\, b+3m,\, c+3m\}) = t,\ m \in \{0,\, 1,\, \dots,\, r-2\}\, ,
$$
and
$$
    \psi_r(\{a,\, b+3m,\, c+3m+3\}) = \psi_r(\{a,\, c+3m,\, b+3m+3\}) = s,\ 
m\in \{0,\, 1,\, \dots,\, r-3\}\, ,
$$
and therefore we can construct for any binary string $w = w_1w_2\ldots w_{n-2}$ avoiding the substring $tt$ a coloring $(n, \lambda_w)\preceq L_{n}$ such that, for $i=2,3,\dots,n-1$, $\lambda_w(\{1, i, i+1\}) = w_{i-1}$. By parts $1$ and $3$ of Lemma~\ref{LFibG}, $|X_n| \geq F_{n}\ge G_n$. Switching to the reversals we deduce in a similar way that $\psi(E) = s$ for $E$ satisfying $|E \cap T_1| = 2$ and $|E \cap T_2| = 1$ (or $|E \cap T_3| = 1$).

Thus $\psi(T_1) = \psi(T_2) = \psi(T_3) = 1$ and $\psi(E) = s$ for all other triples $E \subset T_1 \cup T_2 \cup T_3$. We have $s = 0 \ne 1$ by the condition on the colors of $A$ and $A'$. So in $L_r$, $\psi_r(T_i) = 1$ for $i \in [r]$ and all other triples have color $0$. Thus $S(3) \subset X$, where $S(3)$ is the ideal of colorings mentioned after the statement of Theorem~\ref{T1FD}, and $|X_n| \geq G_{n}$.
\end{proof}

\noindent
As for type $W_1$ colorings, also this lower bound is tight. Indeed, the ideal $S(3)\subset\mathcal{C}_3$ has growth $|S(3)_n|=G_n$ and for every $r\in\mathbb{N}$
contains an $r$-wealthy coloring of type $W_{4,1}$. In more details, $S(3)$ consists of the colorings $(n,\chi)$ for which there exist $3$-intervals $I_1<I_2<\dots<I_r$
in $[n]$ such that $\chi(I_j)=0$ for every $j$ but $\chi(E)=1$ for all other edges $E$. In particular, for every $r\in\mathbb{N}$ one has that $(4r,\chi_r)\in S(3)$ where 
$\chi_r(\{4j-3,4j-2,4j-1\})=0$ for $j\in[r]$ and $\chi_r(E)=1$ for all other edges $E$. But  $(4r,\chi_r)$ is an $r$-wealthy coloring of type $W_{4,1}$. 

\begin{proposition}\label{LW42}
If an ideal of colorings $X$ contains for every $r \in \mathbb{N}$ an $r$-wealthy coloring of type $W_{4,2}$ then 
$$
    |X_n| \ge \binom{\lfloor \frac{2(n-4)}{5} \rfloor }{\lfloor \frac{n-4}{5} \rfloor}^2\approx 1.741^n
$$
for every $n \ge 9$. Thus $|X_n| \geq F_n$ for all $n \ge 75$ and $|X_n| \geq G_n$ for all $n \ge 20$.
\end{proposition}

\begin{proof}
We suppose that $X$ is an ideal of colorings and that $K_r = (4r, \chi_r) \in X$, $r \in \mathbb{N}$, for some $r$-wealthy colorings $K_r$ of type $W_{4, 2}$. We may suppose that for every $i \in [r]$ and for $\{r+3i-2, r+3i-1, r+3i\}=\{a_i, b_i, c_i\}$ we have $\chi(\{i, a_i, b_i\}) \ne \chi(\{i, a_i, c_i\})$. Let 
$$
    S_i = \{i,\, r+3i-2,\, r+3i-1,\, r+3i\}\, ,
$$
$i \in [r]$. We consider various cases and in all but the last one we show that $|X_n|\ge F_n\ge G_n$ for every $n\ge1$. 
In the last case we still have the bound $|X_n|\ge F_n$ but only for $n\ge n_0$. To get a smaller $n_0$, we therefore compare in this case $|X_n|$ with $G_n$ instead. 

{\bf Step 1: shift condition. }We consider the same coloring 
$$
    c :\; \binom{V}{3} \to \{0,\, 1\}^{220}\ \mathrm{with} \ V = \{S_i \;|\; i \in [r]\}
$$
as in the previous proof (note, however, that the quadruples $S_i$ are now different):
\begin{equation*}
    c(S_{i_1},\, S_{i_2},\, S_{i_3}) = (\chi_r(E) \mid E \subset S_{i_1} \cup S_{i_2} \cup S_{i_3},\, |E|=3) \in \{0,\, 1\}^{\binom{12}{3}},\ 1 \le i_1 < i_2 < i_3 \le r
\end{equation*}
(with the same lexicographic order of the triples $E$). As before we may suppose using the Ramsey theorem for $3$-uniform hypergraphs that for each $K_r$ the coloring $c$ is constant, and again we call this the {\em shift condition.} Let 
$$
    Z = [r]\ \mathrm{and}\ Y = [r+1, 4r]\, . 
$$
The shift condition implies that $Z$ is monochromatic. 

{\bf Step 2: triples $E$ with $|E \cap Z| = 1$ and $|E \cap S_i| \leq 1$ for any $i \in [r]$.} First we fix a vertex $v \in Z$ and handle the case where, for infinitely many $r$, there are two triples $E \subset [4r]$ with $v \in E$, $|E \cap Y| = 2$ and $|E \cap S_i| \leq 1$ for all $i \in [r]$ have different colors. So let $E = \{v, e_1, e_2\}$ and $F = \{v, f_1, f_2\}$ with $e_i,f_i\in Y$, $e_1 < e_2$, $f_1 < f_2$, and $E$, $F$ take from each $S_i$ at most one element, be such that $\chi_r(E) \ne \chi_r(F)$. One can take even such triples $E$ and $F$ that either 
\begin{description}
\item[(a)] $x := e_1 = f_1$ and $e_2<f_2$ (so $E=\{v<x<e_2\}$ and $F=\{v<x<f_2\}$), or
\item[(b)] $x := e_2 = f_2$ and $e_1<f_1$ (so $E=\{v<e_1<x\}$ and $F=\{v<f_1<x\}$).
\end{description}
Indeed, for $e_1=f_1$ we have case (a) (we swap $E$ and $F$ if needed). Suppose that $e_1 < f_1$ (we swap $E$ and $F$ if needed) and set $c = \chi_r(\{v, e_1, f_2\})$. If $c = \chi_r(E)$, we set $x = f_2$, get that $c = \chi_r(\{v, e_1, x\}) \ne \chi_r(\{v, f_1, x\})$, and have case (b). If $c \ne \chi_r(E)$, we set $x = e_1$, get that $\chi_r(\{v, x, e_2\}) \ne c = \chi_r(\{v, x, f_2\})$, and have case (a) (we swap $E$ and $F$ if needed). 

In both cases (a) and (b) we show that $|X_n| \geq F_n$ for every $n$. We consider case (b) in detail, and after that we discuss case (a) more briefly. By the shift condition, we may suppose that the four vertices $v, e_1, f_1, x$ lie in four different sets $S_i$, thus we set indices $j_0, j_1, j_2, j_3 \in [r]$ such that $v \in S_{j_0}$ (in fact, $v=j_0$), $e_1 \in S_{j_1}$, $f_1 \in S_{j_2}$, and $x \in S_{j_3}$. Since $e_1<f_1<x$, also $j_1 < j_2 < j_3$. If $j_0 \notin [j_1, j_2]$ (for infinitely many $r$) then the shift condition implies that, for any $r \in \mathbb{N}$, $K_r$ contains an $r$-wealthy coloring of type $W_1''$. For example, if $j_1<j_2<j_0<j_3$ then we may take $j_0=r-1,j_3=r$, keep $v$ and $x$ fixed, and replace $S_{j_1}$ and $S_{j_2}$ with all $S_{j_1'}$ and $S_{j_2'}$, respectively, for all $j_1'<j_2'<j_0$, the colors of the triples $\{v<e_1'<x\}$ and $\{v<f_1'<x\}$ then create the pattern of an $r$-wealthy coloring of type $W_1''$. Hence $|X_n| \geq 2^{n-2} \geq F_n$ for $n \geq 2$ by Lemma~\ref{LW1}.

We turn to the case when $j_0 \in [j_1,j_2]$, that is, $j_0 \in (j_1,j_2)$ (for infinitely many $r$). We show that either for any $r \in \mathbb{N}$ the ideal $X$ contains an $r$-wealthy coloring of type $W_1''$, or for any $r \in \mathbb{N}$ the ideal $X$ contains an $r$-wealthy coloring of type $W_2$ (obtained by symmetries from the $W_2''$ coloring). Clearly we may assume that $\chi_r(E) = 0$, $\chi_r(F) = 1$, and that $x$ is the least element of the triple $S_j \cap Y$ in which it lies: $x = r+3j_3-2$. If $e_1$ and $f_1$ do not share the same order (as the 1st, 2nd, or 3rd element) in the triple $S_j \cap Y$ they lie in, it follows by the shift condition that $W_1''$ coloring appears and we have the lower bound from the previous paragraph. Thus we may assume that also $e_1$ and $f_1$ are the least elements of the triples: $e_1 = r+3j_1-2$ and $f_1 = r+3j_2-2$.
We define the partition matrix $M_r(i,j)$ of $K_r$, $i,j\in[r-1]$, by
$$
    M_r(i,\, j) = \chi(\{i,\, r+3j-2,\, 4r-2\})\ (=\chi(E),\, \chi(F))\, .
$$
By our assumption on the colors of $E$ and $F$, the matrix $M_r$ has $0$s below the main diagonal and $1$s above it. Also, by the shift condition, the main diagonal of $M_r$ is monochromatic. One may see the example of an upper diagonal matrix $U_r = M_r$ in Figure~\ref{fig:naughty_2}(a). We see that the restriction and normalization of $K_r$ on the vertex set $ S = [r-1]\cup\{r+3j-2\;|\;j\in[r]\}$ is an $(r-1)$-wealthy coloring of type $W_2$: if $M_r$ has $1$s on the main diagonal, we have directly $W_2''$, and if it has $0$s on the main diagonal, transition to and swapping both colors yields $W_2''$. 
Hence $X$ contains an $r$-wealthy coloring of type $W_2$ for every $r\in\mathbb{N}$ and by Proposition~\ref{LW2}, $|X_n|\ge F_n$ for every $n\in\mathbb{N}$.

In the case (a) we proceed similarly to the case (b). Now $E=\{v<x<e_2\}$, $F=\{v<x<f_2\}$, $v\in S_{j_0}$, $x\in S_{j_1}$, $e_2\in S_{j_2}$, and $f_2\in S_{j_3}$ for four distinct indices $j_i\in[r]$ with $j_1<j_2<j_3$. The case when $j_0\not\in [j_2, j_3]$ leads to $W_1'$ colorings, and the case $j_0\in [j_2, j_3]$ leads to $W_1'$ colorings or to $W_2$ colorings (obtained by symmetries from the $W_2''$ coloring).

The last part of step 2 is to consider triples $E$ such that $|E \cap Z| = 1$ and $|E \cap S_i| \leq 1$ for any $i \in [r]$, but they do not share the vertex of $Z$. However, by the shift condition it follows that there are two of the described triples $E$, $E'$ such that $\chi_r(E)\ne\chi_r(E')$, $E \cap Z\ne E' \cap Z$, and $E\cap Y=E'\cap Y$. Hence $X$ contains for every $r$ an $r$-wealthy coloring of type $W_1'$ and we are done by Lemma \ref{LW1}.

So all the described triples have the same color which we call $s$, and by $t\not= s$ we denote the other color.


{\bf Step 3: triples $E$ that for two distinct $i,j \in [r]$ satisfy $|E \cap S_i| = 2$ and $|E \cap S_j| = 1$.} 

First we show that all triples $E$ that for two distinct $i,j \in [r]$ satisfy $|E \cap Y \cap S_i| = 2$ and $|E \cap Z \cap S_j| = 1$ have the same color $\chi_r(E) = s$ and defer the case $i \in E$, $|E \cap Y \cap S_i| = 1$ (and $|E \cap Y \cap S_j| = 1$) to the end of this step. 
Suppose not: there is a triple $E = \{v, e_1, e_2\}$ such that for two different indices $i,j\in[r]$ one has $v \in Z \cap S_j$, $e_1, e_2 \in Y \cap S_i$, and $\chi_r(E) = t$. Without loss of generality, $j<i$. Using the shift condition we may assume that $j = v = 1$ and $i = 2$, and further we may take $e_1 = r+4$ and $e_2 = r+5$. We show that for infinitely many, and hence for every, $r\geq 3$ and every binary string $w = w_1w_2 \dots w_{r-2}$ avoiding consecutive substring $tt$ there exists a coloring $K_w=(r,\chi_w)\in X$ such that $\chi_w(\{1, i, i+1 \}) = w_{i-1}$ for $i = 2, 3, \dots, r-1$. Indeed, we set 
$$
C = \{3i+r+4 \mid w_i = s\}, \quad D = \{3i+r+2 \mid w_i = t\} \quad \mbox{and} \quad S=\{1,\, r+4\}\cup C \cup D\, ,
$$
and consider restriction and normalization of $K_r = (4r, \chi_r)$ to the set $S$. Since the triples $E = \{v_1, v_2, v_3\} \subset S$ with fixed first vertex $v_1 = 1$ and second vertex $v_2$ successing by third vertex $v_3$ satisfy $\chi_r(E) = 1$ if and only if $v_2 \in C$ and $v_3 \in D$, it is easy to see that the result is $K_w$. By part 1 of Lemma~\ref{LFibG} we have $|X_n| \geq F_n$ for every $n$.

\begin{figure}
    \centering
    \subfloat[Step 2. Here, $r=4$, partition matrix $M_r(i, j)$ is the upper triangular matrix with ones on the diagonal and above and the set $S = \{1, 2, 3, 5, 8, 11, 14\}$. ]{{\includegraphics[width=0.36\textwidth]{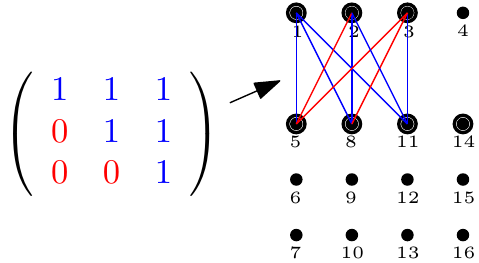} }}%
    \qquad
    \subfloat[Step 3. Containment $K_w\preceq K_r$ for $w = 01010001010$, $r=13$, 
    $C = \{1, 3, 5, 6, 7, 9, 11\}$, $C'=(3C+r+4)\cup\{r+4\}$, $D = \{2, 4, 8, 10\}$, $D'=3D+r+2$ and $S = \{1, 17, 20, 21, 26, 27, 32, 35, 38, 39, 44, 45, 50\}$. ]{{\includegraphics[width=0.57\textwidth]{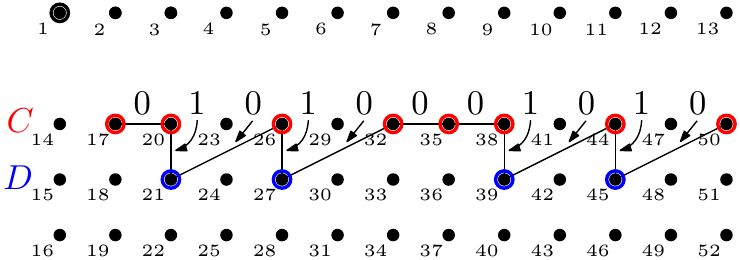} }}%
    \caption{Two examples of the restrictions of the original wealthy coloring $K_r$of type $W_{4, 2}$ to the set $S$ (marked by bigger dots).}%
    \label{fig:naughty_2}%
\end{figure}

An example of the coloring $K_w$ for $r=13$, $w = 01010001010$, and $t=1$ is given in Figure~\ref{fig:naughty_2}~(b). The elements $j\in[r]$ are in the first row, the triples $Y\cap S_j=\{r+3j-2,r+3j-1,r+3j\}$ are placed vertically below them, the elements of $S$ are circled and the mentioned triples of $K_w$
consist of $1$ and two elements from $C\cup D$ that are joined in the picture by segments labeled $0$ or $1$, according to  the $\chi_w$-color of the triple $E$. 

Now we show that $\chi_r(E) = s$ for all triples $E$ with $|E \cap Z| = 1, |E \cap Y| = 2$ that for some $i\in [r]$ satisfy $i\in E$ and $|E\cap S_i|=2$. For contrary, let $\chi_r(E) = t$ for an triple $E = \{i, e_1, e_2\}$ of this form, where we additionally may suppose that $e_1 = r+3i-2$, $e_2 = r+3j-2$, and $j>i$ (the case $j<i$ is treated similarly). It is clear that for
$r\geq 1$ the restriction and normalization of $K_r$ to the set $[r-1]\cup(3[r-1]-2)\cup\{4r-2\}$ is an $(r-1)$-wealthy coloring of type $W_2'$. By Proposition~\ref{LW2}, $|X_n| \geq F_n$ for every $n$.

{\bf Step 4: Conclusion. }Thus we may assume, due to the shift condition and due to the previous steps, that in the coloring $K_r$ all triples 
$E$ with $|E\cap Z| = 1$, $|E\cap Y| = 2$, and $|E\cap S_i|\leq 2$ for every $i \in [r]$ have the same color $s$, say $\chi_r(E) = s = 1$. Since $K_r$ is an $r$-wealthy coloring 
of type $W_{4, 2}$, for every $i\in [r]$ there is an triple $F\subset S_i$ such that $i\in F$ and $\chi_r(F) = 0$. Without loss of generality we suppose that $F=\{i, r+3i-2, r+3i-1\}$.

We finally show that $|X_n| \geq F_n$ for $n\ge n_0$. Let $n = 5m + \varepsilon$ where $\varepsilon \in \{0, 1, 2, 3, 4\}$ and
$$
    A = \{a_1 < a_2 < \dots < a_m\}\subset[2m + \varepsilon]\;\text{ and } B = \{b_1 < b_2 < \dots < b_m\}\subset[2m]
$$ 
be two $m$-element sets. A coloring $K = (n,\chi)$ is {\it $(A, B)$-disobedient} if $\chi(E) = 1$ for all triples $E$ such that 
$|E \cap [2m+\varepsilon]| = 1$ and $|E \cap[2m+\varepsilon+1, 5m+\varepsilon]| = 2$, except for the triples
$$
    F_i = \{a_i,\, 2m+\varepsilon+b_i+i-1,\, 2m+\varepsilon+b_i+i\},\ i \in [m],\;\text{ with }\;\chi(F_i) = 0\, .
$$
To mark $K$ as $(A,B)$-disobedient we write $K = K_{A,B}$ (more precisely it means that $K \in K_{A,B}$). Clearly, for different pairs $(A, B) \ne (A', B')$ of $m$-element subsets of $[2m]$ we always have $K_{A, B} \ne K_{A', B'}$ (more precisely, $K_{A, B} \cap K_{A', B'} = \emptyset$). Let $r=3m+\varepsilon$. We prove that for every pair $(A,B)$ of two $m$-element sets as displayed above the coloring $K_r = K_{3m+\varepsilon} = (12m + 4 \varepsilon, \chi_r)$ contains an $(A, B)$-disobedient coloring $K_{A, B}$. 

For it we define $a_0 = b_0 = 0$, $a_{m+1} = 2m + \varepsilon + 1$ and $b_{m+1} = 2m+1$, and for $i=0,1,\dots, m$ the quantities $\alpha_i = a_{i+1} - a_i - 1$ and $\beta_i = b_{i+1} - b_i -1 $ which count the elements of the ground sets $[2m+\varepsilon]$, resp. $[2m]$, in the gaps determined by the elements in $A$ and in $B$. Further, let $t_0 = 0$ and 
$$
    t_i = i + \sum_{j=0}^{i-1} (\alpha_j + \beta_j) = a_i + b_i - i
$$ 
for $i \in [m]$. For $i=0,1,\dots, m$ we define the gap sets $C_i = \{t_i+1, \dots, t_i+\alpha_i\}$ and $D_i = \{t_i+\alpha_i+1, \dots, t_i+\alpha_i+\beta_i\}$. Clearly, $|C_i| = \alpha_i$, $|D_i| = \beta_i$, and 
\begin{align} \label{eq:disj_part_[r]}
    C_0 < D_0 < \{t_1\} < C_1 < D_1 < \{t_2\} < C_2 < D_2 < \dots < \{t_m\} < C_m < D_m 
\end{align}
is a partition of $[r] = [3m + \varepsilon]$ into $3m+2$ (possibly empty) sets. 
Finally, we label some subsets of $S_i = \{i,r+3i-2,r+3i-1,r+3i\}$: $S_i^1 = \{i\}$, $S_i^2 = \{r+3i-2\}$, $S_i^{23} = \{r+3i-2,r+3i-1\}$, $S_i^{123} = S_i^{1} \cup S_i^{23}$, and set
$$
	S =  \bigcup_{i=1}^m S_{t_i}^{123} \cup \bigcup_{i=0}^m \left(\bigcup_{j \in C_i} S_j^{1} \cup \bigcup_{j \in D_i} S_j^{2}\right) \subset [12m + 4\varepsilon]=[4r]\, .
$$
The partition (\ref{eq:disj_part_[r]}) of $[3m+\varepsilon]$ implies that $|S| = 5m + \varepsilon=n$. We show that restriction and normalization of $K_r = (4r, \chi_r)$ to the set $S$ defines an $(A,B)$-disobedient coloring. Let $f : S\to[n]$ be the increasing bijection. First, $f(t_i) = a_i$, that is, $f(S_{t_i}^{1}) = \{a_i\}$, 
because  (note that $|S_j^x|$ equals to the length of the word $x$) $t_i$ has in $S$ exactly
$$
    \sum_{l=1}^{i-1}|S_{t_l}^1| + \sum_{l=0}^{i-1} \sum_{j \in C_l}|S_j^1| = i-1 + \sum_{l=0}^{i-1} |C_l| = i-1 + (a_i - a_0 - i) = a_i - 1
$$
predecessors. Also, 
\begin{align} \label{eq:map_of_Z}
   f(Z\cap S)=f([r]\cap S)= f\left( \bigcup_{i=1}^m S_{t_i}^1 \cup \bigcup_{i=0}^m \bigcup_{j \in C_i} S_j^1 \right) = [2m + \varepsilon]\, .
\end{align}
because the set inside the large brackets has $m+\sum_{i=0}^m |C_i|=2m+\varepsilon$ elements. Similarly,  
$f(S_{t_i}^{23}) = \{2m + \varepsilon + b_i + i - 1, 2m + \varepsilon + b_i + i\}$ because $S_{t_i}^{23}$ has in $Y = [r+1, 4r]\cap S$ exactly
$$
  \sum_{l=1}^{i-1}|S_{t_l}^{23}| + \sum_{l=0}^{i-1} \sum_{j \in D_l}|S_j^2| = 2(i-1) + \sum_{l=0}^{i-1} |D_l| = 2(i-1) + b_i - b_0 - i = b_i + i - 2
$$
predecessors. Hence $f(S_{t_i}(123))=F_i$, the $i$-th triple of an $(A, B)$-disobedient coloring. In view of all of this (especially recall the form of the coloring $\chi_r$ mentioned at the beginning of 
Step 4), it follows that the restriction and normalization of $K_r=(4r,\chi_r)$ to $S$ (which is the coloring $(n,\chi_r\circ f^{-1})$ where for $f^{-1}$ we abuse notation a little in the obvious way)
is an $(A, B)$-disobedient coloring $K_{A, B}$---see Figure \ref{fig:naughty_4} for a concrete example (like in Figure \ref{fig:naughty_2}, 
the quadruples $S_i$ are visualized by columns of four dots, and their first ``elements'' $S_i^1$ forming $Y=[r]$ lie in the topmost row).

\begin{figure}
    \centering
    \includegraphics[width=0.9\textwidth]{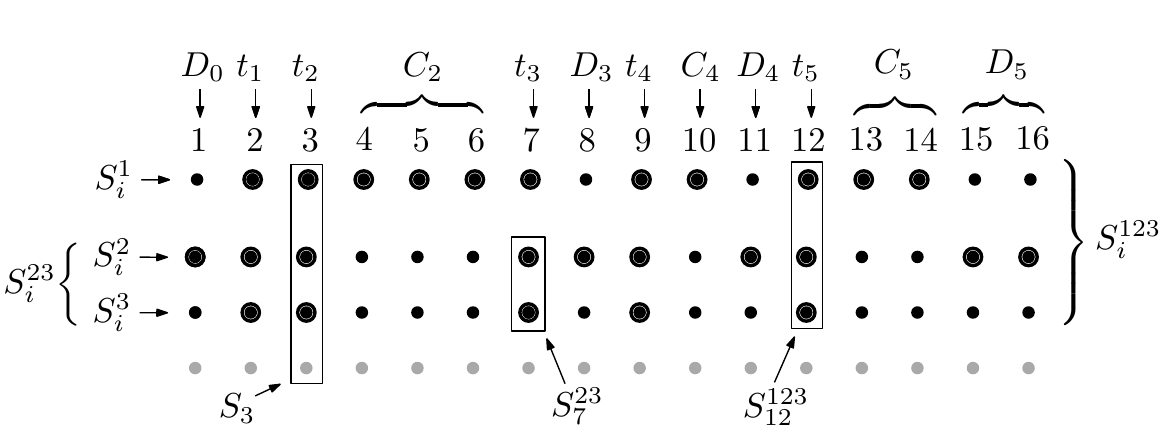}
    \caption{\label{fig:naughty_4} Step 4. An example of containment $K_{A, B} = (26, \chi)\preceq K_r = (64, \chi)$ where the set $S$
    is marked by bigger dots. Here $n=26$, $r=16$, $A = \{1, 2, 6, 7, 9\}$, and $B = \{2, 3, 4, 6, 8\}$.}
\end{figure}

Different pairs $A,B$ of $m$-element subsets $A\subset[2m+\varepsilon]$ and $B\subset[2m]$ give distinct $(A, B)$-disobedient coloring $K_{A, B}$, therefore for every $m\ge1$ we have
$$
	|X_n| \geq \binom{2m+\varepsilon}{m}\binom{2m}{m}\ge\binom{2m}{m}^2\;\text{ where }\;n=5m+\varepsilon\le5m+4
$$
and the value $m = (n-4)/5$ completes the proof of the first part of the statement.

It remains to find an $m_0$, resp. $m_1$, such that if $m\ge m_0$, resp. $m \ge m_1$ then 
$$
    |X_n|\ge\binom{2m}{m}^2\ge G_{5m+4}\ge G_n\,, \; 
    \mbox{ resp. }\; 
    |X_n|\ge\binom{2m}{m}^2\ge F_{5m+4}\ge F_n\, .
$$
To find $m_0$, we again use the bounds on middle binomial coefficient of N.\,D. Kazarinoff \cite{kaza}. Thus, for every $m\in\mathbb{N}$, 
$$
    \binom{2m}{m}^2\ge\frac{16^m}{4m}\;\text{ and }\; 0.418 \cdot 1.466^m>G_m\, .
$$
Since $\frac{1}{4 \cdot 0.418 \cdot 1.466^4}>0.129$ and $\frac{16}{1.466^5}>2.362$, we need an $m_0$ with $m\ge m_0$ implying $0.129\cdot2.362^m\ge m$. It is easy to check that this holds with $m_0=4$. Thus in Step 4, $|X_n|\ge G_n$ holds for every $n\ge5\cdot4=20$. 

To find $m_1$ we proceed similarly with
$$
    |X_n| \geq \binom{2m}{m}^2 \geq \frac{16^m}{4m} 
    > 3.064 \cdot 11.090^m \geq F_{5m+4} \geq F_n
$$
for any $m \geq 15$, thus $n \geq 75$. That concludes the second part od the statement and the proof of Proposition~\ref{LW42} is complete.
\end{proof}


\section{Crossing matrices}

This section contains the second third of the proof of Theorem~\ref{T1FD}. First we review some results on two-dimensional matrices. Then we introduce 
three-dimensional ``crossing'' matrices with entries in $\{0,1,*\}$ derived from colorings, and reduce by Lemmas~\ref{3DKlazar}, \ref{3DKlazar2} and \ref{lemma_on_diagonals} 
problems on three-dimensional matrices to two-dimensional situation. Finally we obtain by means of crossing matrices a polynomial upper bound on the number 
of ``$p$-tame'' colorings. We call a two- or a three-dimensional matrix simply a matrix when the dimension is clear from the context. 

\subsection{Results on two-dimensional matrices}\label{subsec_review2dim}

We review some results in the article \cite{Klazar08} by the second author in order that we can use them later; they also inspired this article. Then we give and prove one new
result on two-dimensional matrices that we also use later.
For $r,s\in\mathbb{N}$ let
$$
N\colon[r]\times[s]\to\{0,\,1\}
$$ 
be an $r\times s$ binary matrix. Every row and column of $N$ consists of alternating intervals of zeros and ones, and we denote by $\al(N)$ the maximum 
number of these intervals in a row or a column of $N$. Equivalently, $\al(N)$ is one plus the maximum number of subwords $01$ and $10$ in a row or a column of $N$. 
For $j\in [s]$ we let $C(N,j)\subset [r]$ denote the set of the row indices of the largest entries of these intervals 
in column $j$, with $r$ omitted: $i\in C(N,j)$ if and only if $N(i,j)\ne N(i+1,j)$. We set $C(N) = \bigcup_{j=1}^{s} C(N,j)$. Similarly, for $i \in [r]$ we let 
$R(N,i)\subset [s]$ be the column indices $j$ such that $N(i,j)\ne N(i,j+1)$ and $R(N) = \bigcup_{i=1}^{r} R(N,i)$. Unlike $|C(\cdot)|$, the quantity $|R(\cdot)|$ 
is in \cite{Klazar08} defined only implicitly 
as the number $a$ in \cite[Lemma 3.12]{Klazar08}, which is the next lemma. Below in Proposition~\ref{new_variation_onKlazar} and in the proof of Proposition~\ref{P2} 
we consider more general matrices $N$ with entries $0$, $1$ and $*$. Then, as before,
$\al(N)$ is one plus the maximum number of subwords $01$ and $10$ in a row or a column of $N$, $R(N,i)$ is the set of $j$ such that 
$\{N(i,j),N(i,j+1)\}=\{0,1\}$, $R(N) = \bigcup_{i=1}^{r} R(N,i)$ and also $C(N,j)$ and $C(N)$ are defined as before. The following paraphrases a lemma in \cite{Klazar08}.

\begin{lemma}[Lemma 3.12 in \cite{Klazar08}]\label{KlazarL3.12}
For every binary matrix $N$,
$$
|R(N)|\leq(\al(N)-1)(2|C(N)|+1).
$$
\end{lemma}

\noindent
By transposing the matrix we get the same bound with exchanged $R(N)$ and $C(N)$. Thus if a sequence $(\al(N_n))_{n\geq 1}$, where $N_n$ are binary matrices, 
is bounded then either both sequences $(R(N_n)|)_{n \geq 1}$ and $(|C(N_n)|)_{n \geq 1}$ are bounded, or both are unbounded. 

Let $I_r$ be the $r\times r$ identity matrix with 1's on the main diagonal and 0's elsewhere and $U_r$ be the $r\times r$ upper diagonal matrix with 1's 
on the main diagonal and above it and 0's below it. Two square matrices $N'$ and $N$ are {\it similar} if $N=N'$ or one arises from the other by vertical and/or horizontal 
flip and/or by exchanging $0$ and $1$. This matrix symmetries result from the symmetries of colorings appearing in Section~\ref{sec_wealthy} in the definitions of colorings of types $W_2$ and $W_3$. 
If the horizontal flip is not allowed, we say that $N'$ and $N$ are {\it strongly similar}. Strong similarity implies similarity. There are four matrices that are both similar and strongly similar to $I_r$, 
four matrices strongly similar to $U_r$ and eight matrices similar to $U_r$. For two matrices $N$ and $N'$ we say that  {\it $N'$ is contained 
in $N$}, and write $N'\preceq N$, if $N'$ arises from $N$ by deleting some rows and some columns. In other words, if 
$$
N'\colon[r']\times[s']\to\{0,\,1,\,*\}\;\text{ and }\;N\colon[r]\times[s]\to\{0,\,1,\,*\}
$$ 
then for some increasing injections $f\colon[r']\to[r]$ and $g\colon[s']\to[s]$ we have for every $i\in[r']$ and $j\in[s']$ that
$$
N'(i,\,j)=N(f(i),\,g(j)).
$$
We then also say that $N'$ is a {\it submatrix} of $N$. 
Another result from \cite{Klazar08} used in the end of the proof of Theorem~\ref{T1FD} is the next one.

\begin{lemma}[Lemma 3.13 in \cite{Klazar08}]\label{KlazarL3.13}
Let $(N_n)_{n \geq 1}$ be an infinite sequence of binary matrices such that the sequence $(\al(N_n))_{n \geq 1}$ is bounded but the sequence $(|C(N_n)|)_{n \geq 1}$ 
is unbounded. Then one of the following holds.
\begin{enumerate}
\item[(i)] For every $r\in\mathbb{N}$ there is an $n$ such that $N_n$ contains a matrix $I_r'$ strongly similar to $I_r$.
\item[(ii)] For every $r\in\mathbb{N}$ there is an $n$ such that $N_n$ contains a matrix $U_r'$ strongly similar to $U_r$.
\end{enumerate}
\end{lemma}
By transposing the matrices $N_n$ or by applying the remark after Lemma~\ref{KlazarL3.12} we deduce that this lemma holds also when $(|C(N_n)|)_{n \geq 1}$ is replaced with 
$(|R(N_n)|)_{n \geq 1}$. In Lemma~\ref{3DKlazar} we lift Lemma~\ref{KlazarL3.13} to three-dimensional matrices.

\begin{remark}\label{rema_on_symmetry}
In \cite[bottom of p. 15]{Klazar08} strong similarity is called similarity. Our relation of similarity is not used in \cite{Klazar08}.
\end{remark}

In the end of the proof of Theorem~\ref{T1FD} we need the next result which does not appear in \cite{Klazar08} and which 
we prove here. We recall further terminology from \cite{Klazar08}. A coloring $\chi=([3r],\chi)$, $r\in\mathbb{N}$, of pairs (i.e. $\chi\colon\binom{[3r]}{2}\to\{0,1\}$) 
is \emph{$r$-wealthy of type $2$} if none of the $r$ triples $\{3i-2,3i-1,3i\}$, $i\in[r]$, is $\chi$-monochromatic.

\begin{proposition}\label{new_variation_onKlazar}
Let $(N_n)_{n\geq 1}$ be a sequence of symmetric square matrices with $*$s on the main diagonal and $0$s and $1$s outside it and such that the sequence 
$(\al(N_n))_{n \geq 1}$ is bounded but the sequence 
$$
(|R(N_n)|)_{n\ge1}=(|C(N_n)|)_{n\ge1}
$$ 
is unbounded. Each matrix $N_n$ defines a binary coloring of pairs $K_n=([r(n)],\psi_n)$, where $r(n)$ is the number of rows or columns in $N_n$, by 
$$
\psi_n(\{i,\,j\})=N_n(i,\,j)=N_n(j,\,i).
$$
Then ($\alpha$) or ($\beta$) holds.
\begin{enumerate}
    \item[($\alpha$)] For every $r\in\mathbb{N}$ there is an $n$ such that $K_n$ contains an $r$-wealthy coloring of type $2$. 
    \item[($\beta$)] For every $r\in\mathbb{N}$ there is an $n$ such that $N_n$ contains a binary matrix $I_r'$ strongly similar to $I_r$, or for every 
    $r\in\mathbb{N}$ there is an $n$ such that $N_n$ contains a binary matrix $U_r'$ strongly similar to $U_r$. Moreover, for every $r$ the matrices $I_r'$ and $U_r'$ lie above the diagonal of the matrix $N_{n(r)}$.
\end{enumerate}
\end{proposition}
\begin{proof}
For a given $K_n$ we consider (as in \cite{Klazar08}) the \emph{interval decomposition of $K_n$} which is the interval partition 
$I_{n,1}<I_{n,2}<\dots<I_{n,k(n)}$ 
of $[r(n)]$ such that $I_{n,1}$ is the longest initial interval in 
$[r(n)]$ with all pairs in $\binom{I_{n,1}}{2}$ having the same color in $\psi_n$, $I_{n,2}$ is the longest monochromatic interval following after  $I_{n,1}$, and so on.
In Section~\ref{subsec_ptame} we use the same decomposition (called there nuclear decomposition) for colorings of triples.
Clearly, for every $i<k(n)$ we have $|I_{n,i}|\ge 2$ and every interval $I_{n,i}$ with $i<k(n)$ contains elements $a=a_{n,i},b=b_{n,i}$ such that $a\ne b$ and 
$\psi_n(\{a,b\})\ne\psi_n(\{b,\min(I_{n,i+1})\}$. There are two cases, either the sequence $(k(n))_{n\ge1}$ of lengths of the interval decompositions is unbounded or it is
bounded. 

In the former case, when the quantity $k(n)$ attains arbitrarily large values, we consider for $i=1,2,\dots,\lfloor k(n)/2\rfloor$ the triples 
$T_{n,i}=\{a_{n,2i-1},b_{n,2i-1},\min(I_{n,2i})\}$. These triples are disjoint, in fact
$$
T_{n,1}<T_{n,2}<\dots<T_{n,\lfloor k(n)/2\rfloor},
$$
and each triple $T_{n,i}$ is non-monochromatic in $\psi_n$. Thus we get case ($\alpha$) of the present proposition.

The latter case is that the quantity $k(n)$ as a function of $n$ is bounded, say $k(n)\le c$ for every $n$ and a constant $c\in\mathbb{N}$. 
Then we consider for every pair of intervals $I=I_{n,i},J=I_{n,j}$ with $1\le i,j\le k(n)$ the submatrix $M_{I,J}$ of $N_n$ formed by the positions in 
the rows with indices in $I$ and in the columns with indices in $J$. Note that every diagonal matrix $M_{I,I}$ consists of $*$s on the main diagonal and of only $0$s or 
only $1$s elsewhere because every interval $I_{n,i}$ is $\psi_n$-monochromatic. Also, $M_{I,J}=M_{J,I}^T$ and if $i<j$ then $M_{I,J}$ lies above the diagonal of $N$. 

We claim that for any $n \in \mathbb{N}$ there exist indices $i_n,j_n\in[k(n)]$ with $i_n<j_n$
and such that for $I_n=I_{n,i_n},J_n=I_{n,j_n}$ the sequence $(|R(M_{I_n,J_n})|)_{n\ge 1}$ is unbounded. To see it, recall that $R(N_n)$ is the set of the column indices
$j\in[k(n)]$ such that $\{N_n(i,j),N_n(i,j+1)\}=\{0,1\}$ for some row index $i=i_j\in[k(n)]$. The column indices $j\in R(N_n)$ are of two kinds. The first $j$s
are such that the two positions $(i_j,j)$ and $(i_j,j+1)$ lie in one matrix $M_{I,J}$; both $j$ and $j+1$ lie in one interval $J$ of the interval decomposition of 
$K_n$. The second $j$s are the remaining ones when $j$ and $j+1$ lie in two consecutive intervals of the 
interval decomposition of $K_n$. So if a $j$ is of the second kind then $j = \max(J)$ for an interval $J$ of the interval decomposition of $K_n$ and there are at
most $c$ of them. Also, we already noted that for no $j$ of the first kind the two positions $(i_j,j)$ and $(i_j,j+1)$ lie in a diagonal matrix $M_{I,I}$ and we know
there are at most $c^2$ matrices $M_{I,J}$. Thus there exist intervals $I_n<J_n$ such that
$$
|R(M_{I_n,J_n})|\ge\frac{|R(N_n)|-c}{c^2}
$$
and the sequence $(|R(M_{I_n,J_n})|)_{n\ge 1}$ is unbounded. But since the 
sequence $(\al(N_n))_{n \geq 1}$ is bounded, $(\al(M_{I_n,J_n})))_{n \geq 1}$ is bounded too. We may therefore apply Lemma~\ref{KlazarL3.13} to the sequence of binary matrices 
$(M_{I_n,J_n})_{n \geq 1}$ and get case ($\beta$) of the present proposition.
\end{proof}

\subsection{Three-dimensional matrices}\label{subs_def_al}

A (three-dimensional) matrix $M$ is a map $M\colon[r]\times[s]\times[t]\to\{0,1,*\}$, $r,s,t\in\mathbb{N}$, we also say that $M$ is {\it $*$-binary}. We visualize three-dimensional matrices $M$ as cubes in $\mathbb{R}^3$, with edges parallel to the coordinate axes and such that the origin in $M$ is the front top left corner, the first coordinate increases in the left-to-right direction, the second one in the top-to-bottom direction, and the third one in the front-to-back direction, as shown in Figure~\ref{fig:Cuttingmatrices}. We will work with special $*$-binary matrices derived from colorings which have entry $*$ when two of the three elements in a triple in the binary coloring  coincide, as given in the next formal definition.

Let $H = (n,\chi)$ be a coloring of triples and 
$$
X = \{x_1 < \dots < x_r\},\;Y = \{y_1 < \dots < y_s\}\;\text{ and }\;Z = \{z_1 < \dots < z_t\}
$$ 
be nonempty subsets of $[n]$. We call the $r\times s\times t$ matrix $M\colon[r]\times[s]\times[t]\to\{0,1,*\}$, given as
\begin{equation*}
M(i, j, k) = \begin{cases}
               \chi(\{x_i,\, y_j,\, z_k\})   & \text{if $x_i$, $y_j$ and
              $z_k$ are three distinct elements and} \\
               $*$					   & \text{else (when $|\{x_i, y_j, z_k\}|\le2$)}\;,
             \end{cases}
\end{equation*}
the {\it crossing matrix of $H$} and denote it $M_{X,Y,Z}$. The sets $X, Y$ and $Z$ are the {\it base sets} of the crossing matrix $M$. 
Clearly, a crossing matrix is binary if and only if its base sets are disjoint.

Let $M$ be an $r \times s \times t$ $*$-binary matrix. For fixed $J \in [s]$ and $K \in [t]$, the finite sequence 
$$
\overline{r}(J,\,K) := (M(i,\,J,\,K))_{i=1}^r\in\{0,\,1,\,*\}^r
$$ 
is a {\it row} of $M$. In some situations we understand under a row also the set 
$$
\overline{r}(J,\,K)=\{(i,\,J,\,K):\ i\in[r]\}
$$ 
of its positions in $M$.
We define {\it columns} $\overline{c}(I,K)$ and {\it shafts} $\overline{s}(I,J)$ in $M$ similarly by fixing the first and third, respectively the first and second, coordinate of $M$. Rows, columns, and shafts are the {\em lines} of $M$.

A row in a two-dimensional matrix has the first coordinate fixed and the second one variable. However, a row in a three-dimensional matrix here has the first coordinate variable and the second and third one fixed. Similarly for columns in which case the second coordinate is variable, while the first and third one are fixed. This discrepancy in terminology for two- and three-dimensional matrices forces us later to exchange coordinates in layer and cross-matrices (defined below). A three-dimensional $r\times s\times t$ matrix has $st$ rows, $rt$ columns and $rs$ shafts, but they cannot be ordered in a natural way as rows and columns in two-dimensional matrices.

Let $\al(M)$ be one plus the maximum number of the pairs $a_ia_{i+1} \in \{01,10\}$ in a line of $M$, taken over all $st + rt + rs$ lines of $M$. As for two-dimensional matrices, for any row given by the coordinates $(J,K)$ we define $R(M;J,K) \subset [r]$ to be the indices $i$ such that $M(i,J,K) M(i+1,J,K) \in \{01, 10\}$. For example, for the row 
$$
    \overline{r}(J,\,K) = 00011**11*010 
    \qquad \mbox{we have} \qquad
    R(M;J,K) = \{3,11,12\}\, .
$$
We set $R(M) = \bigcup_{J,K} R(M;J,K)$, with the union over all rows of $M$. So $R(M) \subset [r]$ contains exactly the indices $i$ such that the two $(j,k)$-entries in the two-dimensional matrices $M(i,j,k)$ and $M(i+1,j,k)$ are $0$ and $1$ or vice versa.  

Similarly we define $C(M;I,K) \subset [s]$, $S(M;I,J) \subset [t]$, $C(M) = \bigcup_{I,K} C(M;I,K)$, and $S(M) = \bigcup_{I,J} S(M;I,J)$. In the case of a row given by a pair $\overline{r} = (J,K)$ we also write $R(M;\overline{r})$ instead of $R(M;J,K)$, and similarly we write $C(M;\overline{c})$ and $S(M;\overline{s})$. Clearly, $\al(M)$ is the maximum of the numbers $|R(M;\overline{r})|$, $|C(M;\overline{c})|$ and $|S(M;\overline{s})|$, when $\overline{r}$ runs through all $st$ rows, $\overline{c}$ through all $rt$ columns  and $\overline{s}$ through all $rs$ shafts of $M$. Note that if $M$ is a crossing matrix of a coloring then any line of $M$ either contains at most two stars or consists only of stars. We prove an analogue of Lemma~\ref{KlazarL3.12}.

\begin{lemma}\label{rema_RCS}
For every $r \times s \times t$ binary matrix $M$, if $\al(M) \leq l$ and both $|R(M)|$ and $|C(M)|$ are bounded by $m$ then
$$
    |S(M)| \leq (l-1)(m+1)^2\, .
$$
Moreover, two symmetric variants hold: $|S(M)|$ may be switched with $|R(M)|$ or with $|C(M)|$.
\end{lemma}
\begin{proof}
We assume that $\al(M)\le l$, $|R(M)|\le m$ and $|C(M)|\le m$. The other two symmetric variants are treated in the same way and we omit their proof. We first bound the cardinality of the set
$D$ consisting of the pairs $(x,y)\in[r]\times[s]$ such that $M(x,y,k)\ne M(x,y,k+1)$ for some $k\in[t-1]$ but $M(x',y',k)=M(x',y',k+1)$ whenever $x'+y'<x+y$. Let $(x,y)\in D$ with $x,y>1$. Then
$$
    M(x,\,y,\,k)\ne M(x,\,y,\,k+1)\;\text{ but }\; M(x-1,\,y,\,k)=M(x-1,\,y,\,k+1)
$$ 
and $x-1 \in R(M)$. Also  $y-1\in C(M)$, by a symmetric argument. We have $x,y>1$ and the map
$$
    D \rightarrow R(M)\times C(M) 
    \quad \mbox{such that} \quad
    (x,\,y) \mapsto (x-1,\,y-1) 
$$
is an injection, hence
$D$ has at most $m^2$ such elements $(x,y)$. By a similar argument $D$ has at most $2m$ elements of the form $(1,y)$ and $(x,1)$ with $x,y>1$. Thus $|D|\le m^2+2m+1=(m+1)^2$. 
We consider the map 
$$
    S(M) \rightarrow D
    \quad \mbox{such that} \quad
    k \mapsto (x,\,y)
    \;\text{ (with $M(x,\,y,\,k)\ne M(x,\,y,\,k+1)$)}\, .
$$
If $|S(M)| > (l-1)(m+ 1)^2$ then there exists a pair $(x,y)\in D$ and $l$ shaft indices $k_i$ with $1 \leq k_1<k_2< \dots <k_l < t$ 
such that $M(x,y,k_i) \not= M(x,y,k_i+1)$ for 
every $i \in [l]$. The shaft $\overline{s}(x,y)$ has at least $l+1$ alternating intervals of zeros and ones, in contradiction with the bound $\al(M) \leq l$.
\end{proof}

\noindent
Thus in the definition of $p$-tame colorings in Section~\ref{subsec_ptame} it does not matter which two of the three quantities $R(M)$, $C(M)$ and $S(M)$ 
we chose; we select $R(M)$ and $C(M)$.

For two $*$-binary matrices 
$$
    M\colon[r]\times[s]\times[t]\to\{0,\,1,\,*\}\;\text{ and }\;M'\colon[r']\times[s']\times[t']\to\{0,\,1,\,*\}
$$ 
we say that $M'$ is a {\em submatrix} of $M$, and write $M'\preceq M$, if there are increasing injections 
$f\colon[r']\to[r]$,  $g\colon[s']\to[s]$ and $h\colon[t']\to[t]$ such that for every $(I,J,K)\in[r']\times[s']\times[t']$ one has $M'(I,J,K)=M(f(I),g(J),h(K))$. 
We also say that $M'$ is \emph{contained in $M$}. Easily, $M''\preceq M'$ and $M'\preceq M$ imply $M''\preceq M$.  It is clear that if 
$M'\preceq M$ and $M=M_{X,Y,Z}$ is a crossing matrix of a coloring, then there are sets $X'\subset X$, $Y'\subset Y$ and $Z'\subset Z$ such that $M'=M_{X',Y',Z'}$
is a crossing matrix of the same coloring. More precisely, if we introduce for sets $A=\{a_1<a_2<\dots<a_r\}\subset\mathbb{N}$, $A\ne\emptyset$, and $B\subset[r]$ 
the set
$$
    A(B):=\left\{a_b:\ b\in B\right\}\, ,
$$
then the sets $X'$, $Y'$ and $Z'$ are given by
$$
    X'=X(f([r'])),\ Y'=Y(g([s']))\;\text{ and }\;Z'=Z(h([t']))
$$
where $X$, $Y$, $Z$, $f$, $g$, $h$, $r'$, $s'$ and $t'$ are as above. We call $X'$, $Y'$ and $Z'$  the \emph{base sets of the submatrix $M'$ of $M$}. 

For a three-dimensional $r\times s\times t$ $*$-binary matrix $M$ we define in a moment two-dimensional $*$-binary layer and cross- matrices $N$. They are motivated by colorings of type $W_2$ and $W_3$, respectively.
We fix the order of kinds of lines in two- and three-dimensional matrices as: rows first, then columns and shafts last. We define the $N$ derived from $M$ 
so that this order is preserved when we go from the lines in $N$ to their counterparts in $M$. Therefore we define a layer matrix $N$ of $M$ with exchanged 
coordinates; we already mentioned this terminological peculiarity. We do the same for cross-matrices (and illustrate it by the examples below) 
even if for them some lines in $N$ do not have corresponding lines in $M$. 
\begin{itemize}
\item For $i\in[3]$, $N$ is an {\it $i$-layer matrix of $M$} if $N$ arises from $M$ by fixing the $i$-th coordinate and swapping the remaining ones. 
\begin{enumerate}
    \item $N$ is a $1$-layer matrix of $M$ if  $N(a,b) = M(z,b,a)$ with a fixed $z$.
    \item $N$ is a $2$-layer matrix of $M$ if  $N(a,b) = M(b,z,a)$ with a fixed $z$.
    \item $N$ is a $3$-layer matrix of $M$ if  $N(a,b) = M(b,a,z)$ with a fixed $z$.
    \end{enumerate}
For example, in part 1 a row in $N$ determines a column in $M$, a column in $N$ determines a shaft in $M$, and the orders `row' $<$ `column' and 
`column' $<$ `shaft' agree. 

We call $N$ a {\it layer matrix of $M$} if it is a $i$-layer matrix of $M$ for some $i\in [3]$. If $M=M_{X,Y,Z}$ is a crossing matrix of a coloring with
$r=|X|$, $s=|Y|$ and $t=|Z|$, and if $N$ is the $1$-layer matrix of $M$ given by the fixed coordinate $z\in[r]$, then we call the three sets 
$$
    (X',\,Y',\,Z'):=(X(\{z\}),\,Y,\,Z)
$$
the \emph{base sets of the $1$-layer matrix $N$ of $M$}. Note that $N$ can be viewed as the crossing matrix $M'=M_{X',Y',Z'}$ of the same coloring as $M$ because $N(a,b)=M'(1,b,a)$.
If $N$ is a $2$-layer (resp. a $3$-layer) matrix of $M$, we define its base sets in the analogous way.
\item For $\{i<j\}\subset[3]$, the {\it $(i,j)$-d-cross-matrix $N$ of $M$} is defined by equating the $i$-th and $j$-th coordinate in $M$ 
and swapping them with the remaining coordinate.
\begin{enumerate}
    \item $N$ is the $(1,2)$--d-cross-matrix of $M$ if $N(a,b) = M(b,b,a)$ (we assume that $r=s$). 
    \item $N$ is the $(1,3)$--d-cross-matrix of $M$ if $N(a,b) = M(b,a,b)$ (we assume that $r=t$).
    \item $N$ is the $(2,3)$--d-cross-matrix of $M$ if $N(a,b) = M(b,a,a)$ (we assume that $s=t$).
\end{enumerate}
For example, in part 2 a row in $N$ determines variable coordinates $\{1,3\}$ in $M$, a column in $N$ determines variable coordinate $\{2\}$ in $M$, 
and the orders  `row' $<$ `column' and $\{1,3\}<\{2\}$ (by the minimal elements) agree. 
\item Similarly, for $\{i<j\}\subset[3]$ we define the {\it $(i,j)$-ad-cross-matrix $N$ of $M$} as follows. 
\begin{enumerate}
    \item $N$ is the $(1,2)$--ad-cross-matrix of $M$ if $N(a,b) = M(b,s-b+1,a)$ (we assume that $r=s$). 
    \item $N$ is the $(1,3)$--ad-cross-matrix of $M$ if $N(a,b) = M(b,a,t-b+1)$ (we assume that $r=t$).
    \item $N$ is the $(2,3)$--ad-cross-matrix of $M$ if $N(a,b) = M(b,a,t-a+1)$ (we assume that $s=t$).
\end{enumerate}
For example, in part 3 a row in $N$ determines variable coordinate $\{1\}$ in $M$, a column in $N$ determines variable coordinates $\{2,3\}$ in $M$, 
and the orders `row' $<$ `column' and $\{1\}<\{2,3\}$ (by the minimal elements) agree. 

\item We say that $N$ is an \emph{$(i,j)$-cross-matrix} of $M$ if it is an $(i,j)$-d-cross-matrix or an $(i,j)$-ad-cross-matrix of $M$. We say that $N$ is a \emph{cross-matrix} of $M$ if it is an $(i,j)$-cross-matrix  of $M$ for some $\{i<j\}\subset[3]$. 

If $M=M_{X,Y,Z}$ is a crossing matrix 
of a coloring $K=(n,\chi)$ with $r=|X|$, $s=|Y|$ and $t=|Z|$, and if $N$ is a cross-matrix of $M$, then we call the three sets $X,Y,Z$ the \emph{base sets of the cross-matrix 
$N$ of $M$}. Thus a cross-matrix $N$ of $M=M_{X,Y,Z}$ has the same base sets as $M$. Unlike for submatrices and layer matrices, now $N$ cannot be viewed as a crossing matrix of the coloring $K$ 
but is of course still determined by $K$. For example, if $N$ is the $(1,3)$-ad-cross-matrix of $M=M_{X,Y,Z}$ then 
$$
    N(i,\,j)=\chi(\;\{X(\{j\}),\,Y(\{i\}),\,Z(\{t-j+1\})\}\;)
$$
or $=*$ if two of the three elements in $\{\cdot,\cdot,\cdot\}$ coincide. 
\end{itemize}
\begin{figure}
    \centering
    \subfloat{{\includegraphics[width=0.25\textwidth]{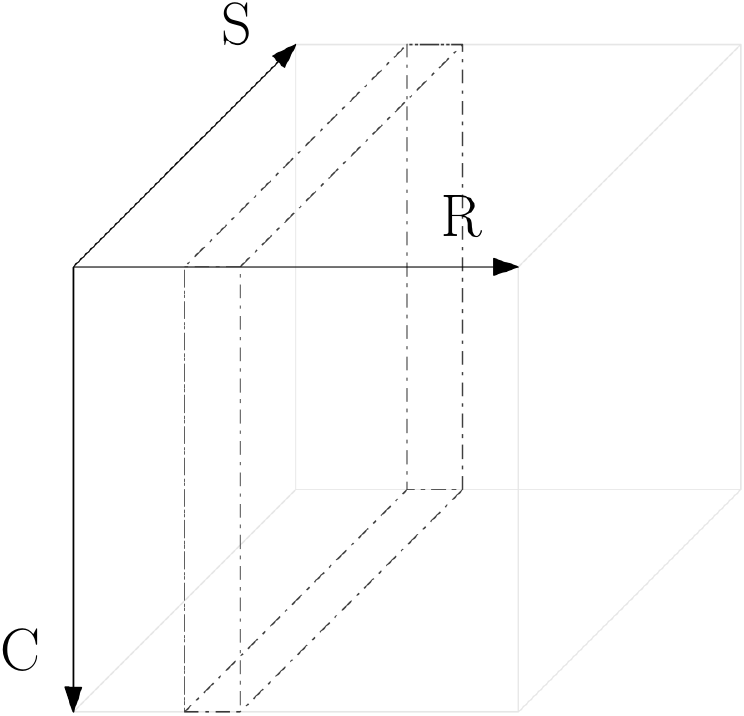} }}%
    \qquad
    \subfloat{{\includegraphics[width=0.25\textwidth]{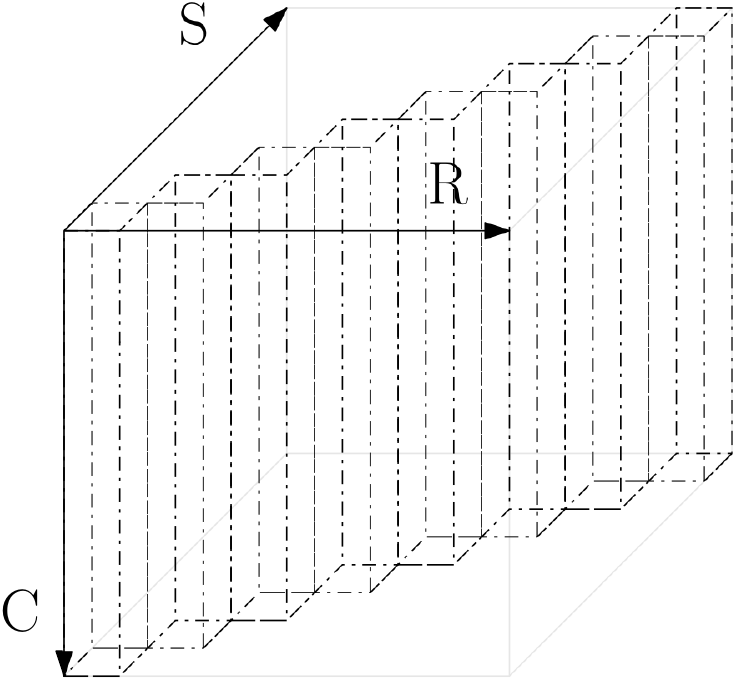} }}%
    \qquad
    \subfloat{{\includegraphics[width=0.25\textwidth]{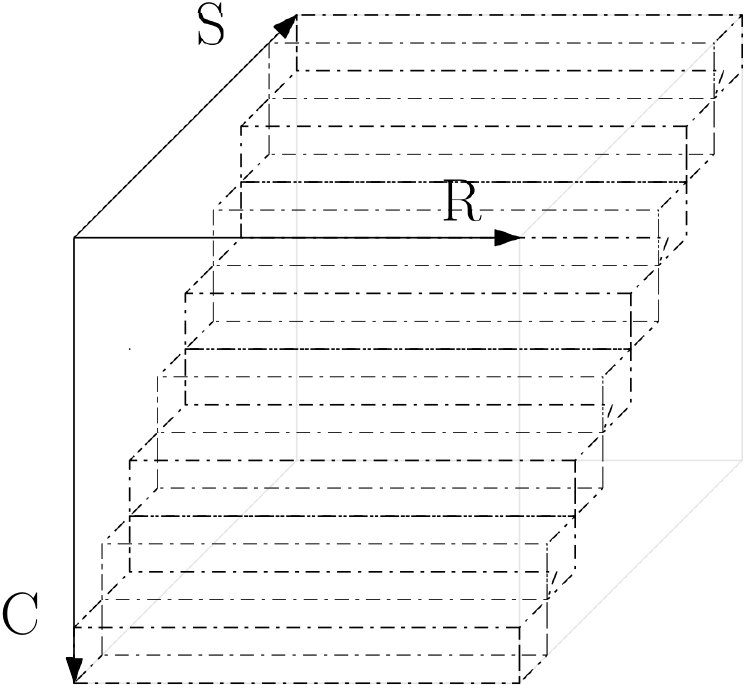} }}%
    \caption{The dashed part of the matrix $M$ displays a $1$-layer matrix, the $(1, 3)$-d-cross-matrix and the $(2, 3)$-ad-cross-matrix.}%
    \label{fig:Cuttingmatrices}%
\end{figure}
The above defined two-dimensional matrices $N$ related to $M$ are illustrated in Figure~\ref{fig:Cuttingmatrices}. In the following lemma whose easy  proof we omit we review some relations between the operations
of taking a submatrix, taking a layer matrix and taking a cross-matrix.

\begin{lemma}\label{subm_layer_crossm}
Let $M$ be a three-dimensional $*$-binary matrix. Then for any two-dimensional $*$-binary matrix $N'$ the following holds.
$$
    1.\ (\exists\;\text{layer matrix $N$ of $M$}:\ N'\preceq N)\iff(\exists\,M':\ M'\preceq M\;\&\;\text{$N'$ is a layer matrix of $M'$})\, .
$$
$$
    \hspace{-0.2cm}2.\ (\exists\;\text{cross-matrix $N$ of $M$}:\ N'\preceq N)\implies(\exists\,M':\ M'\preceq M\;\&\;\text{$N'$ is a cross-matrix of $M'$})
$$
but the opposite implication in general does not hold. 
\end{lemma}
In the direction $\impliedby$ in part 1 matrices $N'$ and $M'$ determine $N$ uniquely, and similarly in the direction $\implies$ in part 2. 
In the direction $\implies$ in part 1 the matrix $M'$ is not determined uniquely by $N'$ and $N$, and to make it unique we may take for example the smallest $M'$.

In Section~\ref{sec_five} we will work with matrices $N'$ obtained from matrices $M$, $N$ and $M'$ in the ways described in the left and right
sides in parts 1 and 2, and the base sets of $N'$ will be important. It is straightforward to determine them if $M=M_{X,Y,Z}$ is a crossing matrix of a coloring, 
if we know the fixed coordinate for the layer matrix $N$, and if we know how $M'$ and $N'$ embed in their supermatrices. For illustration we describe it 
in detail in one of the four situations, the left side in part 1. Suppose that $M=M_{X,Y,Z}$ is a crossing matrix of a coloring and with $|X|=r$, $|Y|=s$ 
and $|Z|=t$, that $N$ is the $t\times r$ $2$-layer matrix of $M$ given by the fixed coordinate $z\in[s]$, and that $N'$ is the $t'\times r'$ submatrix of $N$ given by 
the increasing injections $f\colon[t']\to[t]$ and  $g\colon[r']\to[r]$. Then $M$ has the base sets $X,Y$ and $Z$, $N$ has the  base sets $X,Y(\{z\})$ and $Z$, and 
$N'$ has the base sets $X(g([r'])),Y(\{z\})$ and $Z(f([t']))$.

Let $N\colon[r]\times[s]\to \{0,1,*\}$ be a two-dimensional $*$-binary matrix. We say that $N$ is \emph{$R$-full} if there exist $r$ distinct column indices 
$s_i\in[s-1]$ such that 
$$
    \{N(i,\,s_i),\,N(i,\,s_i+1)\}=\{0,\,1\}\;\text{ for }\;i=1,\,2,\,\dots,\,r\, . 
$$
We define similarly that $N$ is \emph{$C$-full} if for some $s$ distinct row indices $r_i$, for $i \in [s]$ the 
values $N(r_i,i)$ and $N(r_i+1,i)$ are $0$ and $1$ or vice versa. The following crucial lemma is a three-dimensional version of Lemma~\ref{KlazarL3.13} and reduces three-dimensional matrices to two-dimensional ones.

\begin{lemma}\label{3DKlazar}
Let $(M_n)_{n \geq 1}$ be an infinite sequence of three-dimensional $*$-binary matrices such that the sequence $(\al(M_n))_{n \geq 1}$ is bounded but the sequence $(|R(M_n)|)_{n \geq 1}$ 
is unbounded. Then (i) or (ii) holds. 
\begin{enumerate}
\item[(i)] Every matrix $M_n$ has a $2$-layer matrix $N_n$ such that the sequence $(|R(N_n)|)_{n \geq 1}$ is unbounded, or every matrix $M_n$ has a $3$-layer matrix 
$N_n$ such that the sequence $(|R(N_n)|)_{n \geq 1}$ is unbounded.
\item[(ii)] Every matrix $M_n$ has a submatrix $M_n'$ which has a $(2,3)$-cross-matrix $N_n$ that is $R$-full and such that the sequence $(|R(N_n)|)_{n \geq 1}$ is unbounded.
\end{enumerate}
\end{lemma}

\begin{proof}
Let the sequences $(\al(M_n))_{n \geq 1}$ and $(|R(M_n)|)_{n \geq 1}$ be bounded and unbounded, respectively. Hence for every $n \in \mathbb{N}$ there is a finite sequence 
$$
    \left((J_{n,1},\, K_{n,1}),\,(J_{n,2},\, K_{n,2}),\,\dots,\,(J_{n,m(n)},\, K_{n,m(n)})\right)
$$
of mutually distinct coordinates $(J_{n,j}, K_{n,j})\in\mathbb{N}^2$ of rows in $M_n$ such that the sequence $(m(n))_{n \geq 1}$ is unbounded 
and for every $j \in [m(n)]$ we have that
$$
    R\left(M_n;\,J_{n,j},\,K_{n,j}\right) \setminus \bigcup_{i=1}^{j-1} R(M_n;\,J_{n,i},\, K_{n,i})\ne\emptyset\, .
$$ 
First we suppose that there is a sequence $(J_n)_{n\ge1}\subset\mathbb{N}$ such that the sequence 
$$
    \left(|\{j\in[m(n)]\mid J_{n,j} = J_n\}|\right)_{n\ge1}
$$ 
is unbounded, i.e. some second coordinate $J$ has unbounded multiplicity. Then by the above displayed non-equality the $2$-layer matrices 
$$
    N_n(a,\,b) = M_n(b,\,J_n,\,a)
$$ 
of $M_n$ are such that the sequence $(|R(N_{n})|)_{n\ge1}$ is unbounded and we have case (i). Similarly, if there is a sequence $(K_n)_{n\ge1}\subset\mathbb{N}$ 
such that the numbers of indices $j\in[m(n)]$ with $K_{n,j} = K_n$ form an unbounded sequence, then the $3$-layer matrices $N_n(a,b) = M_n(b,a,K_n)$ of $M_n$ 
give an unbounded sequence $(|R(N_{n})|)_{n\ge1}$.

Suppose that the maximum multiplicity of both the $J$ and the $K$ coordinate is bounded and therefore the numbers of distinct $J$ coordinates 
and of distinct $K$ coordinates are unbounded. Using twice the Erd\H os--Szekeres lemma, which says that every sequence of $(k-1)^2+1$ numbers has a $k$-term 
monotone subsequence, we deduce that there is a sequence $(Z_n)_{n \geq 1}$ of sets of indices 
$$
    Z_n = \{j_{n,1} < j_{n,2} < \dots < j_{n,|Z_n|}\} \subset [m(n)]
$$ 
such that the sequence 
$(|Z_n|)_{n\ge1}$ is unbounded and for every $n$ both
sequences 
$$
    \overline{J_n}=\left(J_{n,j_{n,i}}\right)_{1 \le i \le |Z_n|}\;\text{ and }\;\overline{K_n}=\left(K_{n,j_{n,i}}\right)_{1 \le i \le |Z_n|}
$$ 
are monotonic, each strictly increases or strictly decreases. We define
$$
    M_n^0(a,\,b,\,c) = M_n(a,\,J_{n,j_{n,b}},\, K_{n,j_{n,c}})
$$
and consider the submatrix $M_n'$ of $M_n$ with the positions $(a, J_{n,j_{n,b}}, K_{n,j_{n,c}})$ where $a$ runs in the first dimension of $M_n$ 
and $b,c$ run in $[|Z_n|]$.
When the sequences $\overline{J_n}$ and $\overline{K_n}$ are monotonic in the same sense, both increase or both decrease, we take the $(2,3)$-d-cross-matrix $$
    N_n(a,\,b) = M_n^0(b,\,a,\,a)\;\text{ or }\;N_n(a,\,b) = M_n^0(b,\,|Z_n|-a+1,\,|Z_n|-a+1)
$$
of $M_n'$, respectively. When $\overline{J_n}$ increases and $\overline{K_n}$ decreases
or vice versa, we take the $(2,3)$-ad-cross-matrix 
$$
    N_n(a,\,b) = M_n^0(b,\,a,\,|Z_n|-a+1)\;\text{ or }\;N_n(a,\,b) = M_n^0(b,\,|Z_n|-a+1,\,a)
$$
of $M_n'$, respectively. By the above displayed non-equality every matrix $N_n$ is $R$-full. In particular, since $(|Z_n|)_{n\ge1}$ is unbounded, the sequence $(|R(N_n)|)_{n \geq 1}$ is unbounded 
and we have case (ii). 
\end{proof}

When we replace rows with columns, we obtain the following symmetric result. It has an analogous (or symmetric) proof which we
omit. But note that because of the interchange of coordinates in the definitions of layer and cross-matrices, we cannot just replace 
in Lemma~\ref{3DKlazar} `R' with `C'.

\begin{lemma}\label{3DKlazar2}
Let $(M_n)_{n \geq 1}$ be an infinite sequence of three-dimensional $*$-binary matrices such that the sequence $(\al(M_n))_{n \geq 1}$ is bounded 
but the sequence $(|C(M_n)|)_{n \geq 1}$ is unbounded. Then (i) or (ii) holds. 
\begin{enumerate}
\item[(i)] Every matrix $M_n$ has a $1$-layer matrix $N_n$ such that the sequence $(|R(N_n)|)_{n \geq 1}$ is unbounded, or every matrix $M_n$ has a $3$-layer matrix 
$N_n$ such that the sequence $(|C(N_n)|)_{n \geq 1}$ is unbounded.
\item[(ii)] Every matrix $M_n$ has a submatrix $M_n'$ which has a $(1,3)$-cross-matrix $N_n$ that is $C$-full and such that the sequence $(|C(N_n)|)_{n \geq 1}$ is unbounded.
\end{enumerate}
\end{lemma}

\begin{remark} \label{rem:fullrow}
Note that in the case (ii) of Lemma~\ref{3DKlazar} each row of the matrix of $N_n$ is also a row of the matrix $M_n$, in the sense that if $M_n=M_{X_n,Y_n,Z_n}$ is a crossing matrix of a coloring then 
the three first base sets of the matrices $N_n$, $M_n'$ and $M_n$ are all equal to $X_n$. Similarly, in the case (ii) of Lemma~\ref{3DKlazar2} each column of $N_n$ is in the analogous sense  a column of the matrix $M_n$.
\end{remark}

\begin{lemma}\label{lemma_on_diagonals}
Suppose that $(M_n)_{n\ge1}$ is a sequence of three-dimensional $*$-binary matrices, that every matrix $M_n$ has a binary submatrix $M_n'$
such that the sequence $(\al(M_n'))_{n\ge 1}$ is bounded, and that every matrix $M_n'$ has a cross-matrix $N_n$ such that the sequence 
$(\al(N_n))_{n\ge 1}$ is unbounded. Then every matrix $M_n'$ has a (binary) layer matrix $P_n$ such that the sequence $(\al(P_n))_{m\ge 1}$ is bounded and 
both sequences $(|R(P_n)|)_{n\ge 1}$ and $(|C(P_n)|)_{n\ge 1}$ are unbounded.
\end{lemma}
\begin{proof}
By passing to a subsequence of $n=1,2,\dots$ we may assume that the type of cross-matrix $N_n$ of $M_n'$ is constant in $n$. For concreteness we assume that each $N_n$ is a 
$(2,3)$-d-cross-matrix of $M_n'$, the other five cases for other cross-matrices are treated by similar arguments. Thus, by the above definition of cross-matrices, 
$$
    N_n(a,\,b)=M_n'(b,\,a,\,a)\, . 
$$
Let $L$ be a general three-dimensional $r\times s\times s$ binary matrix. We define a
\emph{$(2,3)$-diagonal of $L$} as $L(x,y,y)$ where $x\in[r]$ is fixed and $y$ runs in $[s]$ and set
$$
    \al_{23d}(L)=1+ \max_{x\in[r]}\left| \{i\in[s-1]:\ \{L(x,\,i,\,i),\,L(x,\,i+1,\,i+1)\}=\{0,\,1\}\} \right| \, .
$$
Unboundedness of the sequence $(\al(N_n))_{n\ge 1}$ implies unboundedness of the sequence $(\al_{23d}(M_n'))_{n\ge 1}$ because rows in $N_n$ are rows in $M_n'$ and thus contribute to $\al(N_n)$ by a bounded amount, but columns in $N_n$ are $(2,3)$-diagonals in $M_n'$ to which boundedness of $\al(M_n')$ does 
not apply. Thus we see that every matrix $M_n'$ has a $1$-layer matrix $P_n$ (see Figure~\ref{fig:cross_c} in Section~\ref{sec_endof proof} for a similar situation) 
such that, if $s(P_n)$ denotes the number of subwords $01$ and $10$
on the main diagonal of $P_n$, the sequence $(s(P_n))_{n\ge1}$ is unbounded. Now let $N$ be any two-dimensional binary $r\times r$ matrix with $s(N)=k\ge1$
and $\al(N)=l$. Clearly, $l\ge2$. Let $i\in[r-1]$ be such that 
$$
    \{N(i,\,i),\,N(i+1,\,i+1)\}=\{0,\,1\}\, , 
$$
that is, the position $(i,i)$ on the diagonal of $N$ contributes $1$ to $s(N)$. Considering the 
position $(i+1,i)$ in $N$ we see that either $i\in R(N)$ or $i\in C(N)$. Thus either $|R(N)|\ge k/2$ or $|C(N)|\ge k/2$. If the latter inequality holds then by Lemma~\ref{KlazarL3.12}
we have that 
$$
    |R(N)|\ge\frac{1}{2}\left(\frac{|C(N)|}{l-1}-1\right)\ge\frac{1}{2}\left(\frac{k}{2(l-1)}-1\right)\, .
$$
As this is less than $k/2$, the displayed lower bound on $|R(N)|$ is true also when the former inequality holds. Since the sequence $(\al(P_n))_{n\ge1}$ is bounded ($P_n$ is a layer matrix of $M_n'$) and the sequence $(s(P_n))_{n\ge1}$ is unbounded, by the displayed lower bound on $|R(N)|$ we see that the sequence $(|R(P_n)|)_{n\ge1}$ 
is unbounded. By symmetry, this argument shows that also the sequence $(|C(P_n)|)_{n\ge1}$ is unbounded.
\end{proof}

\noindent
To be precise, we will use this lemma in the situation when not $\al(M_n')$ but only the numbers $\al(M_n)$ are bounded, but $M_n$ will be such that this still implies boundedness of 
$\al(M_n')$ for any (not necessarily binary) submatrix $M_n'$ of $M_n$. This is explained in the remark just before Proposition~\ref{P1}.

\subsection{Crossing matrices and $p$-tame colorings}\label{subsec_ptame}

For a coloring $H = (n,\chi)$ we define its \emph{nuclear decomposition $\gr(H)$} to be the interval partition $I_1 < I_2 < \dots < I_s$, 
$s\in\mathbb{N},$ of $[n]$ such that $I_1$ is the longest initial $\chi$-monochromatic interval, $I_2$ is the longest $\chi$-monochromatic 
interval following after $I_1$, and so on. Then each interval $I_i$ is $\chi$-monochromatic and for every $i=1,2,\dots,s-1$ we have that $|I_i|\ge3$ and the set $I_i\cup\{\min(I_{i+1})\}$ is not $\chi$-monochromatic. We note that if the nuclear decomposition of $H$ has length $s\ge 2r$ then 
$H\succeq K$ where $K$ is an $r$-wealthy coloring of type $W_{4,1}$. Indeed, it is easy to find triples $a_i,b_i,c_i\in I_{2i-1}$, 
$i\in[r]$, of mutually distinct elements such that none of the $r$ 
quadruples
$$
\{\,a_i,\,b_i,\,c_i,\,\min(I_{2i})\}\subset I_{2i-1}\cup I_{2i},\ i\in[r]\,,
$$
is $\chi$-monochromatic. Nuclear decompositions are analogous to the decompositions in the proof of Proposition~\ref{new_variation_onKlazar}.
Let $p\in\mathbb{N}$. We say that a coloring $H = (n,\chi)$ is {\it $p$-tame} if its nuclear decomposition
$$
\gr(H)=\{I_1 < I_2 < \dots < I_s\}
$$ 
satisfies the following five conditions (cf. the remark after Lemma~\ref{rema_RCS}). Recall that for subsets (here actually intervals) 
$I,J,K\subset[n]$, $M_{I,J,K}$ denotes the crossing  matrix of $H$ with the base sets $I,J$ and $K$, as defined in Section~\ref{subs_def_al}.

\begin{enumerate}
\item The length of the nuclear decomposition $s\le p$.
\item For all integers $u,v,w$ with $1 \leq u < v < w \leq s$ one has $\al(M_{I_u,I_v,I_w}) \leq p$ (see Section~\ref{subs_def_al} for the definition of $\al(\cdot)$).
\item For all integers $u,v,w$ with $1 \leq u < v < w \leq s$ one has $|R(M_{I_u,I_v,I_w})| \leq p$ and $|C(M_{I_u,I_v,I_w})| \leq p$ (see Section~\ref{subs_def_al} for the definition of $R(\cdot)$ and $C(\cdot)$).
\item For all integers $u,v$ with $1\leq u < v \leq s$ one has $\al(M_{I_u,I_u,I_v}) \leq p$ and $\al(M_{I_u,I_v,I_v}) \leq p$.
\item For all integers $u,v$ with $1\leq u < v \leq s$ one has $|R(M_{I_u,I_u,I_v})| \leq p$, $|C(M_{I_u,I_u,I_v})| \leq p$, 
$|R(M_{I_u,I_v,I_v})| \leq p$ and $|C(M_{I_u,I_v,I_v})| \leq p$.
\end{enumerate}

\begin{proposition}\label{M-TAME}
For every $n\ge2$ and $p\geq 3$ there are at most $n^{10p^6}$ $p$-tame colorings $(n,\chi)$.
\end{proposition}
\begin{proof}
Let $p,n\in\mathbb{N}$ with $n \ge 2$ and $p \geq 3$ and $H=(n,\chi)$ be a $p$-tame coloring.
We upper-bound the number of different colorings $H$. Recall the well known formula $\binom{n-1}{k-1}$ for the number of partitions of $[n]$ into 
$k$ nonempty intervals $I_1<I_2<\dots<I_k$. The number of possibilities for the interval partitions $I_1 < I_2 < \dots < I_s$ of $[n]$ satisfying condition 
$1$ and for the colors $\chi\, |\, \binom{I_i}{3}\equiv0, 1$ as $i=1,2,\dots,s$ is at most 
$$
    c_1 = \sum_{s=1}^p \binom{n-1}{s-1} 2^s \le n^{p-1}(2+2^2+\dots+2^p) \leq (2n)^p\, .
$$
The colors of the edges outside $\binom{I_i}{3}$ are determined by crossing matrices of three types:
\begin{itemize}
\item[(i)] $M_{I_u,I_v,I_w}$ for $u$, $v$ and $w$ with $1\leq u < v < w \leq s$,
\item[(ii)] $M_{I_u,I_u,I_v}$ for $u$ and $v$ with $1\leq u < v \leq s$, and
\item[(iii)] $M_{I_u,I_v,I_v}$ for $u$ and $v$ with $1\leq u < v \leq s$. 
\end{itemize}
First we bound, for fixed $u$, $v$ and $w$, the number of matrices $M = M_{I_u,I_v,I_w}$ of type (i) satisfying conditions $2$ and $3$. Let $p_1 = |I_w| \leq n$. Using $\al(M) \leq p$ we bound the number of all possible types of shafts (as binary strings) in $M$ by 
$$
    c_2 = 2\sum_{i=1}^{p} \binom{p_1-1}{i-1}\leq c_1\le(2n)^{p}\, .
$$
We set $p_2 = |I_v| \leq n$ and bound the number of two-dimensional 1-layer matrices of $M$ obtained by fixing the first (row) coordinate. Each of these layer 
matrices consists of $p_2$ shafts, parametrized by the second coordinate, which form at most $|C(M)|+1$ different intervals of shafts because two consecutive shafts 
may differ only when the first shaft has its (second) coordinate in $C(M)$. Each of these intervals of shafts can be selected in at most $c_2$ ways (by the type of shafts in 
the interval) and positions of these intervals (the second coordinate of the first shaft in the interval) in at most $\binom{p_2-1}{|C(M)|}$ ways. Hence we have 
an upper bound on the number of those two-dimensional 1-layer matrices
$$
    c_3 = \sum_{i=0}^{p} \binom{p_2 - 1}{i}c_2^{i+1} \leq (p+1)(n-1)^p(2n)^{p^2+p} \leq (2n)^{2p^2}\, .
$$
Finally, we bound the number of whole matrices $M$. Let $p_3 = |I_u| \leq n$. Like before, $M$ consists of $p_3$ 1-layer matrices (parametrized by the row coordinate) that form exactly $|R(M)|+1$ 
different intervals. Thus the number of matrices $M$ of type (i) is bounded by
$$
    c_4 = \sum_{i=0}^{p} \binom{p_3-1}{i}c_3^{i+1} \leq (p+1)(n-1)^{p}(2n)^{2p^3+2p^2} \leq (2n)^{4p^3}
$$
(as before we use that $(p+1)(n-1)^{p}\le(2n)^p$).

To bound, for fixed $u$ and $v$, the number of crossing matrices of types (ii) and (iii) satisfying conditions 4 and 5 we note two things. First, these numbers for types (ii) and (iii) 
are equal (consider reversals). Second, the number of matrices of type (ii) has the same upper bound $(2n)^{4p^3}$ as that of matrices of type (i). 
This is because any matrix $M$ of type (ii) is symmetric in the first two coordinates ($M(i,i,k) = *$ and $M(i,j,k) = M(j,i,k)$ for $i \ne j$) and therefore 
in every 1-layer matrix of $M$ obtained by fixing the first coordinate to $i$ it suffices to consider the triangle with the second coordinate $j$ satisfying $j<i$. Now the bounds 
used in the previous paragraph apply also here and bound the number of these triangles. 

Thus the number of interval partitions $I_1 < I_2 < \dots < I_s$ of $[n]$, $s\le p$, with monochromatic intervals is at most $c_1$, the triples $\{a < b < c \}$ with elements from three distinct intervals can be colored in at most $c_4^{\binom{s}{3}}$ ways, and those from two distinct intervals (either $a, b$ or $b, c$ lie in the same interval) in at most $c_4^{\binom{s}{2}}$ ways. Together there are at most 
$$
    c_1 c_4^{2\binom{s}{2}} c_4^{\binom{s}{3}} \leq c_1 c_4^{p^3} \leq (2n)^{4p^6+p} \leq n^{10p^6}
$$
$p$-tame different colorings $H$. 
\end{proof}

\section{The final part of the proof}\label{sec_five}

We have arrived in the last third of the proof of Theorem~\ref{T1FD}. First we state and prove six lemmas ensuring presence of $r$-wealthy colorings of type $W_i$ in the colorings 
whose crossing matrices satisfy certain conditions, usually involving quantities $\al(\cdot)$, $|R(\cdot)|$ and  $|C(\cdot)|$. The lemmas are given here in the order in which they are employed in the proofs of Propositions~\ref{P1} and \ref{P2}. Then we combine the lemmas with the previous results and finish in Propositions~\ref{P1} and \ref{P2} the proof of Theorem~\ref{T1FD}.

\subsection{Six auxiliary lemmas on matrices and colorings}

Finite sets $A_1,A_2,\dots,A_k\subset\mathbb{N}$ are \emph{non-intertwined} if for some permutation 
$\{i_1,i_2,\dots,i_k\}=\{1,2,\dots,k\}$ one has 
$$
    A_{i_1}<A_{i_2}<\dots<A_{i_k}\, .
$$ 
A three-dimensional crossing matrix $M$ or its submatrix $M'$ 
or its layer matrix $N$ or its cross-matrix $N'$ is {\em non-intertwined} if the base sets of the respective matrix are non-intertwined.  
Clearly if $M$ is non-intertwined then so are the matrices $M'$, $N$ and $N'$. It may happen that $M$ is not non-intertwined but $M'$ or $N$ is (but not $N'$ that has the same base sets as $M$). Non-intertwined base sets are motivated by enforcing colorings of type $W_2$ and 
type $W_3$ in the next three lemmas and their corollaries. 

Recall that $I_r$ denotes the two-dimensional $r\times r$ identity matrix 
and $U_r$ denotes the two-dimensional $r\times r$ upper-triangular matrix. Every  non-intertwined crossing matrix is binary, has no $*$ entry. 

\begin{lemma} \label{poz_W34}
Let $K=(n,\chi)$ be a coloring, $M=M_{X,Y,Z}$ be a crossing matrix of $K$, $M'=M_{X',Y',Z'}$ be a submatrix of $M$ and $N$ be a non-intertwined layer matrix of $M'$. If $N$ is similar 
to the matrix $I_r$ or to the matrix $U_r$, then $K\succeq K'$ where $K'$ an $r$-wealthy coloring of type $W_2$. Moreover, each base set of the coloring $K'$ is mapped in the containment $K'\preceq K$ in one of the sets $X$, $Y$ and $Z$.
\end{lemma}
\begin{proof}
We assume that the matrix $N$ is similar to $I_r$ (the case of $U_r$ is analogous and leads to colorings of type $W_{2,2}$) and that $N$ is a 3-layer matrix of $M'$ 
(the cases of 1-layer and 2-layer  matrices are analogous). Thus $M'$ has dimensions $r \times r \times t$ with $r=|X'|=|Y'|$ and $t=|Z'|$. We see by the definition of crossing matrices, layer matrices, base sets and of similarity to $I_r$ that if
$$
    X'=\{a_1<a_2<\dots<a_r\},\  Y'=\{b_1<b_2<\dots<b_r\}\;\text{ and }\;Z'=\{c_1<c_2<\dots<c_t\}
$$
(these are subsets of $[n]$) then for some $k\in[t]$ the sets $X'$, $Y'$ and $\{c_k\}$ are non-intertwined and for every $i,j\in [r]$ we have 
\begin{align*} 
	N(i,\,j) = M'(j,\,i,\,k) =\chi(\{a_j,\, b_i,\, c_k\}) =   \left\{ \begin{array}{ll}
     1 & \mbox{if }i = j\mbox{ and} \\
     0 & \mbox{if }i \not= j\, ,
    \end{array} \right.
\end{align*}
or $1$ and $0$ are switched and $i$ and/or $j$ may be in $\chi(\cdots)$ replaced with $r-i+1$ and/or $r-j+1$, respectively. Then the coloring $K'\preceq K$ obtained by normalization of the restriction of $K$ to the set $X'\cup Y'\cup\{c_k\}\subset[n]$ is an $r$-wealthy coloring of type $W_{2,1}$.  It is clear that each base set of 
$K'$ is mapped in the containment in one of the sets $X$, $Y$ and $Z$.
\end{proof}

\begin{corollary}\label{cor_prvniholemm}
Suppose that $(K_m)_{m\ge1}$ is a sequence of colorings. We assume for every $m$ that $M_m$ is a crossing matrix of $K_m$, $M_m'\preceq M_m$, and that 
$N_m$ is a non-intertwined layer matrix of $M_m'$ such that the sequence $(\al(N_m))_{m\geq 1}$ is bounded but the sequence 
$(|R(N_m)|)_{m \geq 1}$ or the sequence $(|C(N_m)|)_{m \geq 1}$ is unbounded. Then for every $r\in\mathbb{N}$ there is an $m$ such that $K_m\succeq K_m'$ where $K_m'$
is an $r$-wealthy coloring of type $W_2$. Moreover, each base set of the coloring $K_m'$ is mapped in the containment $K_m'\preceq K_m$ in one of the base sets of $M_m$.
\end{corollary}
\begin{proof}By Lemma~\ref{KlazarL3.13} and the remark after it, for every $r\in\mathbb{N}$ there is an $m=m(r)$ such that $N_m$ has a submatrix $I_r'$ strongly 
similar to $I_r$ or a submatrix $U_r'$ strongly similar to $U_r$. We apply Lemma~\ref{poz_W34} to the non-intertwined matrices 
$I_r'$ and $U_r'$ (each is also a layer matrix of a submatrix of $M_{m(r)}'$, see part 1 of Lemma~\ref{subm_layer_crossm}) and get that for every $r\in\mathbb{N}$ 
there is an $m$ such that the coloring $K_m$ contains an $r$-wealthy coloring $K_m'$ of type $W_2$. It is clear that each base set of 
$K'_m$ is mapped in the containment in one of the base sets of $M_m$.
\end{proof}

\begin{lemma}\label{poz_W56}
Let $K=(n,\chi)$ be a coloring, $M=M_{X,Y,Z}$ be a crossing matrix of $K$, $M'=M_{X',Y',Z'}$ be a submatrix of $M$ and $N$ be a non-intertwined cross-matrix of $M'$. 
If $N$ is similar to the matrix $I_r$ or to the matrix $U_r$, then $K\succeq K'$ where $K'$ an $r$-wealthy coloring of type $W_{3,1}$ or type $W_{3,2}$. Moreover, each base set of the coloring $K'$ is mapped in the containment $K'\preceq K$ in one of the sets $X$, $Y$ and $Z$.
\end{lemma}
\begin{proof}Unlike in Lemma~\ref{poz_W34} here the matrices $N$ and $M'$ have the same base sets. 
We assume that the matrix $N$ is similar to $I_r$  (the case of $U_r$ is analogous and leads to colorings of type $W_{3,2}$) and that $N$ 
is a $(1,2)$-d-cross-matrix of $M'$ (the other five cases of cross-matrices are analogous). Thus $M'$ has dimensions $r \times r \times r$ with 
$r=|X'|=|Y'|=|Z'|$. We see by the definition of crossing matrices, cross-matrices, base sets and of similarity to $I_r$ that the sets
$$
    X'=\{a_1<a_2<\dots<a_r\},\  Y'=\{b_1<b_2<\dots<b_r\}\;\text{ and }\;Z'=\{c_1<c_2<\dots<c_r\}
$$
are non-intertwined subsets of $[n]$. Thus for every $i,j\in [r]$ we have 
\begin{align*} 
	N(i,\,j) = M'(j,\,j,\,i) = \chi(\{a_j,\,b_j,\, c_i\}) =   \left\{ \begin{array}{ll}
     1 & \mbox{if } i = j\mbox{ and} \\
     0 & \mbox{if } i \not= j\,,
    \end{array} \right.
\end{align*}
or $1$ and $0$ are switched and $i$ and/or $j$ may be in $\chi(\cdots)$ replaced with $r-i+1$ and/or $r-j+1$, respectively. Then the coloring $K'\preceq K$ obtained 
by normalization of the restriction of $K$ to the set $X'\cup Y'\cup Z'\subset[n]$ is an $r$-wealthy coloring of type $W_{3, 1}$. It is clear that each base set of 
$K'$ is mapped in the containment in one of the sets $X$, $Y$ and $Z$.
\end{proof}

\begin{corollary}\label{cor_druheholemm}
Suppose that $(K_m)_{m\ge1}$ is a sequence of colorings. We assume for every $m$ that $M_m$ is a crossing matrix of $K_m$, $M_m'\preceq M_m$, and that 
$N_m$ is a non-intertwined cross-matrix of $M_m'$ such that the sequence $(\al(N_m))_{m\geq 1}$ is bounded but the sequence 
$(|R(N_m)|)_{m \geq 1}$ or the sequence $(|C(N_m)|)_{m \geq 1}$ is unbounded. Then for every $r\in\mathbb{N}$ there is an $m$ such that $K_m\succeq K_m'$ 
where $K_m'$ is an $r$-wealthy coloring of type $W_{3,1}$ or type $W_{3,2}$. Moreover, each base set of the coloring $K_m'$ is mapped in the containment $K_m'\preceq K_m$ in one of the base sets of $M_m$.
\end{corollary}
\begin{proof}
Again, unlike in Corollary~\ref{cor_prvniholemm} here the matrices $N_m$ and $M_m'$ have the same base sets.
By Lemma~\ref{KlazarL3.13} and the remark after it, for every $r\in\mathbb{N}$ there is an $m=m(r)$ such that $N_m$ has a submatrix $I_r'$ strongly 
similar to $I_r$ or a submatrix $U_r'$ strongly similar to $U_r$. We apply Lemma~\ref{poz_W56} to the  non-intertwined matrices 
$I_r'$ and $U_r'$ (each is also a cross-matrix of a submatrix of $M_{m(r)}'$, see part 2 Lemma~\ref{subm_layer_crossm}) and get that for every $r\in\mathbb{N}$ 
for an $m$ the coloring $K_m$ contains an $r$-wealthy coloring $K_m'$ of type $W_{3,1}$ or type $W_{3,2}$. It is clear that each base set of 
$K'_m$ is mapped in the containment in one of the base sets of $M_m$.
\end{proof}

\begin{lemma}\label{lema_dvaapul}
Suppose that $(K_m)_{m\ge1}$ is a sequence of colorings. We assume for every $m$ that $M_m$ is a crossing matrix of $K_m$, that $M_m'\preceq M_m$ and the sequence 
$(\al(M_m'))_{m\ge1}$ is bounded, and that $N_m$ is a non-intertwined cross-matrix of $M_m'$ such that the sequence $(|R(N_m)|)_{m \geq 1}$ or the sequence 
$(|C(N_m)|)_{m \geq 1}$ is unbounded. Then for every $r\in\mathbb{N}$ there is an $m$ such that $K_m\succeq K_m'$ where $K_m'$ is an $r$-wealthy coloring of 
one and the same type $W_2$ or type $W_{3,1}$ or type $W_{3,2}$. Moreover, each base set of $K_m'$ is mapped in the containment $K_m'\preceq K_m$ in one of the base sets of $M_m$.
\end{lemma}
\begin{proof}We assume that the sequence $(|R(N_m)|)_{m \geq 1}$ is unbounded, the other case when the sequence $(|C(N_m)|)_{m \geq 1}$ is unbounded is very similar.
We consider two cases depending on whether the sequence $(\al(N_m))_{m\geq 1}$ is bounded or not. If it is bounded then we apply Corollary~\ref{cor_druheholemm} to the matrices $N_m$ and conclude 
that for every $r\in\mathbb{N}$ there is an $m$ such that the coloring $K_m$ contains an $r$-wealthy coloring of type $W_{3,1}$ or type $W_{3,2}$ and that the condition on bases sets holds. If the 
sequence $(\al(N_m))_{m\geq 1}$ is unbounded then by Lemma~\ref{lemma_on_diagonals} every matrix $M_m'$ has a (square) layer matrix $P_m$ such 
that the sequence $(\al(P_m))_{m\geq 1}$ is bounded but the sequence $(|R(P_m)|)_{m\geq 1}$ is unbounded. Note that $P_m$ is non-intertwined because 
$N_m$ and $M_m'$ are non-intertwined. We apply Corollary~\ref{cor_prvniholemm} to the matrices $P_m$ and conclude that for every $r\in\mathbb{N}$ there is an $m$ such that the coloring $K_m$ 
contains an $r$-wealthy coloring of type $W_2$. It is clear that the condition on base sets is again satisfied.
\end{proof}

\begin{lemma}\label{treti_lemma}
Let $r\in\mathbb{N}$, $K=(n,\chi)$ be a coloring, $M=M_{X,X,Y}$ with $X<Y$ be a crossing matrix of $K$, and $M'$ be a submatrix of $M$ whose 
$(2,3)$-d-cross-matrix $N$ 
has dimensions $r\times 3r$ and satisfies 1 or 2.
\begin{enumerate}
    \item For $i\in[r]$, $N(i,3i-2)=*$ and $\{N(i,3i-1),N(i,3i)\}=\{0,1\}$.
    \item For $i\in[r]$, $\{N(i,3i-2),N(i,3i-1)\}=\{0,1\}$ and $N(i,3i)=*$.
\end{enumerate}
Then $K$ contains an $r$-wealthy coloring $K'$ of type $W_{4,2}$. Moreover, each base set of $K'$ is mapped in the containment in one of the 
sets $X$ and $Y$.
\end{lemma}
\begin{proof}
We assume that 1 holds, assumption 2 is treated similarly.
We have that $N(a,b)=M'(b,a,a)$, $M'=M_{X',X'',Y'}$ for some sets $X',X''\subset X$ and $Y'\subset Y$ with
$$
    X'=\{x_1<x_2<\dots<x_{3r}\},\ X''=\{x_1<x_4<x_7<\dots<x_{3r-2}\},\ Y'=\{y_1<y_2<\dots<y_r\}
$$
and that for every $i\in[r]$,
$$
    \{\chi(\{x_{3i-2}<x_{3i-1}<y_i\}),\;\chi(\{x_{3i-2}<x_{3i}<y_i\})\}=\{0,\,1\}\, .
$$
Normalization of the restriction of $K$ to the set $X'\cup Y'\subset[n]$ is an $r$-wealthy coloring $K'$ of type $W_{4,2}$. It is clear that 
the first base set of $K'$ is in the containment mapped in $X$ and the second one in $Y$.
\end{proof}

Let $K$ be a coloring, $M=M_{X,Y,Y}$ with $X<Y$ be a three-dimensional $*$-binary crossing matrix of $K$ and $M'$ be a submatrix of $M$ 
with the base sets $\{r_1<\dots<r_s\}\subset X$ and $\{c_1<\dots<c_r\},\{s_1<\dots<s_r\}\subset Y$. If $c_i<s_i$ (resp. $c_i>s_i$) for every $i\in[r]$, we say that 
the $r\times s$ $(2,3)$-d-cross-matrix $N$ of $M'$, $N=N(i,j) = M(r_j,c_i,s_i)$, is an \emph{above-diagonal} (resp. a \emph{below-diagonal}) cross-matrix.  Clearly, 
each such $N$ is binary. 

\begin{lemma}\label{ctvrte_lemma}
Let $(K_m)_{m \geq 1}$ be a sequence of colorings and $M_m = M_{X_m,Y_m,Y_m}$ with $X_m<Y_m$ be $*$-binary crossing matrices of $K_m$. Let each matrix $M_m$ have a submatrix $M_m'$ whose (two-dimensional) $(2,3)$-d-cross-matrix 
$N_m'$ is above-diagonal, resp. below-diagonal. Suppose that the sequence $(\al(M_m))_{m\ge1}$ is bounded but
$$
    \text{ both sequences $(|R(N_m')|)_{m\ge1}$ and  $(\al(N_m'))_{m\ge1}$ are unbounded}\,.
$$
Then for every $r\in\mathbb{N}$ there is an $m$ such that $K_m$ contains an $r$-wealthy coloring of one and the same type $W_{3,3}$ or type $W_2$.
\end{lemma}
\begin{proof}
We assume that all matrices $N_m'$ are above-diagonal and treat the other case when they are below-diagonal at the end. As we know, every $N_m'$ is binary.
Since the numbers of subwords $01$ and $10$ in the rows of 
$N_m'$ are bounded (each row of $N_m'$ is a row of $M_m'$), for every $m$ there exists a column of $N_m'$ with index $c_m$ such that 
the sequence $(\al(N_m'(\cdot,c_m)))_{m \geq 1}$ is unbounded, where $\al(N_m'(\cdot,c_m))$ is defined as one plus the number of subwords $01$ and $10$ in the $c_m$-th column of $N_m'$. Let $P_m$ be the $1$-layer matrix of $M_m$ containing the positions of the $c_m$-th column of $N_m'$ (see Figure~\ref{fig:cross_c} at the end of the proof of Proposition~\ref{P2}). 
Clearly, $P_m$ is a square symmetric matrix with $*$s on the main diagonal and $0$s and $1$s elsewhere. Let $r(m)$ be the number of rows in $N_m'$ and 
$$
    (c_m,\,x_{m,i},\,y_{m,i}),\ i=1,\, 2,\, \dots,\, r(m)\, ,
$$
with
$$
    x_{m,1}<x_{m,2}<\dots<x_{m,r(m)}\;\text{ and }\;y_{m,1}<y_{m,2}<\dots<y_{m,r(m)}
$$
be the positions in $M_m$ of the $c_m$-th column of $N_m'$. Hence
$$
    N_m'(a,\,c_m)=M_m(c_m,\,x_{m, a},\,y_{m, a})=P_m(y_{m, a},\,x_{m, a}).
$$
The sequence $(\al(N_m'(\cdot,c_m)))_{m \geq 1}$ is unbounded and thus so is  $(|R(P_m)|)_{m \geq 1}=(|C(P_m)|)_{m \geq 1}$. Indeed, 
consider in $P_m$ the zig-zag path 
$$
    Z_m=\left(u_1,\,u_1',\,u_2,\,u_2',\,u_3,\,\dots,u_{r(m)-1}',\,u_{r(m)}\right),\  u_i=(y_{m,i},\,x_{m,i})\;\text{ and }\;u_i'=(y_{m,i+1},\,x_{m,i})\, ,
$$
which has turns in the $u_i'$s and the $u_i$s with $1<i<r(m)$. Since columns and shafts of $M$ correspond to rows and columns of $N$, the path $Z_m$ lies in $P_m$ below the main diagonal of $*$s and contains 
therefore only $0$s and $1$s. Let $e_m =\al(N_m'(\cdot,c_m))-1$ be the number of subwords $01$ and $10$ in the binary sequence
$$
    \left(P_m(u_1),\,P_m(u_2),\,\dots,\,P_m(u_{r(m)})\right)
$$
where $P_m(u_1)$ means $P(y_{m,1},x_{m,1})$ etc. It follows that at least $e_m/2$ of the vertical segments $(u_i,u_i')$ of $Z_m$ 
contain a subword $01$ or $10$, or this 
holds for the horizontal segments $(u_i',u_{i+1})$. Thus $(|R(P_m)|)_{m \geq 1}=(|C(P_m)|)_{m \geq 1}$ is unbounded. 

Hence we can apply Proposition~\ref{new_variation_onKlazar} to the sequence $(P_m)_{m \geq 1}$  and its case ($\alpha$) or ($\beta$) holds. 
It is not hard to see that in case ($\alpha$) because of the type 2 colorings of pairs produced by the matrices $P_m$, for every $r\in\mathbb{N}$ there is an $m$ such that 
the coloring $K_m$ contains an $r$-wealthy coloring  of type $W_{3,3}$ (determined by the fixed vertex corresponding to the fixed first coordinate and by an 
$r$-wealthy coloring of pairs of type $2$ that follows after it). In case ($\beta$) we consider submatrices $P_m'\preceq P_m$ strongly similar to the identity matrix 
$I_r$ or to the upper triangular matrix $U_r$. Each $P_m'$ is non-intertwined as it lies above the diagonal 
of $P_m$. By part 1 of Lemma~\ref{subm_layer_crossm} each $P_m'$ is also a layer matrix of a submatrix of $M_m$. We apply Lemma~\ref{poz_W34} to the matrices $P_m'$ and conclude that for every $r\in\mathbb{N}$ there is an $m$ such that the coloring $K_m$ contains 
an $r$-wealthy coloring of type $W_2$.

If all matrices $N_m'$ are below-diagonal, the same argument works for the zig-zag path 
$$
    Z_m'' = \left(u_1,\,u_1'',\,u_2,\,u_2'',\,u_3,\,\dots,u_{r(m)-1}'',\,u_{r(m)}\right),\  u_i=(y_{m,i},\,x_{m,i})\;\text{ and }\;u_i''=(y_{m,i},\,x_{m,i+1})\, ,
$$ 
which lies above the diagonal of $*$s.
\end{proof}

\begin{lemma}\label{pate_lemma}
Let $(K_m)_{m \geq 1}$ be a sequence of colorings and $M_m = M_{X_m,Y_m,Y_m}$ with $X_m<Y_m$ be $*$-binary crossing matrices of $K_m$. Let each matrix $M_m$ have a submatrix $M_m'$ whose (two-dimensional) $(2,3)$-d-cross-matrix $N_m'$ is above-diagonal, resp. below-diagonal. Suppose that
$$
    \text{the sequence $(|R(N_m')|)_{m\ge1}$ is unbounded but the sequence $(\al(N_m'))_{m\ge1}$ is bounded}\,.
$$
Then for every $r\in\mathbb{N}$ there is an $m$ such that 
$K_m$ contains an $r$-wealthy coloring of one and the same type $W_{3,1}$ or type $W_{3,2}$ or type $W_{4,2}$.
\end{lemma}
\begin{proof} Let $K_m = (n_m,\chi_m)$. We assume that all matrices $N_m'$ are above-diagonal and treat the other case when they are below-diagonal at the end. As we know, every $N_m'$ is binary.
Using Lemma~\ref{KlazarL3.13} to the sequence of matrices $(N_m')_{m \geq 1}$ and passing to a subsequence of $m=1,2,\dots$ and to submatrices of the corresponding matrices $N_m'$, we may suppose 
that every matrix $M_m$ has a submatrix $M_m'$ whose $(2, 3)$-d-cross-matrix $N_m'$ is above-diagonal and strongly similar to the $m\times m$ identity matrix 
$I_m$, or to the $m\times m$ upper triangular matrix $U_m$. We first assume that for every $m$ one has $N_m'=I_m$ and then we consider 
other matrices $N_m'$ strongly similar to $I_m$. The case of $N_m'$ strongly similar to $U_m$ is deferred to the end. Let 
$$
    \{r_{m,1}<\dots<r_{m,m}\}\subset X_m \ \text{and} \  \{c_{m,1}<\dots<c_{m,m}\},\{s_{m,1}<\dots<s_{m,m}\} \subset Y_m
$$
be the base sets of $N_m'$ and $M_m'$, so for any $i, j \in [m]$ we have $N_m'(i,j) = M_m'(r_{m,j},\,c_{m,i},\,s_{m,i})$. Moreover, $c_{m,i}<s_{m,i}$, since $N_m'$ is above diagonal. By the assumption,  $\chi_m(\{r_{m,j},c_{m,i},s_{m,i}\})=1$ if $i=j$ and $=0$ if $i\ne j$. We consider the numbers
$$
	T_m = \max\left(\left\{|L|:\ L\subset [m]\;\&\;\bigcap_{l \in L}(c_{m,l},\,s_{m,l}) \ne\emptyset \right\}\right)
$$
(here $(c_{m,l},\,s_{m,l})=\{x\in\mathbb{N}:\ c_{m,l}<x<s_{m,l}\}$). We distinguish the cases of unbounded and bounded sequence $(T_m)_{m\ge1}$, respectively. In the former case we see that
for every $r\in\mathbb{N}$ there is an $m$ such that $K_m$ contains a coloring $K_r'=(3r,\psi_r)$ satisfying for every $i,j\in[r]$ that
$$
    \psi_r(\{j,\,r+i,\,2r+i\})=\left\{
    \begin{array}{lll}
     1    & \dots & i=j\, ,  \\
     0    & \dots & i\ne j\, .
    \end{array}
    \right.
$$
This is an $r$-wealthy coloring of type $W_{3, 1}$. In the latter case $(T_m)_{m \geq 1}$ is bounded. Then for every $r\in\mathbb{N}$ there is an $m$ such that $K_m$ contains the coloring $K_r''=(6r,\theta_r)$ satisfying for every $i,j\in[2r]$ that
\begin{align}\notag
	\theta_r(\{j,\, 2r+2i-1,\, 2r+2i\}) = \left\{ \begin{array}{lll}
    	1 & \dots & i = j\, , \\
    	0 & \dots & i \not= j\, .
   	\end{array} \right.
\end{align}
We partition $[2r+1,6r]$ in intervals $S_1<S_2<\dots<S_r$ of length four each,
$$
    S_i = \{2r+4i-3,\,2r+4i-2,\,2r+4i-1,\,2r+4i\}\, . 
$$
For each $j\in[r]$ we take the three triples 
$$
    U_{j,k}=\{2j-1<2r+4j-3+k<2r+4j-2+k\},\ k=0,1,2\, . 
$$
Their respective $\theta_r$-colors are $1,?,0$. We take the two triples for 
$k=0,1$ or $k=1,2$ with different colors. Thus for every $j\in[r]$ there exist three distinct elements $a_j,b_j,c_j \in S_j$  such that 
$$
    \theta_r(\{2j-1,\,a_j,\,b_j\}) \not= \theta_r(\{2j-1,\,a_j,\,c_j\})\, .
$$
We see that $K_r''$ and $K_m$ contain an $r$-wealthy coloring of type $W_{4,2}$. 

If $N_m' \not= I_m$ is strongly similar to the identity matrix $I_m$, it arises from $I_m$ by exchanging $0$ and $1$ and/or replacing the main diagonal with 
the antidiagonal. The exchange of $0$ and $1$ has no effect on the resulting type $W_{4,2}$ coloring and we assume that $N_m=I_m'$ is the antidiagonal unit matrix.
Suppose that $(T_m)_{m\ge1}$ is unbounded. In the colorings $K_r'$ the order of elements in the interval $[r]$ is reversed but this leads again to
colorings of type $W_{3,1}$. If $(T_m)_{m\ge1}$ is bounded, in the colorings $K_r''$ the order of elements in the interval $[2r]$ is reversed. However, when we replace ``$2j-1$'' with ``$2r-2j+2$'' in the argument of the triples $U_{j,k}$, this leads again to
colorings of type $W_{4,2}$.

Suppose that for every $m$ the matrix $N_m'$ is strongly similar to the upper triangular matrix $U_m$. We first assume that $N_m'=U_m$ for every $m$ and 
modify the previous argument by replacing the phrases ``$i=j$'' and ``$i\ne j$'' with ``$i \leq j$'' and ``$i>j$'', respectively. If the sequence 
$(T_m)_{m \geq 1}$ is unbounded, we obtain by an argument analogous to the previous one colorings of type $W_{3, 2}$. If the sequence $(T_m)_{m \geq 1}$ is bounded, for every $r\in\mathbb{N}$ 
there is an $m$ such that $K_m$ contains the coloring $L_r=(6r,\phi_r)$ with the property that for every $i,j\in[2r]$, $\phi_r(\{j,2r+2i-1,2r+2i\})=1$ if 
$i\le j$ and $=0$ if $i>j$. Using the same argument with the triples $U_{j,k}$ as in the case with $N_m'=I_m$, we get $r$-wealthy colorings of type $W_{4,2}$. 

Suppose that $N_m'\ne U_m$ but $N_m'$ is strongly similar to $U_m$. The effect of exchange of $0$ and $1$ is again clear and therefore
we may assume that for every $m$ we have $N_m'=U_m'$, the matrix with $1$'s on the antidiagonal and above it and $0$s below it.
We modify the argument in the first part by replacing the phrases ``$i=j$'' and ``$i\ne j$'' with ``$i \leq d+1-j$'' and ``$i>d+1-j$'', respectively, 
where $d$ is the dimension of the square matrix in question, $d=m,r$ and $2r$, respectively. If the sequence $(T_m)_{m \geq 1}$ is unbounded, we obtain as before $r$-wealthy 
colorings of type $W_{3, 2}$.
If the sequence $(T_m)_{m \geq 1}$ is bounded, for every $r\in\mathbb{N}$ there is an $m$ such that $K_m$ contains the 
coloring $L_r'=(6r,\phi_r')$ with the property that for every $i,j\in[2r]$, $\phi_r'(\{j,2r+2i-1,2r+2i\})=1$ if $i\le 2r-j+1$ and $=0$ if $i>2r-j+1$. For each $j\in[r]$ the three triples 
$$
    U_{j,k}'=\{2j-1<6r-4j+3+k<6r-4j+4+k\},\ k=0,1,2\, , 
$$
with $\phi_r'$-colors $1,?,0$ show as before that for every $r\in\mathbb{N}$ for some $m$ the colorings
$L_r'$ and $K_m$ contain an $r$-wealthy coloring of type $W_{4,2}$. 

 The case when all matrices $N_m'$ are below-diagonal is handled by an argument almost identical to the previous one, only the phrases ``$c_{m,i}<s_{m,i}$'', 
 ``$(c_{m,l},s_{m,l})$'' and ``$c_{m,l}<x<s_{m,l}$'' have to be replaced with ``$s_{m,i}<c_{m,i}$'', ``$(s_{m,l},c_{m,l})$'' and ``$s_{m,l}<x<c_{m,l}$'', respectively.
\end{proof}

\subsection{Conclusion of the proof of Theorem~\ref{T1FD}}\label{sec_endof proof}
We finish the proof of Theorem~\ref{T1FD}. For $p\in \mathbb{N}$ a set $X$ of colorings is {\it $p$-tame} if every coloring 
$H\in X$ is $p$-tame (see Section~\ref{subsec_ptame} for the definition). If an ideal $X\subset\mathcal{C}_3$ is not $p$-tame for any $p$, then 
for every $p\in\mathbb{N}$ there is a coloring $H\in X$ such that its nuclear decomposition $\gr(H)$ violates one of 
the conditions 1--5 of $p$-tameness in the definition in Section~\ref{subsec_ptame} for this $p$. This clearly implies that there is a $c\in[5]$ such that for every $p\in\mathbb{N}$
there is a coloring $H\in X$ such that $\gr(H)$ violates condition $c$ of $p$-tameness for $p$.
We first look at violation of conditions 1--4 and then at the more complicated situation when condition 5 
is violated. As Propositions~\ref{P1} and \ref{P2} show, violation of one of the conditions 1--5 produces coloring of one of the four types (with subtypes) $W_1$--$W_4$. 
For better readability of the proofs we emphasize each obtained coloring of {\bf type $\mathbf{W_i}$} by the bold type. We also emphasize by the 
bold type the beginning of discussion of each {\bf first} or {\bf second} or {\bf third} $\dots$ case of the argument in either proof of the two propositions. 
For example, the proof of Proposition~\ref{P1} has the following cases.
\begin{enumerate}
    \item[]\hspace{2.05cm}{\bf Condition 1} is violated.
        \item[]\hspace{2.05cm}{\bf Condition 2} or {\bf condition 4}  is violated.
    \item[] 
    $$ 
    \mbox{\bf Condition 3} \left\{ \begin{array}{l}
        \mathbf{|R(\cdot)|} \mbox{\bf\ unbounded} \left\{
        \begin{array}{l}
            \mbox{{\bf conclusion (i)} holds}  \\
            \mbox{{\bf conclusion (ii)} holds} 
        \end{array} \right. \\
        \mathbf{|C(\cdot)|} \mbox{\bf\ unbounded}
    \end{array} \right.
    $$
\end{enumerate}
The proof of Proposition~\ref{P2} has more complicated structure of cases.

It is true that if $M$ is a binary three-dimensional matrix and $M'\preceq M$ then $\al(M')\le\al(M)$. For $*$-binary matrices this inequality in general 
does not hold, but fortunately a modified version which we use does hold: if $M=M_{I,J,K}$ is a crossing matrix of a coloring with the 
base sets $I=J<K$ or $I<J=K$ and $M'\preceq M$, then $\al(M')\le\al(M)+1$. This holds because every line of $M$ contains at most one $*$ or consists only of $*$s.

\begin{proposition}\label{P1}
Let $c\in[4]$ and let $X\subset\mathcal{C}_3$ be an ideal of colorings such that for every $p\in\mathbb{N}$ there is a coloring $H\in X$ whose nuclear decomposition
$\gr(H)$ violates condition $c$ of $p$-tameness for $p$. Then there is an $i\in [4]$ such that for every $r\in\mathbb{N}$ the ideal $X$ contains 
an $r$-wealthy coloring of type $W_i$.
\end{proposition}

\begin{proof}
We consider three cases.

{\bf Condition 1} is violated. Then for every $r\in\mathbb{N}$ there is a coloring $H\in X$ such that the length of its nuclear decomposition 
is at least $2r$. As we noted in Section~\ref{subsec_ptame}, $H\succeq K$ where $K$ is an $r$-wealthy coloring of type $W_{4,1}$. Therefore for every $r\in\mathbb{N}$ 
the ideal $X$ contains an $r$-wealthy coloring of {\bf type $\mathbf{W_{4,1}}$}.

{\bf Condition 2} is violated or {\bf condition 4} is violated. If condition $2$ is violated then for every $r\in \mathbb{N}$ there is a coloring 
$H\in X$ that has a binary crossing matrix $M=M_{I,J,K}$, where $I<J<K$ are three intervals in $\gr(H)$, with $\al(M)\ge r$. It is easy to see that then $H\succeq K$  
where $K$ is an $(r+2)$-wealthy coloring of type $W_1$. Hence for every $r\in\mathbb{N}$ the ideal $X$ contains an $r$-wealthy coloring of {\bf type $\mathbf{W_1}$}. 
If condition $4$ is violated, the only difference is that the line in the crossing matrix $M$ witnessing $\al(M)\ge r$ may contain one $*$. Again,  
for every $r\in\mathbb{N}$ the ideal $X$ contains an $(r+2)$-wealthy coloring of {\bf type $\mathbf{W_1}$}.

{\bf Condition 3} is violated. We assume that condition $2$ is not violated, which means that there is a constant $p_0$ such that for every coloring $H\in X$ 
and every three intervals $I<J<K$ in $\gr(H)$ we have $\al(M_{I,J,K})\le p_0$. At the same time there exists a sequence 
$(H_m)_{m\ge 1}$ of coloring $H_m = (n_m, \chi_m)$ in $X$ with three intervals $I_m < J_m < K_m$ in $\gr(H_m)$ such that, 
if we denote by $M_m = M_{I_m, J_m, K_m}$ the corresponding crossing matrix, the sequence $(\mathrm{al}(M_m))_{m\geq 1}$ is bounded by $p_0$ but one of 
the sequences $R = (|R(M_m)|)_{m \geq 1}$ and $C = (|C(M_m)|)_{m \geq 1}$ is unbounded. Each matrix $M_m = M_{I_m, J_m, K_m}$ is non-intertwined and so are 
the matrices below derived from $M_m$ and we may apply to them Corollary~\ref{cor_prvniholemm} and Lemma~\ref{lema_dvaapul}.
First we suppose that the {\bf sequence $\mathbf{R}$ is unbounded} and defer 
the case when $C$ is unbounded to the end. Thus the hypothesis of Lemma~\ref{3DKlazar} is satisfied and its conclusion
(i) or (ii) holds. 

If the {\bf conclusion (i) of Lemma~\ref{3DKlazar}} holds then every matrix $M_m$ has a layer matrix $N_m$ such that the sequence $(|R(N_m)|)_{m\geq 1}$ is unbounded. Note that 
the sequence $(\al(N_m))_{m\geq 1}$ is bounded. Therefore by Corollary~\ref{cor_prvniholemm} (with $M_m'=M_m$) for every $r\in\mathbb{N}$ there is an $m$ such that the coloring $H_m$ contains an $r$-wealthy coloring of {\bf type $\mathbf{W_2}$}.

If the {\bf conclusion (ii) of Lemma~\ref{3DKlazar}} holds then every matrix $M_m$ has a submatrix $M_m'$ that has a cross-matrix $N_m$ such that the sequence 
$(|R(N_m)|)_{m\geq 1}$ is unbounded. Also, the sequences $(\al(M_m))_{m\ge1}$ and $(\al(M_m'))_{m\ge1}$ are bounded. We apply Lemma~\ref{lema_dvaapul} to the matrices 
$N_m$ and conclude that for every $r\in\mathbb{N}$ there is an $m$ such that the coloring $H_m$ contains an $r$-wealthy coloring of one and the same type 
{\bf type $\mathbf{W_2}$} or {\bf type $\mathbf{W_{3,1}}$}  or {\bf type $\mathbf{W_{3,2}}$}.

If the {\bf sequence $\mathbf{C}$ is unbounded,} we use Lemma~\ref{3DKlazar2} instead of Lemma~\ref{3DKlazar}. Now the only difference is that in conclusion (i) of Lemma~\ref{3DKlazar2} we have unbounded sequence $(|C(N_m)|)_{m \geq 1}$ instead of $(|R(N_m)|)_{m \geq 1}$. But Corollary~\ref{cor_prvniholemm} applies in this case too. If conclusion (ii) holds then we again use Lemma~\ref{lema_dvaapul}. In total we get again that for every $r\in\mathbb{N}$ there is an $m$ such that the coloring 
$H_m$ contains an $r$-wealthy coloring of one and the same {\bf type $\mathbf{W_2}$} or {\bf type $\mathbf{W_{3,1}}$}  or {\bf type $\mathbf{W_{3,2}}$}.
\end{proof}

The last case is when condition 5 is violated and the previous proposition does not apply. 

\begin{proposition}\label{P2}
Let $X\subset\mathcal{C}_3$ be an ideal of colorings such that for some $p_0\in\mathbb{N}$ for every coloring $H\in X$ it holds that $\gr(H)$ 
satisfies each of the conditions 1--4 of $p$-tameness for $p=p_0$, but for every $p\in\mathbb{N}$ there is a coloring $H\in X$ such that $\gr(H)$ 
violates condition 5  of $p$-tameness for $p$. Then there is an $i\in [4]$ such that for every $r\in \mathbb{N}$ the ideal $X$ contains an $r$-wealthy 
coloring of type $W_i$.
\end{proposition}

\begin{proof}
Let $X$ be as stated. It follows that there is a sequence of colorings $H_m=(n_m,\chi_m)\in X$, $m\in\mathbb{N}$, and intervals $I_m < J_m$ in $\gr(H_m)$ such 
that one of the cases (a)--(d) holds.

\begin{itemize}
    \item[(a)] $M_m := M_{I_{m}, I_{m}, J_{m}}$, the sequence $(\al(M_m))_{m\ge 1}$ is bounded and the sequence $(|R(M_m)|)_{m\ge 1}$ is unbounded.
    \item[(b)] $M_m := M_{I_{m}, I_{m}, J_{m}}$, the sequence $(\al(M_m))_{m\ge 1}$ is bounded and the sequence $(|C(M_m)|)_{m\ge 1}$ is unbounded.
    \item[(c)] $M_m := M_{I_{m}, J_{m}, J_{m}}$, the sequence $(\al(M_m))_{m\ge 1}$ is bounded and the sequence $(|R(M_m)|)_{m\ge 1}$ is unbounded.
    \item[(d)] $M_m := M_{I_{m}, J_{m}, J_{m}}$, the sequence $(\al(M_m))_{m\ge 1}$ is bounded and the sequence $(|C(M_m)|)_{m\ge 1}$ is unbounded.
\end{itemize}
It is in fact unimportant that $I_m < J_m$ are two intervals in $\gr(H_m)$, it suffices to assume that they are just two subsets in $[n_m]$. We prove then that
in each of the cases (a)--(d) there is an $i\in[4]$ such that for every $r\in\mathbb{N}$ there is an $m$ such that $H_m\succeq K_m$ for an $r$-wealthy 
coloring $K_m$ of type $W_i$. 

{\bf Cases (b) and (d) reduce to case (a)}. Before we begin with case (a) we show that cases (b) and (d) reduce to it. Indeed, the matrices 
$M_m = M_{I_{m}, I_{m}, J_{m}}$ are symmetric in the first two coordinates, which means that for every $m$ every $3$-layer matrix of $M_m$ is a symmetric 
(two-dimensional) matrix and $C(M_m)=R(M_m)$. So cases (a) and (b) are the same. Suppose that the colorings $H_m$ and sets $I_m<J_m$ satisfy conditions of case (d). The reversing map $r_m\colon[n_m]\to[n_m]$ given by 
$r_m(x)=n_m-x+1$ turns each coloring $H_m=(n_m,\chi_m)$ in its reversal
$H_m'=(n_m,\chi_m')$, $\chi'=\chi\circ r_m$, given 
by 
$$
    \chi_m'(\{a,\,b,\,c\})=\chi_m(\{n_m-a+1,\,n_m-b+1,\,n_m-c+1\})\, .
$$ 
Each matrix 
$M_m = M_{I_{m}, J_{m}, J_{m}}$ then turns in the matrix $M_m'=M_{I_m',I_m', J_m'}$ where $I_m'=r_m(J_m)$ and $J_m'=r_m(I_m)$. Then 
$$
    \al(M_m')=\al(M_m)\;\text{ and }\;|C(M_m')|=|R(M_m')|=|C(M_m)|\,.
$$
Thus the colorings $H_m'$ and sets $I_m'<J_m'$ produce matrices $M_m'$ satisfying conditions of case (a). Assuming that we have solved it, we have an $i\in[4]$ such that for every 
$r\in\mathbb{N}$ there is an $m$ such that $H_m'\succeq K_m$ for an $r$-wealthy 
coloring $K_m$ of type $W_i$. Since the family of wealthy colorings of type $W_i$ is closed to reversals and $H_m$ arises from $H_m'$ as its reversal (and $\succeq$ is preserved by reversals), 
$H_m\succeq K_m$ too. We see that an $r$-wealthy coloring of
the same {\bf type $\mathbf{W_i}$} is contained in $H_m$ as well and case (d) is solved.

\medskip{}
{\bf Case (a)}. Recall that $M_m = M_{I_{m}, I_{m}, J_{m}}$ for two subsets $I_m<J_m$ in $[n_m]$, the sequence $(\al(M_m))_{m \geq 1}$ is bounded and the sequence $(|R(M_m)|)_{m \geq 1}$ is unbounded. The hypothesis of Lemma~\ref{3DKlazar} is satisfied and its conclusion (i) or (ii) holds.

Suppose that {\bf conclusion (i) of Lemma~\ref{3DKlazar}} holds. Now, unlike in the previous proposition, we have to treat its two subcases differently because the second base set of $M_m$ coincides with the first one but the third base set lies after the first one. The {\bf first subcase} is that every matrix $M_m$ has a $2$-layer $*$-binary matrix $N_m$, where $N_m$ is a matrix that has one column of $*$s and elsewhere only $0$s and $1$s, such that the sequence $(\al(N_m))_{m \geq 1}$ is bounded but the sequence $(|R(N_m)|)_{m \geq 1}$ is unbounded. The column of $*$s splits $N_m$ in two submatrices 
$N_m'$ and $N_m''$ which are non-intertwined because each lies on only one side of the column of $*$s. For one of them, say $N_m'$, we have that $|R(N_m')|\ge\frac{1}{2}|R(N_m)|$. The other case with $N_m''$ is similar. Thus the sequence $(\al(N_m'))_{m\geq 1}$ is bounded but the sequence $(|R(N_m')|)_{m \geq 1}$ is unbounded. By part 1 of Lemma~\ref{subm_layer_crossm}, each matrix $N_m'$ is also a layer matrix of a submatrix of $M_m$. We apply Corollary~\ref{cor_prvniholemm} to the matrices $N_m'$ and
get that for every $r\in\mathbb{N}$ there is an $m$ such that the coloring $H_m$ contains an $r$-wealthy coloring of {\bf type $\mathbf{W_2}$}.

The {\bf second subcase} in conclusion (i) is that every matrix $M_m$ has a $3$-layer $*$-binary matrix $N_m$ such that the sequence $(\al(N_m))_{m \geq 1}$ is bounded but the sequence 
$$
    (|R(N_m)|)_{m \geq 1}=(|C(N_m)|)_{m \geq 1}
$$ 
is unbounded. Here even $R(N_m)=C(N_m)$ because each $N_m$ is a square 
symmetric matrix with $*$s on the diagonal and $0$s and $1$s elsewhere (the first two base sets of $M_m$ and $N_m$ are the same). We apply Proposition~\ref{new_variation_onKlazar} to the matrices 
$(N_m)_{m \geq 1}$ and get that its case ($\alpha$) or ($\beta$) holds. {\bf Case ($\alpha$)} implies that for every
$r\in\mathbb{N}$ there is an $m$ such that $H_m$ contains an $r$-wealthy coloring of {\bf type $\mathbf{W_{3, 3}}$} (determined by the fixed vertex 
corresponding to the fixed third coordinate and by an $r$-wealthy coloring of pairs of type $2$ that precedes it). In {\bf case ($\beta$)} we can use Lemma~\ref{poz_W34} 
because the submatrices $I_r'$ and $U_r'$ of $N_m$ produced by case ($\beta$) are non-intertwined (they lie above the diagonal of 
$N_m$) and each is also a layer matrix of a submatrix of $M_m$, by part 1 of Lemma~\ref{subm_layer_crossm}. We get that 
for every $r\in\mathbb{N}$ there is an $m$ such that the coloring $H_m$ contains an $r$-wealthy coloring of {\bf type $\mathbf{W_2}$}.

Suppose that {\bf conclusion (ii) of Lemma~\ref{3DKlazar}} holds. Then every matrix $M_m=M_{I_m,I_m,J_m}$ has a submatrix $M_m'$ that has a $*$-binary $R$-full $(2,3)$-cross-matrix 
$N_m$ such that the sequence $(|R(N_m)|)_{m \geq 1}$ is unbounded. For concreteness we suppose that every matrix {\bf $\mathbf{N_m}$ is 
a $\mathbf{(2,3)}$-d-cross-matrix} of $M_m'$, which means that  $N_m(i,j)=M_m'(j,i,i)$, and postpone the other case with $(2,3)$-ad-cross-matrix to the end of case (a). $R$-fullness of the matrices $N_m$ implies that if $N_m$ has $r(m)$ rows 
then there are $r(m)$ distinct column indices $p_{m,i}$ such that for each $i=1,2,\dots,r(m)$ we have 
$$
    \{N_m(i,\,p_{m,i}),\,N_m(i,\, p_{m,i}+1)\}=\{0,1\}\, .
$$
In the sense of Remark~\ref{rem:fullrow}, each row of $N_m$ is a row in $M_m$. Thus each row of $N_m$, as a word over $\{0,1,*\}$, has exactly one $*$. Also, the numbers of 
subwords $01$ and $10$ in the rows of $N_m$ are bounded as a function of $m$. The last fact implies, since the sequence $(|R(N_m)|)_{m \geq 1}$ is unbounded, 
that the sequence $(r(m))_{m \geq 1}$ of numbers of rows in the matrices $N_m$ is unbounded. Let the column index of the $*$ in the row $i$ of $N_m$ be $q_{m,i}$. 
We note that $q_{m,1}<q_{m,2}<\dots<q_{m,r(m)}$. Either $p_{m,i} > q_{m,i}$ or $p_{m,i} < q_{m,i}$ holds for at least half of the indices $i \in [r(m)]$. Note that $p_{m,i}=q_{m,i}$ is impossible because the left side is a column index of a $0$ or a $1$
and the right side is a column index of a $*$. 

We assume that the {\bf former inequality  $\mathbf{p_{m,i}>q_{m,i}}$} holds for at least half of the rows and later indicate how to 
deal with the latter inequality. Using again the Erd\H{o}s--Szekeres lemma we may suppose, by passing to a subsequence of the sequence $m=1,2,\dots$ and to 
submatrices of the corresponding matrices $N_m$, that for every $m \in \mathbb{N}$ the sequence 
$$
    P(m) = \left(p_{m,1},\,p_{m,2},\,\dots,\,p_{m,r(m)}\right)
$$ 
is decreasing, or that for every $m \in \mathbb{N}$ it increases. 

If the {\bf sequences $\mathbf{P(m)}$ decreases},
$p_{m,1}>p_{m,2}>\dots>p_{m,r(m)}$, we consider the submatrix $N_m'$
of $N_m$ consisting of the whole columns with indices $\ge p_{m,r(m)}$. The matrix $N_m'$ is non-intertwined because it lies to the right of all $*$s in $N_m$. 
By part 2 of Lemma~\ref{subm_layer_crossm}, $N_m'$ is also a cross-matrix of a submatrix $M_m''$ of $M_m'$. Clearly, the sequence $(\al(M_m''))_{m\ge1}$
is bounded (because $(\al(M_m))_{m\ge1}$ is bounded) and the sequence $(|R(N_m')|)_{m\geq 1}$ is unbounded. We apply Lemma~\ref{lema_dvaapul} to the
matrices $N_m'$ and conclude that for every $r\in\mathbb{N}$ there is an $m$ such that the coloring $H_m$ contains an $r$-wealthy coloring of one and the
same {\bf type $\mathbf{W_2}$} or {\bf type $\mathbf{W_{3,1}}$} or {\bf type $\mathbf{W_{3,2}}$}.

If the {\bf sequences $\mathbf{P(m)}$ increases}, $p_{m,1}<p_{m,2}<\dots<p_{m,r(m)}$, we define for each $m$ an increasing sequence 
$$
    1=t_{m,1}<t_{m,2}<\dots<t_{m,s(m)}
$$
of row indices as follows. Recall that in the row $i$, $p_{m,i}$ is the column index of the first letter of a subword $01$ or $10$ and $q_{m,i}<p_{m,i}$ and is the column index
of the unique $*$. For $i>1$ we let $t_{m,i}$ be the minimum index $j$ such that $j>t_{m,i-1}\;\&\,q_{m,j}> p_{m, t_{m,i-1}}+1$ if such $j$ exists,
and set $t_{m,i}$ to be the last row index in $N_m$ otherwise; thus $t_{m,s(m)}=r(m)$ is the number of rows in $N_m$. The indices $t_{m,i}$ are illustrated 
by Figure~\ref{fig:rpq}. We consider two cases, of bounded and unbounded sequence $(s(m))_{m\ge1}$ of the numbers of the indices $t_{m,i}$, respectively. 

\begin{figure}[t!]
    \centering
    $
    \begin{matrix}
        t_{m,1} \\
            \\
            \\
        t_{m,2} \\
            \\
          \\
         t_{m,3} \\
    \end{matrix}
        \begin{pmatrix}
            * & \mathbf{0} & 1 & 0 & 0 & 1 & 1 & 1 & 1 & 0 \\
            1 & * & \mathbf{1} & 0 & 0 & 0 & 1 & 0 & 0 & 0 \\
            0 & 0 & * & \mathbf{1} & 0 & 1 & 1 & 1 & 0 & 0 \\
            1 & 0 & 1 & * & 0 & \mathbf{0} & 1 & 1 & 1 & 1 \\
            1 & 0 & 1 & 1 & * & 0 & \mathbf{0} & 1 & 1 & 1 \\
            1 & 1 & 1 & 1 & 1 & 1 & * & \mathbf{1} & 0 & 0 \\
            0 & 0 & 0 & 0 & 1 & 1 & 1 & * & \mathbf{0} & 1 
        \end{pmatrix}
    $
\caption{An example of a matrix $N_m$ and row indices $t_{m,i}$ defined in the proof of Proposition \ref{P2}. The positions $(i,p_{m,i})$ are indicated by
the bold type.}
\label{fig:rpq}
\end{figure}

We assume that  the {\bf sequence $\mathbf{(s(m))_{m\ge1}}$ is bounded}. Then for every $m$ there is an index $l(m)\in[r(m)-1]$ such that the sequence
$$
    (u_m)_{m\ge1}:=\left(t_{m,l(m)+1}-t_{m,l(m)}-2\right)_{m\ge1}
$$
is unbounded, where the numbers $u_m$ counts rows in $N_m$ between the $(t_{m,l(m)}+1)$-th and $(t_{m,l(m)+1})$-th row. 
We denote again by $N_m'$ the  $R$-full submatrix of $N_m$ formed by the intersection of these $u_m$ rows and the columns lying to the right of the $(p_{m,l(m)}+1)$-th column. 
Because of the definition of the indices $t_{m,i}$ and since $P(m)$ increases, we see as for the above matrices $N_m'$ that each
of the matrices $N_m'$ here is non-intertwined. Also, the sequence $(|R(N_m')|)_{m\ge1}$ is unbounded because $|R(N_m')|\ge u_m$. 
We argue as for the above matrices $N_m'$ and conclude by part 2 of Lemma~\ref{subm_layer_crossm} and by Lemma~\ref{lema_dvaapul} that for every $r\in\mathbb{N}$ 
there is an $m$ such that the coloring $H_m$ contains an $r$-wealthy coloring of one and the same {\bf type $\mathbf{W_2}$} or {\bf type $\mathbf{W_{3,1}}$} or {\bf type $\mathbf{W_{3,2}}$}.

If the {\bf sequence $\mathbf{(s(m))_{m\ge1}}$ is unbounded} we consider for each $m$ the submatrix of $N_m$ denoted again by $N_m'$ and formed by the intersection of 
$$
    \text{the rows $t_{m,j}$ and the columns $q_{m,t_{m,j}},\;p_{m,t_{m,j}}$ and $p_{m,t_{m,j}}+1$ for $j=1,\,2,\,\dots,\,s(m)-1$}\, .
$$
Let $r'=r'(m)=s(m)-1$. The $r'\times 3r'$ matrix $N_m'$ has for $i\in[r']$ the entries $N_m'(i, 3i-2) = *$ and $\{N_m'(i, 3i-1),\,N_m'(i, 3i)\}=\{0,1\}$. 
By part 1 of Lemma~\ref{treti_lemma}, for every $r'\in\mathbb{N}$ there is an $m$ such that the coloring $H_m$ contains an $r'$-wealthy coloring of {\bf type $\mathbf{W_{4, 2}}$}.

We consider the {\bf case when $\mathbf{p_{m,i} < q_{m,i}}$} holds for at least half of the row indices $i$ of $N_m$. We again use the Erd\H{o}s--Szekeres lemma, pass to a subsequence of indices $m$ and to submatrices of the corresponding matrices $N_m$ 
and again distinguish two cases depending on whether the sequence 
$$
P(m) = \left(p_{m,1},\,p_{m,2},\,\dots,\,p_{m,r(m)}\right)
$$ 
decreases or increases. In the {\bf case when $\mathbf{P(m)}$ decreases}, when $p_{m,1}>p_{m,2}>\dots>p_{m,r(m)}$, we consider the submatrix $N_m'$
of $N_m$ formed by the columns of $N_m$ with indices $\le p_{m,1}$. The matrix $N_m'$ is non-intertwined because it lies to the left of all $*$s in $N_m$. We argue as before and by means of part 2 of Lemma~\ref{subm_layer_crossm} and Lemma~\ref{lema_dvaapul}  conclude that for every $r\in\mathbb{N}$ there is an $m$ such that the coloring $H_m$ contains an $r$-wealthy coloring of one and the
same {\bf type $\mathbf{W_2}$} or {\bf type $\mathbf{W_{3,1}}$} or {\bf type $\mathbf{W_{3,2}}$}.

In the {\bf case when $\mathbf{P(m)}$ increases}, when $p_{m,1}<p_{m,2}<\dots<p_{m,r(m)}$,
we again define certain row indices of $N_m$ denoted again $t_{m,i}$,
$$
    1=t_{m,1}<t_{m,2}<\dots<t_{m,s(m)}\,,
$$ 
as follows. For $i>1$  we set $t_{m,i}$ to be the minimum $j$ such that $j>t_{m,i-1}$ and $p_{m,j}+1 > q_{m,t_{m,i-1}}$ if such 
$j$ exists, and set $t_{m,i}$ to be the last row index in $N_m$ else. Thus again $t_{m,s(m)}=r(m)$ is the number of rows of $N_m$. As before we consider 
two cases depending on whether the sequence $(s(m))_{m\ge1}$ is bounded. In the {\bf case when $\mathbf{(s(m))_{m\ge1}}$ is bounded}, for some indices 
$l(m)\in[r(m)-1]$ the difference 
$$
    t_{m,l(m)+1}-t_{m,l(m)}-1
$$
is unbounded in $m$ and like before we consider the $R$-full
submatrix $N_m'$ of $N_m$ formed by the intersection of this many rows lying between the $(t_{m,l(m)}-1)$-th and $(t_{m,l(m)+1}-1)$-th row and the columns lying 
to the left of the $q_{m,t_{m,l(m)}}$-th column. Then as before the matrix $N_m'$ is non-intertwined because it lies to the left of all $*$s in $N_m$ and $|R(N_m')|$ is at least the number of rows in $N_m'$ and hence is unbounded in $m$. We again conclude 
by means of part 2 of Lemma~\ref{subm_layer_crossm} and  Lemma~\ref{lema_dvaapul} that for every $r\in\mathbb{N}$ 
there is an $m$ such that the coloring $H_m$ contains an $r$-wealthy coloring of one and the same {\bf type $\mathbf{W_2}$} or {\bf type $\mathbf{W_{3,1}}$} or {\bf type $\mathbf{W_{3,2}}$}. 

If the {\bf sequence $\mathbf{(s(m))_{m\ge1}}$ is unbounded} we again consider the matrix $N_m'$ formed by the intersection of $$\text{the rows $t_{m,j}$ and the columns  $p_{m,t_{m,j}}$, $p_{m,t_{m,j}}+1$ and $q_{m,t_{m,j}}$ for $j=1,\,2,\,\dots,\,s(m)-1$}\,.$$ This is an $r'\times 3r'$ matrix, where $r'=r'(m)=s(m)-1$, and for every
$i\in[r']$ we have that $\{N_m'(i,3i-2),N_m'(i,3i-1)\}=\{0,1\}$ and $N_m'(i,3i)=*$. By part 2 of Lemma~\ref{treti_lemma}, 
for every $r'\in\mathbb{N}$ there is an $m$ such that the coloring 
$H_m$ contains an $r'$-wealthy coloring of {\bf type $\mathbf{W_{4, 2}}$}.

The last case of case (a) to consider is when conclusion (ii) of Lemma~\ref{3DKlazar} holds and each matrix  {\bf $\mathbf{N_m}$ is a 
$\mathbf{(2,3)}$-ad-cross-matrix} of a submatrix $M_m'$ of the matrix $M_m=M_{I_m,I_m,J_m}$ which is a crossing matrix of the coloring $H_m$; 
we know that the sequence $(\al(M_m))_{m\ge 1}$ is bounded and the sequence $(|R(N_m)|)_{m\ge 1}$ is unbounded. We use a finer version of 
the argument by which we above reduced case (d) to case (a). We reduce the case of $(2,3)$-ad-cross-matrix $N_m$ to the just resolved case of $(2,3)$-d-cross-matrix $N_m$.
Instead of global reversal of colorings we will use local reversals. Namely, we replace matrices $N_m$, $M_m'$ and $M_m$ and colorings $H_m$ with the matrices 
$F(N_m)$, $F(M_m')$ and $F(M_m)$ and coloring $H_m'$, respectively,
as follows. $F$ is the mapping transforming a general three-dimensional $*$-binary matrix 
$M\colon[r]\times[s]\times[t]\to\{0,1,*\}$ to the matrix $F(M)\colon[r]\times[s]\times[t]\to\{0,1,*\}$ 
given by
$$
    F(M)(a,\,b,\,c)=M(a,\,b,\,t-c+1)\, .
$$
If $M=M_{I,I,J}$ is a crossing matrix of a coloring $H=(n,\chi)$, where $I,J\subset[n]$ with $I<J$, then $F(M)=M_{I,I,J}$ is the crossing matrix of the
coloring $H'=(n,\chi')$ (for the same sets $I$ and $J$) that arises from $H$ by reversing the order of elements in the subset $J$. 
In more details, if $J=\{y_1<y_2<\dots<y_t\}\subset[n]$ and for an $x\in[n]$ we denote $\overline{x}=x$ if $x\in[n]\setminus J$ and $\overline{x}=\overline{y_j}=y_{t-j+1}$ 
if $x=y_j\in J$ then for every $\{x,y,z\}\in\binom{[n]}{3}$ also $\{\overline{x},\overline{y},\overline{z}\}\in\binom{[n]}{3}$, so no new $*$ is created, and
$$
\chi'(\{x,\,y,\,z\}):=\chi(\{\overline{x},\,\overline{y},\,\overline{z}\})\,.
$$
The mapping $F$ reverses orders of positions in shafts in $M$ but preserves orders of positions in rows and columns and only changes their places. $F$ also
transforms the $(2,3)$-ad-cross-matrix $N$ of a submatrix $M'$ of $M$ to the $(2,3)$-d-cross-matrix $F(N)$ of the submatrix $F(M')$ of $F(M)$ 
(in the definitions of $F(N)$, $F(M')$ and $F(M)$ a position $(a,b,c)$ in $F(M)$ becomes position $(a,b,t-c+1)$ in $M$). 
We have $\al(F(M))=\al(M)$ and even $R(F(N))=R(N)$. We can apply to the matrices $F(N_m)$, $F(M_m')$ and $F(M_m)$ and colorings $H_m'$ the previously resolved case of $(2,3)$-d-cross-matrix. We deduce that there is a symbol $i_0$ which is `$2$' or `$3,1$' or `$3,2$' or `$4,2$' such that for every $r\in\mathbb{N}$ there is an $m$ such that $H_m'\succeq K_m$ where $K_m$ is an $r$-wealthy coloring of type $W_{i_0}$. 

We claim that then also $H_m\succeq K_m$. We prove it by verifying that each time $K_m\preceq H_m'$ in the previous argument, the transformation $F^{-1}=F$ yields $K_m\preceq H_m$. We have $K_m\preceq H_m'$ either by four applications of Lemma~\ref{lema_dvaapul} or by two applications of Lemma~\ref{treti_lemma}. By the statements of Lemmas~\ref{lema_dvaapul} and \ref{treti_lemma}, each base set of $K_m$ is mapped in the containment $K_m\preceq H_m'$ in one of the base sets of $F(M_m)$ which are $I_m,I_m$ and $J_m$. Since $H_m$ is obtained back from $H_m'$ by reversing the order of elements in the interval $J_m$ and each family of type $W_{i_0}$ colorings is closed to reversing  the order of elements in any base set, we finally see that $H_m\succeq K_m$ too and $H_m$ contains an $r$-wealthy coloring of {\bf type $\mathbf{W_{i_0}}$}.

\medskip
{\bf Case (c)}. Recall that $M_m = M_{I_{m}, J_{m}, J_{m}}$ for two subsets $I_m<J_m$ in $[n_m]$, the sequence 
$(\al(M_m))_{m \geq 1}$ is bounded and the sequence $(|R(M_m)|)_{m \geq 1}$ is unbounded. Thus, again, the hypothesis of Lemma~\ref{3DKlazar} is satisfied and its conclusion (i) or (ii) holds.

Suppose that {\bf conclusion (i) of Lemma~\ref{3DKlazar}} holds. Thus every matrix $M_m$ has a $2$-layer or a $3$-layer matrix $N_m$, where $N_m$ is a two-dimensional $*$-binary matrix with one row of $*$s and $0$s and $1$s elsewhere, such that the sequence 
$(\al(N_m))_{m \geq 1}$ is bounded but the sequence $(|R(N_m)|)_{m \geq 1}$ is unbounded. Since the second and third base set of $M_m$ are the same, we may assume that $N_m$ is a $3$-layer matrix of $M_m$. The row of $*$s splits $N_m$ in binary matrices $N_m'$ and $N_m''$ and for one of them, say $N_m'$, we have $|R(N_m')|\ge\frac{1}{2}|R(N_m)|$, thus $(N_m')_{\ge1}$ is unbounded. In general, the matrix $N_m$ may not be non-intertwined. However, the splitting of $N_m$ enforces, as in case (a), that both $N_m'$ and $N_m''$ are non-intertwined. Thus we have a sequence $(N_m')_{\ge1}$ of two-dimensional binary
matrices such that the sequence $(\al(N_m'))_{m\geq 1}$ is bounded but the sequence $(|R(N_m')|)_{m \geq 1}$ is unbounded. We conclude as before by part 1 of Lemma~\ref{subm_layer_crossm} and Corollary~\ref{cor_prvniholemm}
that for every $r\in\mathbb{N}$ there is an $m$ such that the coloring $H_m$ contains an $r$-wealthy coloring of {\bf type $\mathbf{W_2}$}. The argument is the same when $|R(N_m'')|\ge\frac{1}{2}|R(N_m)|$.

Suppose that {\bf conclusion (ii) of Lemma~\ref{3DKlazar}} holds. This means that every matrix $M_m$ has a submatrix $M_m'$ that has an $R$-full $(2, 3)$-cross-matrix $N_m$ 
such that the sequence $(|R(N_m)|)_{m \geq 1}$ is unbounded. As before we suppose for concreteness that every matrix {\bf $\mathbf{N_m}$ is a $\mathbf{(2,3)}$-d-cross-matrix} 
of $M_m'$ and defer $(2,3)$-ad-cross-matrices to the end of case (c). We partition $N_m$ by its rows in two submatrices $N_m'$ and $N_m''$ (see Figure~\ref{fig:cross_c}). 
By Remark~\ref{rem:fullrow} the rows in matrices $N_m$, $M_m'$ and $M_m$ are the same. If 
$$
\{c_{m,\,1}<c_{m,\,2}<\dots<c_{m,\,p(m)}\}\subset J_m\;\text{ and }\;\{s_{m,\,1}<s_{m,\,2}<\dots<s_{m,\,p(m)}\}\subset J_m
$$
is the second and third base set of $M_m'$, respectively, with $p(m)$ being the common second and third dimension of $M_m'$, then the $i$-th row of $N_m$ is
$$
    N_m(i,\,j) = M_m(j,\,c_{m,\,i},\,s_{m,\,i})\;.
$$
We define the submatrix $N_m'$ (resp. $N_m''$) as consisting of the rows $i$ of $N_m$ for 
which $c_{m,i} < s_{m,i}$ (resp. $c_{m,i} > s_{m,i}$). We cannot have $c_{m,i}=s_{m,i}$, then the $i$-th row of $N_m$ 
would contain only $*$s and this contradicts $R$-fullness of $N_m$. In the sense of the definition preceding Lemma~\ref{ctvrte_lemma} the matrix $N_m'$ is above-diagonal and the matrix $N_m''$ below-diagonal. Clearly, $N_m$, $N_m'$ and $N_m''$ are binary matrices but in general are not non-intertwined. However, one of the sequences $(|R(N_m')|)_{m \geq 1}$ and $(|R(N_m'')|)_{m \geq 1}$ is unbounded. 

We first suppose that the {\bf sequence $\mathbf{(|R(N_m')|)_{m \geq 1}}$ is unbounded} and discuss the other case of unbounded sequence $(|R(N_m'')|)_{m \geq 1}$ later. 
We consider two subcases. If the {\bf sequence $\mathbf{(\al(N_m'))_{m \geq 1}}$ is unbounded}, by part 2 of Lemma~\ref{subm_layer_crossm} we can apply to the matrices $N_m'$ Lemma~\ref{ctvrte_lemma} and conclude that for every $r\in\mathbb{N}$ there is an $m$ such that the coloring $H_m$ contains an $r$-wealthy coloring of 
one and the same {\bf type $\mathbf{W_{3,3}}$} or {\bf type $\mathbf{W_2}$}. If the {\bf sequence $\mathbf{(\al(N_m'))_{m \geq 1}}$ is bounded} then by part 2 of Lemma~\ref{subm_layer_crossm} we can apply to the
matrices $N_m'$ Lemma~\ref{pate_lemma} and conclude that for every $r\in\mathbb{N}$ there is an $m$ such that the coloring $H_m$ contains an $r$-wealthy coloring of one 
and the same {\bf type $\mathbf{W_{3,1}}$} or {\bf type $\mathbf{W_{3,2}}$} or {\bf type $\mathbf{W_{4,2}}$}.

\begin{figure}
    \centering
    \includegraphics[width=0.8\textwidth]{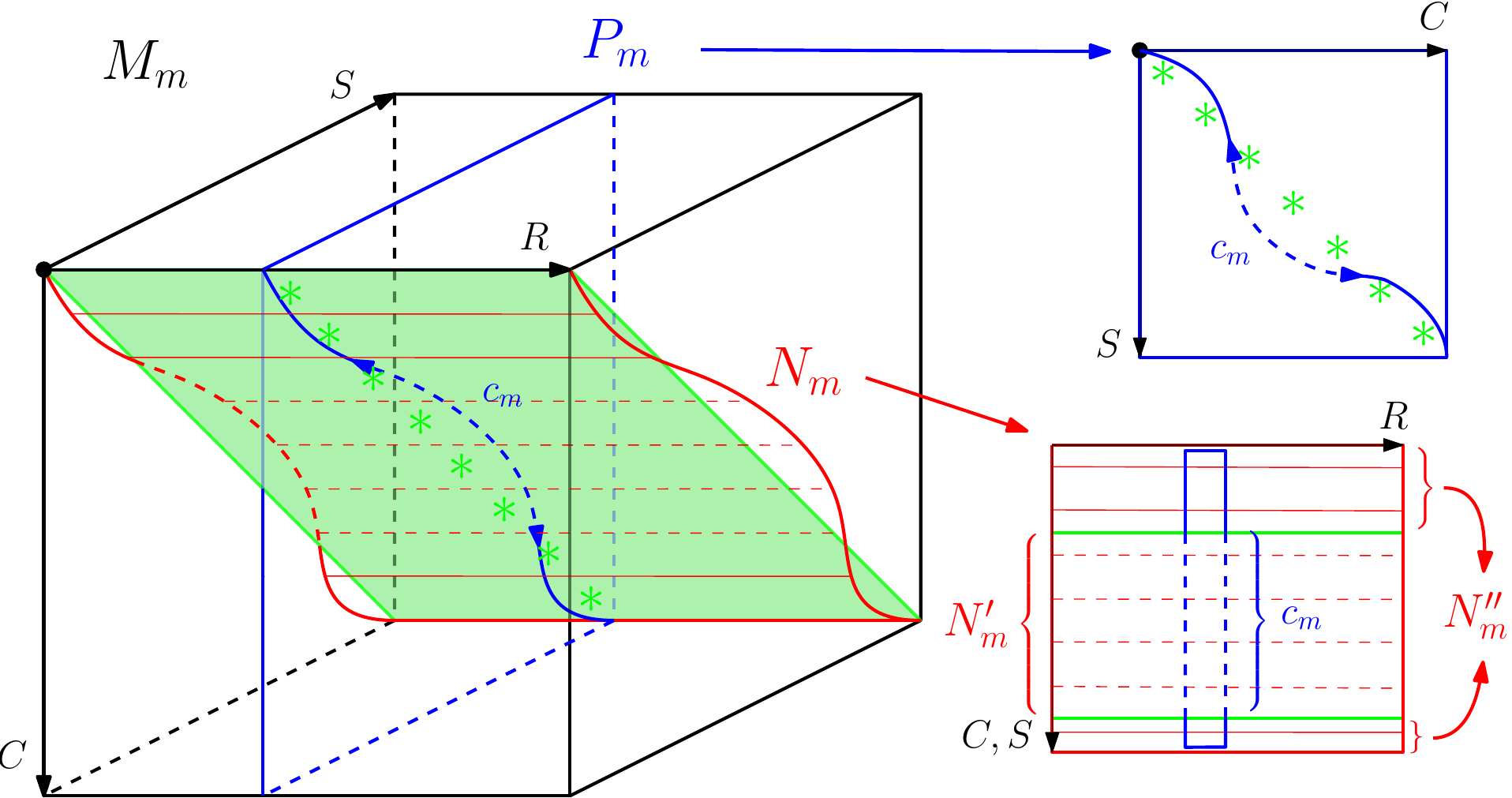}
    \qquad
    \caption{Proof of Proposition~\ref{P2}, case (c). The $(2, 3)$-d-cross-matrix $N_m$ of $M_m'$ is the red wavy surface in the matrix $M_m=M_{I_m,J_m,J_m}$ ($I_m<J_m$) and 
    is disjoint to the green plane that contains only $*$s.}
    \label{fig:cross_c}
\end{figure}

The case when the {\bf sequence $\mathbf{(|R(N_m'')|)_{m \geq 1}}$ is unbounded} is almost identical to the previous case. We can use 
Lemmas~\ref{ctvrte_lemma} and \ref{pate_lemma} as before since they apply to below-diagonal d-cross-matrices $N_m''$ too. We get the same conclusions, for
every $r\in\mathbb{N}$ there is an $m$ such that $H_m$ contains an $r$-wealthy coloring of one and the same {\bf type $\mathbf{W_2}$} or {\bf type $\mathbf{W_{3,1}}$}
or {\bf type $\mathbf{W_{3,2}}$} or {\bf type $\mathbf{W_{3,3}}$} or {\bf type $\mathbf{W_{4,2}}$}. 

The last case of case (c) to consider is when conclusion (ii) of Lemma~\ref{3DKlazar} holds and for every $m$ the matrix {\bf $\mathbf{N_m}$ is a 
$\mathbf{(2, 3)}$-ad-cross-matrix} of the submatrix $M_m'$ of the matrix $M_m=M_{I_m,J_m,J_m}$ which is a crossing matrix of the coloring $H_m$; 
we know that the sequence $(\al(M_m))_{m\ge 1}$ is bounded and the sequence $(|R(N_m)|)_{m\ge 1}$ is unbounded. Now the $i$-th row of $N_m$ is 
$$
N_m(i,\,j)=M_m(j,\,c_{m,\, i},\,s_{m,\, p(m)-i+1})\;.
$$
Like before we partition $N_m$ by its rows in two submatrices and define $N_m'$ (resp. $N_m''$) as consisting of the rows $i$ of $N_m$ such that 
$$
c_{m,\,i}<s_{m,\,p(m)-i+1}\ (\text{resp. }c_{m,\,i}>s_{m,\,p(m)-i+1})\,.
$$
Equality here again does not occur because of $R$-fullness of $N_m$. From these inequalities it follows, since (see above) both the $c_{m,i}$ and the $s_{m,i}$ increase 
in $i$, that for the rows $i$ in $N_m'$ all $c_{m,i}$s precede all $s_{m,p(m)-i+1}$s, and the other way around for $N_m''$. Thus both matrices $N_m'$ and $N_m''$ are non-intertwined. 
Clearly, one of the sequences $(|R(N_m'|)_{m\ge1}$ and $(|R(N_m''|)_{m\ge1}$, say the first one, is unbounded (the other case is identical). By part 2 of Lemma~\ref{subm_layer_crossm} we 
know that each matrix $N_m'$ is also a cross-matrix of a submatrix of the matrix $M_m'$. We apply to the matrices $N_m'$ Lemma~\ref{lema_dvaapul} and conclude
that for every $r\in\mathbb{N}$ there is an $m$ such that the coloring $H_m$ contains an $r$-wealthy coloring of one and the same {\bf type $\mathbf{W_2}$} or 
{\bf type $\mathbf{W_{3,1}}$} or {\bf type $\mathbf{W_{3,2}}$}.
\end{proof}

\begin{proof}(Proof of Theorem \ref{T1FD}.) Let $X\subset\mathcal{C}_3$ be an ideal of colorings. If $X$ is $p$-tame for some $p\in\mathbb{N}$ then 
by Proposition~\ref{M-TAME} we have that 
$$
|X_n|\leq n^{ 10p^6}\; \text{ for every $n\in\mathbb{N}$}\,.
$$
Else by Propositions~\ref{P1} and \ref{P2} there is an $i \in [4]$ such that for every $r\in\mathbb{N}$ the ideal $X$ contains 
an $r$-wealthy coloring of type $W_i$. For $i=1$ Lemma~\ref{LW1} gives $|X_n|\ge 2^{n-2}$ for every $n\in\mathbb{N}$ ($W_1$ colorings). 
For $i=2$ Proposition~\ref{LW2} gives $|X_n|\ge F_n\,(\approx 1.618^n)$ for every $n\in\mathbb{N}$  ($W_2$ colorings). For $i=3$ Proposition~\ref{LW312} gives 
$$
|X_n|>\frac{0.28\cdot1.587^n}{\sqrt{n}}\ge G_n\,(\approx 1.466^n)\;\text{ for every $n\ge23$ ($W_{3,1}$ and $W_{3,2}$ colorings)} 
$$
and Lemma~\ref{LW33} gives $|X_n|\ge F_n$ for every $n\in\mathbb{N}$ ($W_{3,3}$ colorings). For $i=4$ Proposition~\ref{LW41} gives $|X_n|\ge G_n$ for 
every $n\in\mathbb{N}$ ($W_{4,1}$ colorings) and Proposition~\ref{LW42} gives 
$$
|X_n|\ge \binom{\lfloor \frac{2(n-4)}{5} \rfloor }{\lfloor \frac{n-4}{5} \rfloor}^2\,(\approx 1.751^n)\,\ge G_n\;\text{ for every $n \geq 20$ ($W_{4,2}$ 
colorings)}\,.
$$
In the second displayed bound, the range $n\ge23$ applies to the second inequality, the first one holds for every $n\ge 1$. In the third displayed bound, 
the range $n\ge20$ applies again to the second inequality, the first one holds already for $n\ge 9$. So in all cases (see also part 3 of Lemma~\ref{LFibG}) 
we have that $|X_n|\ge G_n$ for every $n \geq 23$.
\end{proof}
\section{Concluding remarks}

Recall that in the following discussion we have $l=2$ colors. If we look at the last proof, we see that the smallest lower bound is for type $W_{4,1}$ colorings, and the next smallest one is for type $W_{3,1}$ and $W_{3,2}$ colorings in Proposition~\ref{LW312}. Therefore we have the following result.

\begin{corollary}
Let $X\subset\mathcal{C}_3$ be an ideal of colorings that for infinitely many $r \in \mathbb{N}$ does not contain an $r$-wealthy coloring of type $W_{4, 1}$. Then 
either $|X_n|$ has at most polynomial growth or $|X_n| > 2^{2(n-2)/3}/\sqrt{2n}$ for any large enough $n \in \mathbb{N}$.
\end{corollary}

\noindent
In the remark after the proof of Lemma~\ref{LW1} we showed that the lower bound $|X_n|\ge2^{n-2}$ for type $W_1$ colorings is tight. Similarly, we noted 
after the proof of Proposition~\ref{LW41} that its lower bound $|X_n|\ge G_n$ for type $W_{4,1}$ colorings is tight too.
We do not know if the lower bounds for the other colorings of types $W_2$, $W_3$ and $W_{4,2}$ are tight. We suspect that some of them are not but could 
not improve them. If one could improve the lower bound in Proposition~\ref{LW312} for  type $W_{3,1}$ and $W_{3,2}$ colorings to $|X_n| \ge F_n$, the next 
strengthening of the corollary would follow.

\begin{hypothesis}
Let $X\subset\mathcal{C}_3$ be an ideal of colorings that for infinitely many $r\in \mathbb{N}$ does not contain an $r$-wealthy coloring of type $W_{4, 1}$. 
Then either $|X_n|$ grows at most polynomially or $|X_n|\geq F_n$ for every large enough $n\in\mathbb{N}$.
\end{hypothesis}

For any fixed $k\ge1$ we define the sequence $(G_n^{k})_{n \geq 1} = (G_1^{k},G_2^{k},\dots)$ by the recurrence 
$$
    G_1^k = G_2^k = \dots = G_{k-1}^k = 1,\, G_k^k = 2 \; \mbox{ and } \; G_n^k = G_{n-1}^k + G_{n-k}^k\;\text{ for $n > k$}\,.
$$
Thus $F_{n-1} = G_n^2$ and $G_n = G_n^3$. The bounds in Theorems~\ref{T0} and \ref{T1FD} suggest the following conjecture.

\begin{hypothesis}
Let $X \subset \mathcal{C}_k$ be an ideal of colorings. Then there is a constant $c > 0$ such that either $|X_n|\le n^c$ for every 
$n \in\mathbb{N}$ or $|X_n|\geq G_n^k$ for every large enough $n$.
\end{hypothesis}

\noindent
The lower bound $G_n^k$ is tight since, as for $k=3$, the ideal $S(k)\subset\mathcal{C}_k$ of colorings $(n,\chi)$, where $\chi\colon\binom{[n]}{k}\to\{0,1\}$, such that 
for some disjoint $k$-intervals $I_1<I_2<\dots<I_r$ in $[n]$ one has $\chi(I_j)=0$ for every $j$ but $\chi(E)=1$ for all other edges, satisfies $|S(k)_n| = G_n^k$.
We considered the case $k=4$ and found all potential wealthy colorings leading to violation of one of the tameness conditions. Hopefully this approach can be generalized 
for any $k > 2$. However, the analog of type $W_{4, 2}$ colorings produces in the general version sufficiently many colorings only for $k\leq 34$. For larger $k$ one has 
to find a better way to generate sufficiently many different colorings, but this should not be a problem since the bound for $W_{4, 2}$ colorings given in 
Proposition~\ref{LW42} is larger than $G_n$.

Growth functions of ideals of ordered $3$-uniform hypergraphs should be investigated in more detail. We only determined the jump from constant to linear growth and 
the jump from polynomial to exponential growth. Finer polynomial jumps and exponential jumps for ordered graphs are described in \cite{BBM_Nes_sb}. To obtain similar
results for ordered $k$-uniform hypergraph for general $k>2$ may be the goal of a future work.

\section{Acknowledgements}
The first author was supported by the grant SVV–2020–260578.

\bibliographystyle{plain}


\end{document}